\documentclass[11pt,oneside,english,reqno]{amsart}

\usepackage[T1]{fontenc}
\usepackage[latin9]{inputenc}
\usepackage[letterpaper]{geometry}
\usepackage{mathrsfs}
\usepackage{amstext}
\usepackage{amsthm}
\usepackage{amssymb}
\usepackage{setspace}
\usepackage{esint}
\makeatletter
\numberwithin{equation}{section}
\numberwithin{figure}{section}
\theoremstyle{plain}
\newtheorem{thm}{\protect\theoremname}[section]
  \theoremstyle{remark}
  \newtheorem{rem}[thm]{\protect\remarkname}
  \theoremstyle{plain}
  \newtheorem{lem}[thm]{\protect\lemmaname}
  \theoremstyle{plain}
  \newtheorem{prop}[thm]{\protect\propositionname}
  \theoremstyle{plain}
  \newtheorem{cor}[thm]{\protect\corollaryname}
  \theoremstyle{definition}
  \newtheorem{defn}[thm]{\protect\definitionname}

\usepackage{dsfont}

\makeatother

\usepackage{babel}
  \providecommand{\corollaryname}{Corollary}
  \providecommand{\definitionname}{Definition}
  \providecommand{\lemmaname}{Lemma}
  \providecommand{\propositionname}{Proposition}
  \providecommand{\remarkname}{Remark}
\providecommand{\theoremname}{Theorem}

\begin{document}

\title{Large Deviation Principle for Interacting Brownian Motions}

\author{Insuk Seo}

\address{Courant Institute of Mathematical Sciences\\
New York University}

\email{insuk@cims.nyu.edu }


\begin{abstract}
We prove the Large Deviation Principle for the empirical process in
a system of interacting Brownian motions with singular interactions
in the nonequilibrium dynamic. Such a phenomenon has been proven only
for two lattice systems: the symmetric simple exclusion process and
zero-range process. Therefore, we have achieved the third result in
this context and moreover the first result for the diffusion-type
interacting particle system. 
\end{abstract}

\maketitle
\tableofcontents

\section{Introduction and Outline}

\subsection{Introduction}

The large scale behavior of tagged particles is a primary concern
in interacting particle systems. The first breakthrough was accomplished
by Kipnis and Varadhan \cite{KV}, whose seminal paper introduced
a general invariance principle for additive functionals of reversible
Markov processes. Furthermore, these authors derived the equilibrium
central limit theorem (CLT) for the tagged particle in the symmetric
simple exclusion process (SSEP) as an application of the general theory.
This equilibrium result has been extended to various models, e.g.,
interacting diffusions \cite{O}, the mean-zero asymmetric simple
exclusion process (ASEP) by Varadhan \cite{V4} and the ASEP for $d\ge3$
by Sethuraman, Varadhan and Yau \cite{SVY}. 

The study of tagged particles in the nonequilibrium is a field with
many untapped possibilities. Recently, several researchers have developed
nonequilibrium CLTs for tagged particles especially for 1D interacting
particle systems. Jara and Landim \cite{JL} proved such a result
for a 1D nearest neighbor SSEP where the equilibrium CLT had been
established by Rost and Vares \cite{RV}. Sethuraman and Varadhan
\cite{SV2} investigated the corresponding large deviation principle
(LDP). The nonequilibrium CLT for the 1D systems has also been proven
for the nearest neighbor ASEP by Goncalves \cite{G3}, simple exclusion
process with long jumps by Jara \cite{J}, zero-range process (ZRP)
by Jara, Landim and Sethuraman \cite{JLS} and locally interacting
Brownian motions by Grigorescu \cite{G1}. 

Another approach is to study the \textit{empirical process} $R_{N}=\frac{1}{N}\sum_{i=1}^{L}\delta_{x_{i}(\cdot)}$
which can be regarded as the averaged tagged particle. In the 1990s,
Quastel, Rezakhanlou and Varadhan found a systematic approach to study
the limit theory and large deviation theory for the empirical process,
which resulted in a series of published work \cite{Q,QRV,R,V2}. Despite
the robustness of their methodology, the LDP for this context is only
known for two models: the SSEP in the case where $d\ge2$ and the
ZRP. The limited applicability of their general method arises from
the lack of intermediate large deviation theory for the \textit{empirical
density of colors} and such a theory was only available for the SSEP
\cite{QRV} and ZRP \cite{GJL}. We refer to the survey paper by Varadhan
\cite{V3} for a comprehensive discussion of this research. 

The main purpose of the current work is to explore the third result
for this context. Our interacting particle system is the locally interacting
Brownian motions on the one-dimensional torus $\mathbb{T}$ that were
introduced by Grigorescu \cite{G1,G2}. The main result of \cite{G1}
is the nonequilibrium CLT and asymptotic independence of two tagged
particles for an interacting Brownian system and together these fulfill
the law of large numbers (LLN) for the empirical process. However,
the approach that was followed in \cite{G1} does not rely on the
empirical density of colors, but on the hydrodynamical analysis of
the local time and hence the LDP of the kind proposed in \cite{QRV}
was unavailable. In this study, we have established the LDP by analyzing
the empirical density of colors.

\subsection{\label{sub:Interacting-diffusion-with}Interacting Diffusion with
Local Interaction }

We start by introducing the interacting Brownian motions on $\mathbb{T}$
with local interactions.\footnote{For the detailed definition of model, we refer \cite{G1,G2}.}
In this model, we assume that $N$ Brownian particles\footnote{We also able to assume that the number of particle is $a_{N}$ such
that $\lim_{N\rightarrow\infty}\frac{a_{N}}{N}=\bar{\rho}$ for some
$\bar{\rho}>0$. } $x_{1}^{N}(\cdot),\,\cdots,\,x_{N}^{N}(\cdot)$ are moving on $\mathbb{T}$
with partial reflections. To state this succinctly, two particles
reflect each other when they collide, but sometimes they change their
labels. We can measure the collision time between two particles $x_{i}^{N}(\cdot)$
and $x_{j}^{N}(\cdot)$ up to time $t$ by the local time and then
the switching of labels occurs as a Poisson process with the constant
rate $\lambda N$$(\lambda>0)$ along this canonical local time. We
now provide a rigorous definition of this informally described interacting
particle system and introduce the notion of empirical density of colors.

\subsubsection{Definition of Particle System}

Our particle system $x^{N}(t)=(x_{1}^{N}(t),\,x_{2}^{N}(t),\,\cdots,\,x_{N}^{N}(t))$
can be regarded as a diffusion process on $\mathbb{T}^{N}$. We can
define this diffusion in three equivalent ways: \textit{generator},\textit{
martingale problem formulation,} and \textit{Dirichlet form}. All
of them are of course useful during our excursion. 

We first introduce an $N$-manifold 
\[
G_{N}=\{x=(x_{1},\,x_{2},\,\cdots,\,x_{N})\in\mathbb{T}^{N}\,:\,x_{i}\neq x_{j}\,\,\text{for all }i\neq j\}
\]
and consider the boundary $\partial G_{N}=\cup_{i<j}\{x_{i}=x_{j}\}$.
Note that each face $\{x:x_{i}=x_{j}\}$ consists of two sides. We
will follow the convention that $F_{ij}$ is the side at which $x_{j}$
approaches $x_{i}$ from the clockwise direction so that with the
usual orientation of $\mathbb{T}$, $x_{j}=x_{i}+0$ on $F_{ij}$
whereas $x_{j}=x_{i}-0$ on $F_{ji}$. We only consider piecewise
smooth functions on $G_{N}$ and that are smooth up to the boundary
such that we can define
\begin{align*}
f_{ij}(x) & =f(\cdots,\,x_{i-1},\,x_{i}+0,\,x_{i+1},\cdots,\,x_{j-1},\,x_{j}-0,\,x_{j+1},\,\cdots)\\
D_{ij}f(x) & =(\nabla_{i}-\nabla_{j})f(\cdots,\,x_{i-1}\,x_{i}+0,\,x_{i+1},\cdots,\,x_{j-1},\,x_{j}-0,\,x_{j+1},\,\cdots)
\end{align*}
for $x\in F_{ij}$ for all $i,\,j$. Let us denote this class of functions
by $\bar{C}(G_{N})$. We are now in a position to define the process
$x^{N}(t)$. 
\begin{enumerate}
\item \textit{Generator}: The generator for the process is $\mathscr{L}_{N}f=\frac{1}{2}\Delta f$
and the domain $\mathcal{D}(\mathscr{L}_{N})$ consists of functions
$f\in\bar{C}(G_{N})$ that satisfies the boundary condition $\mathfrak{U}_{ij}^{\lambda}f(x)=0$
on $F_{ij}$ for each $i\neq j$ where
\begin{equation}
\mathfrak{U}_{ij}^{\lambda}f(x)=D_{ij}f(x)-\lambda N(f_{ij}(x)-f_{ji}(x)).\label{eq:Uij}
\end{equation}
This infinitesimal generator $(\mathscr{L}_{N},\,\mathcal{D}(\mathscr{L}_{N}))$
defines the process.
\item \textit{Martingale problem formulation}: We define the process on
$\mathbb{T}^{N}$ by the measure $\mathbb{P}_{N}$ on $C([0,\,T],\,\mathbb{T}^{N})$
for some fixed final time $T$. Under this measure, we have $N(N-1)$
local times $\bigl\{ A_{ij}^{N}(t)\bigr\}_{1\le i\neq j\le N}$ and
the filtration $\{\mathscr{F}_{t}:t\in[0,\,T]\}$, such that for any
$f\in\bar{C}\left(G_{N}\right)$,
\begin{align}
M_{f}(t)= & f(x^{N}(t))-f(x^{N}(0))\nonumber \\
 & -\left[\frac{1}{2}\int_{0}^{t}\Delta f(x^{N}(s))ds+\sum_{i\neq j}\int_{0}^{t}\mathfrak{U}_{ij}^{\lambda}f(x^{N}(s))dA_{ij}^{N}(s)\right]\label{eq:tanaka}
\end{align}
is a martingale with respect to $\{\mathscr{F}_{t}:0\le t\le T\}$.
The martingale $M_{f}(t)$ can be represented in another way such
that 
\begin{equation}
M_{f}(t)=\sum_{k=1}^{N}\int_{0}^{t}\nabla_{k}f(x^{N}(s))d\beta_{k}(s)+\sum_{i\neq j}\int_{0}^{t}(f_{ji}-f_{ij})(x^{N}(s))dM_{ij}^{N}(s)\label{eq:martingale}
\end{equation}
where $\beta_{k}(t),\,1\le k\le N$ is a family of independent Brownian
motions and 
\[
M_{ij}^{N}(t)=J_{ij}^{N}(t)-\lambda NA_{ij}^{N}(t)
\]
 where $J_{ij}^{N},\,1\le i\neq j\le N$ is a family of pairwise orthogonal
Poisson jump processes with rates $\lambda NA_{ij}^{N}(t)$; hence,
$M_{ij}^{N}(t)$ as well as $M_{ij}^{N}(t)^{2}-\lambda NA_{ij}^{N}(t)$
are martingales for each $i\neq j$. 
\item \textit{Dirichlet Form}: For each $f\in\bar{C}(G_{N})$, the Dirichlet
form is given by
\begin{equation}
\mathbb{D}_{N}(f)=\frac{1}{2}\int_{G_{N}}\left|\nabla f(x)\right|{}^{2}dx+\frac{\lambda N}{2}\sum_{i\neq j}\int_{F_{ij}}(f_{ij}(x)-f_{ji}(x))^{2}dS_{ij}(x)\label{eq:dir}
\end{equation}
where $dS_{ij}(x)$ is the Lebesgue measure on $F_{ij}$ normalized
to have the total measure to be $1$.\footnote{Note that the usual normalization for the Lebesgue measure on this
diagonal face is $\sqrt{2}$} It should be noted that the first part of (\ref{eq:dir}) corresponds
to the Brownian movement of particles, whereas the second part takes
into account the Poisson jump type of interaction. 
\end{enumerate}
Now, we have the process $x^{N}(t)$ on $\mathbb{T}^{N}$ for $t\in[0,\,T]$
which can be regarded as a system of diffusion processes with local
interactions. 
\begin{rem}
Since $\mathbb{R}^{N}$ is a covering space of $\mathbb{T}^{N}$,
we can lift any continuous trajectory in $\mathbb{T}^{N}$ to the
one in $\mathbb{R}^{N}$ in a unique fashion. Therefore, we implicitly
regard $x^{N}(t)$ as a process on $\mathbb{R}^{N}$, sometimes. This
does not cause any technical issues, because we usually work with
the density field of the form $\frac{1}{N}\sum_{i\in I}f(x_{i}^{N}(t))$
with a periodic function $f$ on $\mathbb{R}$. For the detailed explanation,
see Section 1.4 of \cite{G1}. 
\end{rem}
To understand the large scale behavior of our interacting particle
system, we start by considering the \textit{empirical density} given
by
\begin{equation}
\mu^{N}(t)=\frac{1}{N}\sum_{i=1}^{N}\delta_{x_{i}^{N}(t)}\label{emp_den}
\end{equation}
for $0\le t\le T$ which induces a measure $\mathbb{Q}_{N}$ on $C([0,\,T],\,\mathscr{M}_{1}(\mathbb{T}))$.
For the level of the empirical density, our process is equivalent
to the non-interacting case, because our interactions essentially
involve the switching of labels which does not affect (\ref{emp_den}).
The limit theory and large deviation theory of the empirical density
in the non-interacting system is well known. To state such results
in a concrete form, we need the following assumption. \\
\\
\textbf{Assumption 1.} \textit{The initial empirical density $\left\{ \mu^{N}(0)\right\} _{N=1}^{\infty}$
satisfies the LLN in the sense that $\mu^{N}(0)\rightharpoonup\rho^{0}(dx)$
weakly in $\mathscr{M}(\mathbb{T})$ for some non-negative measure
$\rho^{0}(dx)$ on $\mathbb{T}$ with a total mass of $1$. Moreover,
$\left\{ \mu^{N}(0)\right\} _{N=1}^{\infty}$ also satisfies the LDP
with the rate function $I_{init}(\cdot)$ and scale $N$. }\\

Then the LLN and related LDP of the empirical density can be formulated
as follows. 
\begin{thm}
\label{Heat} Under Assumption 1, $\left\{ \mathbb{Q}_{N}\right\} _{N=1}^{\infty}$
converges weakly to the Dirac mass on the trajectory $\{\rho(t,\,x)dx\,:\,0\le t\le T\}$
where $\rho(t,\,x)$ is the solution of the heat equation $\partial_{t}\rho=\frac{1}{2}\Delta\rho$
with initial condition $\rho^{0}(dx)$. Moreover $\left\{ \mathbb{Q}_{N}\right\} _{N=1}^{\infty}$
satisfies the LDP with the rate function 
\begin{equation}
I_{init}(\gamma(0,\,\cdot))+\frac{1}{2}\int_{0}^{T}\left\Vert \partial_{t}\gamma-\frac{1}{2}\Delta\gamma\right\Vert _{-1,\gamma}^{2}dt\label{eq: uncolaex}
\end{equation}
and scale $N$ where the $H^{-1}$ norm is defined in the standard
way.\end{thm}
\begin{proof}
See \cite{G1,KO} for the LLN and LDP, respectively. \end{proof}
\begin{rem}
The LLN and LDP for the interacting system of diffusive particles
have been developed at the level of the empirical density for several
models. The first results were published in the classic papers \cite{G4,DG}
for the LLN and LDP for weakly interacting Brownian motions on $\mathbb{R}$,
respectively. In this model, the interaction comes into play through
the drift coefficient of the form $b(x_{i}(t),\,\mu^{N}(t))$. Recently,
\cite{Shk,DSVZ} developed corresponding theories for case in which
particles interact through their ranks and hence the diffusion coefficients
also depend on $\mu^{N}(t)$. In both models, the limiting dynamics
are governed by McKean-Vlasov type of equations. Another type of result
was presented in \cite{V1} for the two-body model
\begin{equation}
dx_{i}^{N}(t)=-N\sum_{j:j\neq i}V'(N(x_{i}^{N}(t)-x_{j}^{N}(t))dt+d\beta_{i}(t)\label{eq: sde}
\end{equation}
on $\mathbb{T}$, where $V(\cdot)$ is a compactly supported, even
and smooth potential. In this model, the limiting particle density
is given by the unique solution of $\partial_{t}\rho(t,\,x)=\frac{1}{2}[P(\rho(t,\,x))]_{xx}$
where $P(\cdot)$ is the pressure functional depending on $V$. Even
if the author did not explicitly establish the LDP, estimates therein
are sufficient to establish the LDP through standard methodology of
\cite{DV}. It was already pointed out in \cite{G1} that our local
interaction model can be regarded as a limit of the two-body interaction
model. More precisely, if we consider the sequence of the potential
$\{V_{\epsilon}\}_{\epsilon>0}$ satisfying 
\[
\lim_{\epsilon\rightarrow0}\int_{\mathbb{T}}\exp\{2V_{\epsilon}(x)-1\}dx=\frac{1}{\lambda}\delta_{0}
\]
in the sense of distribution, then the corresponding diffusion $\mathbb{P}_{N}^{\epsilon}$
given by (\ref{eq: sde}) with $V_{\epsilon}(\cdot)$ converges to
our locally interacting diffusion $\mathbb{P}_{N}$ with parameter
$\lambda$. We refer the readers to \cite{G1} for details. We also
remark here that the totally asymmetric counterpart of our model has
been studied in \cite{FSW1,FSW2}.
\end{rem}

\subsubsection{Empirical Density of Colors\label{sub:Colored-Process}}

We now introduce the notion of the empirical density of colors, which
is an intermediate object toward the empirical process. 

Let $\{I_{1}^{N},\,I_{2}^{N},\cdots,\,I_{m}^{N}\}$ be a (non-random)
partition of $\{1,\,2,\,\cdots,\,N\}$ such that $I_{c}^{N}$ satisfies
\begin{equation}
\lim_{N\rightarrow\infty}\frac{\left|I_{c}^{N}\right|}{N}=\bar{\rho}_{c}>0\label{eq:aver col c}
\end{equation}
for each $1\le c\le m$ where $\bar{\rho}_{c}$ is the average density
of color $c$. Then, $\{x_{i}^{N}(\cdot):i\in I_{c}^{N}\}$ denotes
the set of particles of color $c$ and empirical density of color
$c$ is defined as 
\begin{equation}
\mu_{c}^{N}(t)=\frac{1}{N}\sum_{i\in I_{c}^{N}}\delta_{x_{i}^{N}(t)}\label{eq:emp color c}
\end{equation}
for $t\in[0,\,T]$. Finally, the empirical density of colors is defined
by 
\begin{equation}
\tilde{\mu}^{N}(t)=(\mu_{1}^{N}(t),\,\mu_{2}^{N}(t),\,\cdots,\,\mu_{m}^{N}(t))^{\dagger}\label{eq:emp color}
\end{equation}
which induces a probability measure $\widetilde{\mathbb{Q}}_{N}$
on $C([0,\,T],\,\mathscr{M}(\mathbb{T})^{m})$. In contrast to the
uncolored empirical density (\ref{emp_den}), we have to take the
interaction into account, because the switching of labels between
particles of different colors affects $\tilde{\mu}^{N}(t)$ in a complex
manner. Similar to the uncolored empirical density, we need an assumption
on $\tilde{\mu}^{N}(0)$ to obtain the limit theory and large deviation
theory. \\
\\
\textbf{Assumption 2.} \textit{The initial empirical density of colors
satisfies the LLN in the sense that 
\begin{equation}
\tilde{\mu}^{N}(0)\rightharpoonup\tilde{\rho}^{0}(dx)=\left(\rho_{1}^{0}(dx),\,\rho_{2}^{0}(dx),\,\cdots,\,\rho_{m}^{0}(dx)\right)^{\dagger}\label{eq: ini conv color}
\end{equation}
weakly in $\ensuremath{\ensuremath{\mathscr{M}}(\mathbb{T})^{m}}$
where the non-negative measure $\rho_{c}^{0}(dx)$ has a total mass
of $\bar{\rho}_{c}$ for each $c$. Moreover, $\left\{ \tilde{\mu}^{N}(0)\right\} _{N=1}^{\infty}$
also satisfies the LDP with the rate function $I_{init}^{m}(\cdot)$
and scale $N$. }\\

Note that Assumption 2 implies Assumption 1 with $\rho^{0}=\sum_{c=1}^{m}\rho_{c}^{0}$.

\subsection{\label{sub:Large-Deviation-Results}Main Results}

Based on the precise description of the model in the previous section,
we now summarize our main results.

\subsubsection{Large Deviation Theory for Empirical Density of Colors}

The hydrodynamic limit theory for $\{\widetilde{\mathbb{Q}}_{N}\}_{N=1}^{\infty}$
is implied by the results of \cite{G1}. It is well known \cite{R,Spo}
that if the scaling limit of the tagged particle is the diffusion
with the generator $\mathscr{L_{\rho}}$, which may depend on the
limiting particle density $\rho$, and any two tagged particles are
asymptotically independent, then the limiting particle density $\rho_{c}$
of each color is the unique weak solution of the PDE $\partial_{t}\rho_{c}=\mathscr{L}_{\rho}^{*}\rho_{c}$
with the initial condition $\rho_{c}^{0}(dx)$ for each color $c$. 

The scaling limit as well as the asymptotic independence of the tagged
particle of our system has been studied \cite{G1}. The scaling limit
turned out to be the diffusion with the time-dependent generator 
\begin{equation}
\mathscr{A}_{\rho}=\frac{\lambda}{2(\lambda+\rho(t,\,x))}\Delta-\frac{(2\lambda+\rho(t,\,x))\nabla\rho(t,\,x)}{2(\lambda+\rho(t,\,x))^{2}}\nabla\label{eq: th generator}
\end{equation}
where $\rho(t,\,x)$ is the solution of the heat equation as in Theorem
\ref{Heat}. Consequently, under Assumption 2, $\rho_{c}$ is the
solution of the parabolic equation
\begin{equation}
\frac{\partial\rho_{c}}{\partial t}=\mathscr{A}_{\rho}^{*}\rho_{c}=\frac{1}{2}\nabla\left[\frac{\lambda}{\lambda+\rho}\nabla\rho_{c}+\frac{\nabla\rho}{\lambda+\rho}\rho_{c}\right]\label{eq:single color PDE}
\end{equation}
with the initial condition $\rho_{c}^{0}(dx)$ for each color $c$.
We are also able to reorganize these equations in the form of a matrix
with the notation $\tilde{\rho}=(\rho_{1},\,\rho_{2},\,\cdots,\,\rho_{m})^{\dagger}$
as 
\begin{equation}
\frac{\partial\tilde{\rho}}{\partial t}=\frac{1}{2}\nabla\cdot\left[D(\tilde{\rho})\nabla\tilde{\rho}\right]\label{eq: lln color}
\end{equation}
with initial condition $\tilde{\rho}^{0}(dx)$. The $m\times m$ diffusion
matrix $D(\tilde{\rho})$ is explicitly given by
\begin{equation}
D(\tilde{\rho})_{ij}=\frac{\delta_{ij}\lambda+\rho_{i}}{\lambda+\rho}.\label{eq:diffusion mat}
\end{equation}
In other words, $\widetilde{\mathbb{Q}}_{N}$ converges weakly to
a Dirac mass on the unique solution of (\ref{eq: lln color}). We
explain the details with another proof of this result in Section 4.

We can decompose the diffusion matrix by $D(\tilde{\rho})=A(\tilde{\rho})\chi(\tilde{\rho})$
where $A(\tilde{\rho})$ and $\chi(\tilde{\rho})$ are defined by
\begin{equation}
A(\tilde{\rho})_{ij}=\frac{\delta_{ij}\lambda\rho_{j}+\rho_{i}\rho_{j}}{\lambda+\rho},\,\,\chi(\tilde{\rho})=\text{diag}\left(\frac{1}{\rho_{1}},\,\frac{1}{\rho_{2}},\,\cdots,\,\frac{1}{\rho_{m}}\right).\label{eq:A and chi}
\end{equation}
Here, $\chi(\tilde{\rho})$ is the Hessian of the entropy functional
$h(\tilde{\rho})=\sum_{i=1}^{m}\rho_{i}\log\rho_{i}$ and $A(\tilde{\rho})=D(\tilde{\rho})\chi(\tilde{\rho})^{-1}$.
Note that this matrix $A(\tilde{\rho})$ is symmetric, which is not
a coincidence; rather, it is a so-called Onsager reciprocity(see \cite{GJL}
for details).

Under Assumption 2 and an additional technical assumption related
to the uncolored initial profile (Assumption 3 in Section 4.1), the
LDP for $\{\widetilde{\mathbb{Q}}_{N}\}_{N=1}^{\infty}$ can be established
with the rate function $I_{color}^{m}(\pi_{\cdot})$ for $\pi_{\cdot}\in C([0,\,T],\,\mathscr{M}(\mathbb{T}^{m}))$.
More precisely, $I_{color}^{m}(\pi_{\cdot})<\infty$ only if $\pi_{t}$
is absolutely continuous with respect to the Lebesgue measure for
each $t$, and if we write such $\pi_{\cdot}$ as $\tilde{\rho}(\cdot,\,x)dx$
then 
\[
I_{color}^{m}(\tilde{\rho}(\cdot,\,x)dx)=I_{init}^{m}(\tilde{\rho}(0,\,x)dx)+I_{dyn}^{m}(\tilde{\rho}(\cdot,\,x)dx)
\]
where $I_{init}^{m}(\cdot)$ and $I_{dyn}^{m}(\cdot)$ explain the
large deviation rates of the initial configuration and the dynamic
evolution of the system, respectively. The initial rate function $I_{init}^{m}(\cdot)$
is just a part of Assumption 2. The dynamical rate function $I_{dyn}^{m}(\cdot)$
is our primary concern and is given by
\begin{equation}
I_{dyn}^{m}(\tilde{\rho}(\cdot,\,x)dx)=\frac{1}{2}\int_{0}^{T}\left\Vert \frac{\partial\tilde{\rho}}{\partial t}-\frac{1}{2}\nabla\cdot\left[D(\tilde{\rho})\nabla\tilde{\rho}\right]\right\Vert _{-1,\,A(\tilde{\rho})}^{2}dt.\label{eq:ravs}
\end{equation}
The rigorous meaning of this expression is carefully explained in
Section 4.2. This large deviation result is the main contribution
of the current work and is explained throughout Sections 2, 3 and
4. In Section 2, we establish some super-exponential estimates which
essentially mollify the local times into the local densities. Section
3 provides the exponential tightness of $\bigl\{\widetilde{\mathbb{Q}}_{N}\bigr\}{}_{N=1}^{\infty}$,
which compactify the upper bound problem. Relying on these preliminary
results, we compute the exact upper and lower bounds in Section 4.

\subsubsection{Large Deviation Theory of Empirical Process}

The empirical process is defined by $\frac{1}{N}\sum_{i=1}^{N}\delta_{x_{i}^{N}(\cdot)}$
which induces a measure $P_{N}$ on $\mathscr{M}_{1}(C([0,\,T],\,\mathbb{T}))$.
As we mentioned earlier, $P_{N}$ converges weakly to the delta mass
on a diffusion $P\in\mathscr{M}_{1}(C([0,\,T],\,\mathbb{T}))$ of
which the generator is given by $\mathscr{A}_{\rho}$ defined in (\ref{eq: th generator}).
Our main concern is the LDP corresponding to this result and developed
in Section 5. 

We now explain the main result. For a probability measure $Q$ on
$C([0,\,T],\,\mathbb{T})$, we can explain the LDP via the rate function
$\mathscr{I}(Q)$, which is finite only if $Q$ has marginal densities
at any time $t\in[0,\,T]$. Let us denote this marginal density by
$q(t,\,x)$, then $\mathscr{I}(Q)<\infty$ only if $q$ is weakly
differentiable in $x$ and satisfies
\[
\int_{\mathbb{T}}q(0,\,x)\log q(0,\,x)dx<\infty\,\,\,\,\,\text{and \,\,\,\,\,}\int_{0}^{T}\int_{\mathbb{T}}\frac{\left|\nabla q\right|^{2}}{q}dxdt<\infty.
\]
For such $Q$, we consider a class of function on $[0,\,T]\times\mathbb{T}$
given by 
\[
\mathscr{B}_{q}=\left\{ b(t,\,x):\partial_{t}q=\frac{1}{2}\Delta q-\nabla\left[bq\right],\,\,\int_{0}^{T}\int_{\mathbb{T}}b^{2}qdxdt<\infty\right\} 
\]
where the sense of PDE is weak. Then we can find the unique diffusion
process $P^{b}$ with the generator $\mathscr{A}_{q,b}=\mathscr{A}_{q}+b\nabla$
and starting measure $q(0,\,x)$ for each $b\in\mathscr{B}_{q}$.
Note that $P^{b}$ also has the marginal density $q(t,\,x)$. Then
we can prove that the relative entropy $H(b):=H[Q|P^{b}]$ is either
finite for all $b\in\mathscr{B}_{q}$ or identically infinite. We
set $\mathscr{I}(Q)=\infty$ for the latter case. For the former case,
we can find a $b_{Q}\in\mathscr{B}_{q}$ which minimizes $H(\cdot)$
on $\mathscr{B}_{q}$, and then the rate function is given by 
\begin{equation}
\mathscr{I}(Q)=I_{init}(q(0,\,\cdot))+\frac{1}{2}\int_{0}^{T}\int_{\mathbb{T}}b_{Q}^{2}(t,\,x)q(t,\,x)dxdt+H\left[Q\left|P^{b_{Q}}\right.\right].\label{eq: rate function final}
\end{equation}
The last two terms measure large deviation rates from $P$ to $P^{b_{Q}}$
and $P^{b_{Q}}$ to $Q$, respectively. 

The final important remark is that this LDP result is a direct consequence
of that of the empirical density of colors. The LDP rate for the finite
dimensional projection of the empirical process can be understood
as the one for the empirical density of colors. By doing so, we can
obtain the full LDP by using Dawson-G\"artner's projective limit
theory. This profound relationship has been revealed by Quastel, Rezakhanlou
and Varadhan in \cite{QRV} and, therefore, our work not only verifies
the robustness of their methodology, but also the universality of
their large deviation result. This is because our rate function (\ref{eq: rate function final})
is quite similar to that of the SSEP, whereas the model dynamics are
seemingly unrelated. We will explain the robustness and universality
in Section 5.

\section{Super-exponential Estimates}

In the LDP theory in the context of interacting particle systems,
the core step is to establish the \textit{replacement lemma}, named
after Guo, Papanicolau and Varadhan's seminal work \cite{GPV}, which
is a super-exponential type of estimate. We introduce an appropriate
form of the replacement lemma and related concepts in Section \ref{sub: REP Lem}.
We prove this by first providing some preliminary estimates in Section
\ref{sub:Preliminary-Estimate}, following which the proof of the
main replacement lemma will be given in Section \ref{sub:Proof-of-Superexponential}.

\subsection{Replacement Lemma\label{sub: REP Lem}}

In general, the empirical density is studied via its corresponding
density field. In the context of the hydrodynamic limit theory as
well as the large deviation theory for the empirical density, the
main difficulty is the replacement of the current, which appeared
at the computation of density field as a result of interactions, by
gradients. Due to the nature of our dynamic, the current is highly
related to the local time as we can see from (\ref{eq:tanaka}). Consequently,
the replacement lemma should replace the local times by the appropriate
gradient type, which we achieved by first introducing the averaged
local times and local densities. 

The basic local times in our process are 
\[
A_{ij}^{N}(t)=\lim_{\epsilon\rightarrow0}\int_{0}^{t}\frac{\mathds{1}{}_{[0,\,\epsilon]}(x_{j}^{N}(t)-x_{i}^{N}(t))}{2\epsilon}=\int_{0}^{t}\delta_{+}(x_{j}^{N}(t)-x_{i}^{N}(t))dt
\]
for $i\neq j$ where $\delta_{+}$ is a delta-type distribution on
$[0,\,1]$ such that $\int_{0}^{1}f(x)\delta_{+}(x)=\frac{1}{2}f(0)$.
We can understand $A_{ij}^{N}(t)$ as the amount of the collision
time between two particles $x_{j}^{N}(\cdot)$ and $x_{i}^{N}(\cdot)$
up to the time $t$ at which the particle $x_{i}^{N}(\cdot)$ approaches
to $x_{j}^{N}(\cdot)$ from the clockwise direction. Of course, each
local time $A_{ij}^{N}(t)$ is quite noisy and impossible to estimate
alone. Fortunately, these noises can be controlled by taking average
among them. Two such examples are 
\begin{align}
A^{N}(t) & =\frac{1}{N^{2}}\sum_{i\neq j}\left[A_{ij}^{N}(t)+A_{ji}^{N}(t)\right]\label{eq:Loc2}\\
A_{i}^{N}(t) & =\frac{1}{N}\sum_{j:j\neq i}\left[A_{ij}^{N}(t)+A_{ji}^{N}(t)\right]\label{eq:Loc3}
\end{align}
for each $i$. Note that $A_{i}^{N}(t)$ is the average collision
time of the particle $x_{i}^{N}(\cdot)$ against all the other particles
up to time $t$ and $A^{N}(t)$ is the total average of local times.
The behavior of these averaged local times has been studied extensively
in \cite{G1} and is also important for our work. However, as our
focus is on the density field of the empirical density of colors $\tilde{\mu}^{N}(\cdot)$,
the main object to be estimated is 
\begin{equation}
A_{i,c}^{N}(t)=\frac{1}{N}\sum_{j\in I_{c}^{N}}\left[A_{ij}^{N}(t)+A_{ji}^{N}(t)\right]\label{eq:Loc4}
\end{equation}
which measures the average collision time of the particle $x_{i}^{N}(\cdot)$
against the particles of color $c$. 

Now, we define the notion of local densities. For $x=(x_{1},\,x_{2},\,\cdots,\,x_{N})\in\mathbb{T}^{N}$,
the local density function of color $c$ around the particle $x_{i}$
is defined by
\[
\rho_{\epsilon,i}^{(c)}(x)=\frac{1}{2N\epsilon}\sum_{j\in I_{c}^{N}}\chi_{[-\epsilon,\epsilon]}(x_{j}-x_{i}).
\]
Here, the function $\chi_{[-\epsilon,\epsilon]}(\cdot)$ is the usual
indicator function on $\mathbb{T}$ and we henceforth simply denote
this function by $\chi_{\epsilon}(\cdot)$. The essence of the replacement
lemma for our model is the replacement of the integral with respect
to the average local time of the form

\[
dA_{i,c}^{N}(t)=\frac{1}{N}\sum_{j\in I_{c}^{N}}\left(dA_{ij}^{N}(t)+dA_{ji}^{N}(t)\right)
\]
by the usual integral of the form $\rho_{\epsilon,i}^{(c)}(x^{N}(t))dt$.
Formally, this replacement can be stated as the following theorem. 
\begin{thm}[Replacement Lemma]
\label{thm: REPLACEMENT_colored} For $\epsilon,\,\delta>0$, $0\le t_{1}<t_{2}\le T$
and two colors $c_{1},\,c_{2}$, let $\mathbf{C}_{N}^{c_{1},\,c_{2}}(t_{1},\,t_{2};\epsilon,\,\delta)$
be the event defined by 
\[
\left\{ x^{N}(\cdot)\,:\,\frac{1}{N}\sum_{i\in I_{c_{1}}^{N}}\left|\int_{t_{1}}^{t_{2}}\rho_{\epsilon,i}^{(c_{2})}(x^{N}(t))dt-\left[A_{i,c_{2}}^{N}(t_{2})-A_{i,c_{2}}^{N}(t_{1})\right]\right|>\delta\right\} .
\]
Then, we have 
\begin{equation}
\limsup_{\epsilon\rightarrow0}\limsup_{N\rightarrow\infty}\frac{1}{N}\log\mathbb{P}^{N}\left[\mathbf{C}_{N}^{c_{1},\,c_{2}}(t_{1},\,t_{2};\epsilon,\,\delta)\right]=-\infty.\label{eq:super exponential replacement lemma_colored}
\end{equation}

\end{thm}
The main object of the remaining part of current section is to prove
this theorem.

\subsection{Preliminary Estimates\label{sub:Preliminary-Estimate}}

\subsubsection{Green's Formula for $G_{N}$}

We frequently use Green's formula on the $N$-manifold $G_{N}$. As
a calculation of this nature is not conventional, we briefly explain
our philosophy in this short subsection. 

Let us denote the $i$th standard unit vector by $e_{i}$. Then, we
can apply Green's formula for $f\in\bar{C}(G_{N})$ and vector field
$\mathbf{V}(x)=\sum_{i=1}^{N}V_{i}(x)e_{i}$ with $V_{i}\in\bar{C}(G_{N})$
for each $i$, so that 
\begin{align}
 & \int_{G_{N}}\left\langle \nabla f(x),\,\mathbf{V}(x)\right\rangle dx\label{eq:green}\\
 & =-\int_{G_{N}}f(x)\left(\nabla\cdot\mathbf{V}\right)(x)dx+\int_{\partial G_{N}}f(x)\left\langle \mathbf{V}(x),\,n(x)\right\rangle dS(x)\nonumber \\
 & =-\int_{G_{N}}f(x)\left(\nabla\cdot\mathbf{V}\right)(x)dx+\sum_{i\neq j}\int_{F_{ij}}f_{ij}(x)\left\langle \mathbf{V}(x),\,e_{i}-e_{j}\right\rangle dS_{ij}(x).\nonumber 
\end{align}

\begin{rem}
Recall from (\ref{eq:dir}) that $dS_{ij}(x)$ is the Lebesgue measure
on $F_{ij}$ normalized to have a total measure of $1$. The unit
normal vector of the boundary $F_{ij}$ is $\frac{1}{\sqrt{2}}(e_{i}-e_{j})$,
but $\frac{1}{\sqrt{2}}$ is eliminated from (\ref{eq:green}) because
of this renormalization. 
\end{rem}
In particular, some special forms of the vector fields provide us
useful results. We summarize such results by the following lemma.
\begin{lem}
\label{lem: GREEN THM}For $f\in\bar{C}(G_{N})$, $h\in C^{1}([0,\,1])$
and $V_{i}(x)=\sum_{j:j\neq i}h(x_{j}-x_{i})$, the vector field
\[
\mathbf{V}(x)=\sum_{i=1}^{N}V_{i}(x)e_{i}.
\]
satisfies
\begin{align*}
\int_{G_{N}}\left\langle \nabla f(x),\,\mathbf{V}(x)\right\rangle dx= & -\int_{G_{N}}f(x)\left(\nabla\cdot\mathbf{V}\right)(x)dx\\
 & -\frac{h(1)-h(0)}{2}\sum_{i\neq j}\int_{F_{ij}}\left[f_{ij}(x)+f_{ji}(x)\right]dS_{ij}(x).
\end{align*}
In addition, suppose that $U_{1}(x),\,\cdots,\,U_{N}(x)\in C(\mathbb{T}^{N})$
satisfy $U_{i}(x)=U_{j}(x)$ whenever $x_{i}=x_{j}$, for all $i\ne j$.
Then, the vector field
\[
\mathbf{W}(x)=\sum_{i=1}^{N}U_{i}(x)V_{i}(x)e_{i}
\]
satisfies
\begin{align*}
\int_{G_{N}}\left\langle \nabla f(x),\,\mathbf{W}(x)\right\rangle dx= & -\int_{G_{N}}f(x)\left(\nabla\cdot\mathbf{W}\right)(x)dx\\
 & -\frac{h(1)-h(0)}{2}\sum_{i\neq j}\int_{F_{ij}}U_{i}(x)\left[f_{ij}(x)+f_{ji}(x)\right]dS_{ij}(x).
\end{align*}
\end{lem}
\begin{proof}
For the first part, it is enough to check boundary terms. Note that
\[
\left\langle \mathbf{V}(x),\,e_{i}-e_{j}\right\rangle =h(x_{j}-x_{i})-h(x_{i}-x_{j})+\sum_{k:k\neq i,j}\left[h(x_{k}-x_{i})-h(x_{k}-x_{j})\right]
\]
and $h(x_{k}-x_{i})-h(x_{k}-x_{j})=0$ on $F_{ij}$ for $k\neq i,\,j$.
Moreover, $h(x_{j}-x_{i})=h(0)$ and $h(x_{i}-x_{j})=h(1)$ on $F_{ij}$
and hence
\begin{align*}
 & \sum_{i=1}^{N}\int_{G_{N}}\left\langle \nabla f(x),\,\mathbf{V}(x)\right\rangle dx\\
 & =-\int_{G_{N}}f(x)\left(\nabla\cdot\mathbf{V}\right)(x)dx-(h(1)-h(0))\sum_{i\neq j}\int_{F_{ij}}f_{ij}(x)dS_{ij}(x)
\end{align*}
by (\ref{eq:green}). Obviously, we can symmetrize the last term as
\[
\sum_{i\neq j}\int_{F_{ij}}f_{ij}(x)dS_{ij}(x)=\frac{1}{2}\sum_{i\neq j}\int_{F_{ij}}\left(f_{ij}(x)+f_{ji}(x)\right)dS_{ij}(x)
\]
and we are done. The proof of the second part is similar.
\end{proof}

\subsubsection{Estimates Based on Dirichlet Form. }

The proof of Theorem \ref{thm: REPLACEMENT_colored} heavily relies
on the Dirichlet form. In particular, we frequently use $\mathbb{D}_{N}(\sqrt{f})$
for some $f\ge0$ and we denote this by $\mathcal{D}_{N}(f)$. By
(\ref{eq:dir}), 
\begin{equation}
\mathcal{D}_{N}(f)=\frac{1}{8}\int_{G_{N}}\frac{|\nabla f(x)|^{2}}{f(x)}dx+\frac{\lambda N}{2}\sum_{i\neq j}\int_{F_{ij}}\left[\sqrt{f_{ij}(x)}-\sqrt{f_{ji}(x)}\right]^{2}dS_{ij}(x).\label{eq: ROOT DIRICHLET FORM}
\end{equation}
In addition, let $\mbox{\ensuremath{\mathscr{P}}}_{N}$ be the class
of non-negative functions $f\in\bar{C}(G_{N})$ which also satisfies
$\int_{\mathbb{T}}f(x)dx=1$. 
\begin{lem}
\label{Lemma1}For any $f\in\mbox{\ensuremath{\mathscr{P}}}_{N}$,
we have 
\begin{equation}
\frac{1}{N}\sum_{i=1}^{N}\int_{G_{N}}\left|\nabla_{i}f(x)\right|dx\le\sqrt{\frac{8\mathcal{D}_{N}(f)}{N}}.\label{eq:diff}
\end{equation}
\end{lem}
\begin{proof}
This bound can be proven by 
\begin{align*}
\frac{1}{N}\int\sum_{i=1}^{N}\left|\nabla_{i}f(x)\right|dx & \le\frac{1}{\sqrt{N}}\int\left|\nabla f(x)\right|dx\le\frac{1}{\sqrt{N}}\left[\int\frac{\left|\nabla f(x)\right|^{2}}{f(x)}dx\,\int f(x)dx\right]^{\frac{1}{2}}.
\end{align*}
\end{proof}
\begin{lem}
\label{Lemma2}For any $f\in\mbox{\ensuremath{\mathscr{P}}}_{N}$,
we have 
\begin{equation}
\frac{1}{N^{2}}\sum_{i\neq j}\int_{F_{ij}}\left(f_{ij}(x)+f_{ji}(x)\right)dS_{ij}(x)\le2+\sqrt{\frac{8\mathcal{D}_{N}(f)}{N}}.\label{eq:local}
\end{equation}
\end{lem}
\begin{proof}
Define a function $h(x)=x-\frac{1}{2}$ on $[0,\,1]$, then $V_{i}(x)=\sum_{j:j\neq i}h(x_{j}-x_{i})$
satisfies $\nabla_{i}V_{i}(x)=-(N-1)$ and $|V_{i}(x)|\le(N-1)/2$.
Hence, by Lemma \ref{lem: GREEN THM}, 
\begin{align*}
 & \frac{1}{2}\sum_{i\neq j}\int_{F_{ij}}\left(f_{ij}(x)+f_{ji}(x)\right)dS_{ij}(x)\\
 & =N(N-1)\int_{\mathbb{T}^{N}}f(x)dx+\sum_{i=1}^{N}\int_{\mathbb{T}^{N}}\left[\nabla_{i}f(x)\right]V_{i}(x)dx\\
 & \le N(N-1)+\frac{N-1}{2}\sum_{i=1}^{N}\int_{\mathbb{T}^{N}}\left|\nabla_{i}f(x)\right|dx.
\end{align*}
We can complete the proof by Lemma \ref{Lemma1}.
\end{proof}
Let us define $M_{\epsilon,i}(x)=\sum_{j:j\neq i}\chi_{\epsilon}(x_{j}-x_{i})$
which counts the number of particles around $x_{i}$. The following
series of lemmas provides estimates related to $M_{\epsilon,i}(\cdot)$.
We also remark here that we shall write $C$ for a constant and as
usual different occurrences of $C$ may denote different constants. 
\begin{lem}
\label{Lemma3}For any $f\in\mbox{\ensuremath{\mathscr{P}}}_{N}$
and $\epsilon>0$, we have 
\begin{equation}
\frac{1}{N^{2}}\int_{G_{N}}f(x)\sum_{i=1}^{N}M_{\epsilon,i}(x)dx\le C\left(1+\sqrt{\frac{\mathcal{D}_{N}(f)}{N}}\right)\epsilon.\label{eq:total}
\end{equation}
\end{lem}
\begin{proof}
We take an auxiliary function $h_{\epsilon}$ on $[0,\,1]$ such that
\[
h_{\epsilon}(x)=\begin{cases}
x-\epsilon & \text{for\,\,}0\le x\le\epsilon\\
0 & \text{for\,\,}\epsilon\le x\le1-\epsilon\\
x-(1-\epsilon) & \text{for\,\,}1-\epsilon\le x\le1
\end{cases}
\]
Note that $V_{i}(x)=\sum_{j:j\neq i}h_{\epsilon}(x_{j}-x_{i})$ satisfies
$|V_{i}(x)|\le(N-1)\epsilon$ and $\nabla_{i}V_{i}(x)=-M_{\epsilon,i}(x)$.
Thus, by Lemma \ref{lem: GREEN THM}, 
\begin{align*}
 & \int_{G_{N}}f(x)\sum_{i=1}^{N}M_{\epsilon,i}(x)dx\\
 & =\sum_{i=1}^{N}\int_{G_{N}}\left[\nabla_{i}f(x)\right]V_{i}(x)dx+\epsilon\sum_{i\neq j}\int_{F_{ij}}\left(f_{ij}(x)+f_{ji}(x)\right)dS_{ij}(x)\\
 & \le(N-1)\epsilon\sum_{i=1}^{N}\int_{G_{N}}\left|\nabla_{i}f(x)\right|dx+\epsilon\sum_{i\neq j}\int_{F_{ij}}\left(f_{ij}(x)+f_{ji}(x)\right)dS_{ij}(x)
\end{align*}
and therefore we can complete the proof by applying Lemma \ref{Lemma1}
and \ref{Lemma2}. \end{proof}
\begin{lem}
\label{Lemma4}For any $f\in\mbox{\ensuremath{\mathscr{P}}}_{N}$
and $\epsilon>0$, we have
\begin{equation}
\frac{1}{N^{2}}\int_{G_{N}}\sum_{i=1}^{N}\left|\nabla_{i}f(x)\right|M_{\epsilon,i}(x)dx\le C\left[1+\left(\frac{\mathcal{D}_{N}(f)}{N}\right)^{\frac{3}{4}}\right]\epsilon^{\frac{1}{2}}.\label{eq:total-1}
\end{equation}
\end{lem}
\begin{proof}
By Cauchy-Schwarz's inequality,
\begin{align*}
\frac{1}{N^{2}}\int_{G_{N}}\sum_{i=1}^{N}\left|\nabla_{i}f(x)\right|M_{\epsilon,i}(x)dx & \le\frac{1}{N^{2}}\left(\int_{G_{N}}\frac{\left|\nabla f(x)\right|^{2}}{f(x)}dx\int_{G_{N}}f(x)\sum_{i=1}^{N}M_{\epsilon,i}^{2}(x)dx\right)^{\frac{1}{2}}\\
 & \le\left(\frac{8\mathcal{D}_{N}(f)}{N}\cdot\frac{1}{N^{2}}\int_{G_{N}}f(x)\sum_{i=1}^{N}M_{\epsilon,i}(x)dx\right)^{\frac{1}{2}}
\end{align*}
since $M_{\epsilon,i}(x)\le N$. Thus, (\ref{eq:total-1}) is direct
from Lemma \ref{Lemma3}. \end{proof}
\begin{lem}
\label{Lemma5}For any $f\in\mbox{\ensuremath{\mathscr{P}}}_{N}$
and $0<\epsilon<\frac{1}{4}$, we have 
\begin{equation}
\frac{1}{N^{3}}\int_{G_{N}}f(x)\sum_{i=1}^{N}M_{\epsilon,i}^{2}(x)dx\le C\left[1+\left(\frac{\mathcal{D}_{N}(f)}{N}\right)^{\frac{3}{4}}\right]\left(\epsilon^{\frac{3}{2}}+\frac{\epsilon}{N}\right).\label{eq:sec}
\end{equation}
\end{lem}
\begin{proof}
Since $M_{\epsilon,i}^{2}(x)=\sum_{k,\,j\neq i}\chi_{\epsilon}(x_{k}-x_{i})\chi_{\epsilon}(x_{j}-x_{i})$,
we need to bound $\chi_{\epsilon}(x)\chi_{\epsilon}(y)$ by a more
tractable object. To this end, we define a function $k_{\epsilon}(\cdot,\,\cdot)\in C(\mathbb{T}^{2})$
for $\epsilon<\frac{1}{4}$. Firstly, along the line $y-x=c$ with
$0\le c\le2\epsilon$, $k_{\epsilon}(x,\,y)$ is defined by 
\[
k_{\epsilon}(x,\,y)=\begin{cases}
\frac{1-2\epsilon}{1-c}(x+\epsilon) & \,\text{if\,\,}-\epsilon\le x\le\epsilon-c\\
\frac{2\epsilon-c}{1-c}(1-\epsilon-c-x) & \text{\,\ if\,\,}\epsilon-c\le x\le1-\epsilon-c\\
0 & \,\text{if\,\,}1-\epsilon-c\le x\le1-\epsilon.
\end{cases}
\]
For $y-x=c$, $-2\epsilon\le c\le0$, we set $k_{\epsilon}(x,\,y):=k_{\epsilon}(y,\,x)$.
Finally, $k_{\epsilon}(x,\,y)=0$ for all the other $x,\,y$. It is
easy to see that $k_{\epsilon}\in C(\mathbb{T}^{2})$ and moreover
$k_{\epsilon}$ satisfies 
\begin{eqnarray}
 & 0\le k_{\epsilon}(x,\,y)\le2\epsilon\chi_{2\epsilon}(x-y)\label{eq: pro2}\\
 & \chi_{\epsilon}(x)\chi_{\epsilon}(y)\le2\epsilon\chi_{2\epsilon}(x-y)+\nabla_{x}k_{\epsilon}(x,\,y)+\nabla_{y}k_{\epsilon}(x,\,y).\label{eq:pro3}
\end{eqnarray}
In particular, (\ref{eq:pro3}) enables us to bound $M_{\epsilon,i}^{2}(x)$
such a way that 
\begin{align*}
M_{\epsilon,i}^{2}(x) & \le M_{\epsilon,i}(x)+\sum_{p,\,q:p,\,q\neq i,\,p\neq q}\left[2\epsilon\chi_{2\epsilon}(x_{p}-x_{q})-\nabla_{i}k_{\epsilon}\left(x_{p}-x_{i},\,x_{q}-x_{i}\right)\right]\\
 & \le M_{\epsilon,i}(x)+\epsilon\sum_{l=1}^{N}M_{2\epsilon,l}(x)-\nabla_{i}K_{\epsilon,i}(x)
\end{align*}
where 
\[
K_{\epsilon,i}(x)=\sum_{p,\,q:p,\,q\neq i,\,p\neq q}k_{\epsilon}(x_{p}-x_{i},\,x_{q}-x_{i}).
\]
Therefore, we have
\begin{align}
\sum_{i=1}^{N}\int_{G_{N}}f(x)M_{\epsilon,i}^{2}(x)dx\le & \sum_{i=1}^{N}\int_{G_{N}}f(x)M_{\epsilon,i}(x)dx+\epsilon N\int_{G_{N}}f(x)\sum_{l=1}^{N}M_{2\epsilon,l}(x)dx\nonumber \\
 & -\sum_{i=1}^{N}\int_{G_{N}}f(x)\nabla_{i}K_{\epsilon,i}(x)dx.\label{eq:stm}
\end{align}
We can bound the first two terms of the RHS by Lemma \ref{Lemma3}.
For the last term, we can apply Green's formula with the vector field
$\mathbf{K}(x)=\sum_{i=1}^{N}K_{\epsilon,i}(x)e_{i}$ so that
\begin{equation}
-\sum_{i=1}^{N}\int_{G_{N}}f(x)\nabla_{i}K_{i,\epsilon}(x)dx=\sum_{i=1}^{N}\int_{G_{N}}\nabla_{i}f(x)K_{i,\epsilon}(x)dx.\label{eq:mogl}
\end{equation}
Note that boundary terms are disappeared since $k_{\epsilon}$ is
continuous on $\mathbb{T}^{2}$. Moreover, by (\ref{eq: pro2}),
\begin{equation}
K_{i,\epsilon}(x)\le2\epsilon\sum_{\substack{p,\,q\neq i\\
p\neq q
}
}\chi_{2\epsilon}(x_{p}-x_{q})\le\epsilon\sum_{l=1}^{N}M_{2\epsilon,l}(x).\label{eq: molg}
\end{equation}
Consequently, we can bound the last term of (\ref{eq:stm}) by (\ref{eq:mogl})
and (\ref{eq: molg}), such that 
\begin{align*}
\left|\sum_{i=1}^{N}\int_{G_{N}}\nabla_{i}f(x)K_{i,\epsilon}(x)dx\right| & \le\left(\int_{G_{N}}\frac{\left|\nabla f(x)\right|^{2}}{f(x)}dx\int_{G_{N}}f(x)\sum_{i=1}^{N}K_{\epsilon,i}^{2}(x)dx\right)^{\frac{1}{2}}\\
 & \le\left(16\epsilon N^{2}\mathcal{D}_{N}(f)\int_{G_{N}}f(x)\sum_{i=1}^{N}K_{\epsilon,i}(x)dx\right)^{\frac{1}{2}}\\
 & \le\left(16\epsilon N^{2}\mathcal{D}_{N}(f)\int_{G_{N}}f(x)\left[N\epsilon\sum_{l=1}^{N}M_{2\epsilon,l}(x)\right]dx\right)^{\frac{1}{2}}
\end{align*}
where we used the trivial bound $K_{\epsilon,i}(x)\le2\epsilon N^{2}$
at the second inequality. Now, we can complete the proof by applying
Lemma \ref{Lemma3}. \end{proof}
\begin{lem}
\label{Lemma6}For any $f\in\mbox{\ensuremath{\mathscr{P}}}_{N}$
and $0<\epsilon<\frac{1}{4}$, we have \textup{
\[
\frac{1}{N^{3}}\sum_{i\neq j}\int_{F_{ij}}M_{\epsilon,i}(x)(f_{ij}+f_{ji})(x)dS_{ij}(x)\le C\left[1+\left(\frac{\mathcal{D}_{N}(f)}{N}\right)^{\frac{3}{4}}\right]\left(\epsilon^{\frac{1}{2}}+\frac{1}{N}\right).
\]
}\end{lem}
\begin{proof}
Let us define two auxiliary functions $v_{\epsilon},\,u_{\epsilon}$
on $[0,\,1]$ by $v_{\epsilon}(x)=\int_{\frac{1}{2}}^{x}\chi_{\epsilon}(y)dy$
and $u_{\epsilon}(x)=\int_{\frac{1}{2}}^{x}v_{\epsilon}(y)dy$. They
enjoy the following properties:
\begin{eqnarray}
u_{\epsilon}'(x)=v_{\epsilon}(x) & \text{\ and}\, & v_{\epsilon}'(x)=\chi_{\epsilon}(x)\label{eq:pprop1}\\
\frac{\epsilon^{2}}{8}\chi_{\frac{\epsilon}{2}}(x)\le u_{\epsilon}(x)\le\frac{\epsilon^{2}}{2}\chi_{\epsilon}(x) & \text{\ and}\, & \left|v_{\epsilon}(x)\right|\le\frac{\epsilon}{2}\chi_{\epsilon}(x)\label{eq:pprop2}\\
u_{\epsilon}(0)=u_{\epsilon}(1)=\frac{\epsilon^{2}}{2} & \text{\ and}\, & v_{\epsilon}(0)=-\frac{\epsilon}{2}\ensuremath{,}v_{\epsilon}(1)=\frac{\epsilon}{2}.\label{eq:pprop3}
\end{eqnarray}

\noindent Let us denote $V_{\epsilon,i}(x)=\sum_{k:k\neq i}v_{\epsilon}(x_{k}-x_{i})$
and $U_{\epsilon,i}(x)=\sum_{k:k\neq i}u_{\epsilon}(x_{k}-x_{i})$.
Then the vector field $\mathbf{W}(x)=\sum_{i=1}^{N}U_{\epsilon,i}(x)V_{\epsilon,i}(x)e_{i}$
satisfies conditions of the second part of Lemma \ref{lem: GREEN THM}
and therefore,
\begin{equation}
P_{1}=P_{2}+P_{3}\label{eq: pppp1}
\end{equation}
where
\begin{align*}
P_{1} & =\frac{v_{\epsilon}(1)-v_{\epsilon}(0)}{2}\sum_{i\neq j}\int_{F_{ij}}U_{\epsilon,i}(x)\left(f_{ij}(x)+f_{ji}(x)\right)dS_{ij}(x)\\
P_{2} & =-\sum_{i=1}^{N}\int_{G_{N}}f(x)\nabla_{i}\left[U_{\epsilon,i}(x)V_{\epsilon,i}(x)\right]dx\\
P_{3} & =-\sum_{i=1}^{N}\int_{G_{N}}U_{\epsilon,i}(x)V_{\epsilon,i}(x)\nabla_{i}f(x)dx.
\end{align*}

We first estimate $P_{1}$. Notice that $v_{\epsilon}(1)-v_{\epsilon}(0)=\epsilon$
by (\ref{eq:pprop3}) and $U_{\epsilon,i}(x)\ge\frac{\epsilon^{2}}{8}M_{\frac{\epsilon}{2},i}(x)$
by (\ref{eq:pprop2}). These together give us
\begin{equation}
P_{1}\ge\frac{\epsilon^{3}}{16}\sum_{i\neq j}\int_{F_{ij}}M_{\frac{\epsilon}{2},i}(x)\left(f_{ij}(x)+f_{ji}(x)\right)dS_{ij}(x).\label{eq: p1p1}
\end{equation}

Now, let us consider $P_{2}$. By (\ref{eq:pprop1}), 
\[
-\nabla_{i}\left[U_{\epsilon,i}(x)V_{\epsilon,i}(x)\right]=U_{\epsilon,i}(x)M_{\epsilon,i}(x)+V_{\epsilon,i}^{2}(x),
\]
which is bounded above by $\frac{3}{4}\epsilon^{2}M_{\epsilon,i}^{2}(x)$
due to (\ref{eq:pprop2}). Thus, we can bound $P_{2}$ as
\begin{equation}
P_{2}\le CN^{3}\left[1+\left(\frac{\mathcal{D}_{N}(f)}{N}\right)^{\frac{3}{4}}\right]\left(\epsilon^{\frac{7}{2}}+\frac{\epsilon^{3}}{N}\right)\label{eq:p2p2}
\end{equation}
by Lemma \ref{Lemma5}. 

We now bound $P_{3}.$ Since $U_{\epsilon,i}(x)V_{\epsilon,i}(x)\le\frac{\epsilon^{3}}{4}M_{\epsilon,i}^{2}(x)$
by (\ref{eq:pprop2}), 
\begin{equation}
\left|P_{3}\right|\le\frac{\epsilon^{3}}{4}\sum_{i=1}^{N}\int_{G_{N}}\left|\nabla_{i}f(x)\right|M_{\epsilon,i}^{2}(x)dx\le CN^{3}\left[1+\left(\frac{\mathcal{D}_{N}(f)}{N}\right)^{\frac{3}{4}}\right]\epsilon^{\frac{7}{2}}\label{eq: p3p3}
\end{equation}
by Lemma \ref{Lemma4}. Now, (\ref{eq: pppp1}), (\ref{eq: p1p1}),
(\ref{eq:p2p2}) and (\ref{eq: p3p3}) implies the desired bound.
\end{proof}
The next and last preliminary estimate controls the discontinuity
of $f$ along the boundary, joint with $M_{i,\epsilon}(x)$. 
\begin{lem}
\label{Lemma7}For any $f\in\mbox{\ensuremath{\mathscr{P}}}_{N}$
and $0<\epsilon<\frac{1}{4}$, we have \textup{
\[
\frac{1}{N^{2}}\sum_{i\neq j}\int_{F_{ij}}|f_{ij}(x)-f_{ji}(x)|M_{i,\epsilon}(x)dS_{ij}(x)\le C\left[1+\left(\frac{\mathcal{D}_{N}(f)}{N}\right)^{\frac{7}{8}}\right]\left(\epsilon^{\frac{1}{4}}+\frac{1}{\sqrt{N}}\right).
\]
}\end{lem}
\begin{proof}
By Cauchy-Schwarz's inequality, 
\begin{align*}
 & \frac{1}{N^{2}}\sum_{i\neq j}\int_{F_{ij}}|f_{ij}(x)-f_{ji}(x)|M_{i,\epsilon}(x)dS_{ij}(x)\\
 & \le\frac{1}{N^{2}}\left[\sum_{i\neq j}\int_{F_{ij}}\left(\sqrt{f_{ij}(x)}-\sqrt{f_{ji}(x)}\right)^{2}dS_{ij}(x)\right]^{\frac{1}{2}}\\
 & \,\,\,\,\,\,\,\,\,\,\,\,\,\,\,\,\,\,\,\times\left[\sum_{i\neq j}\int_{F_{ij}}\left(\sqrt{f_{ij}(x})+\sqrt{f_{ji}(x)}\right)^{2}M_{i,\epsilon}^{2}(x)dS_{ij}(x)\right]^{\frac{1}{2}}\\
 & \le\frac{1}{N^{2}}\left(\frac{4}{\lambda}\cdot\frac{\mathcal{D}_{N}(f)}{N}\left[\sum_{i\neq j}\int_{F_{ij}}\left(f_{ij}(x)+f_{ji}(x)\right)M_{i,\epsilon}^{2}(x)dS_{ij}(x)\right]\right)^{\frac{1}{2}}.
\end{align*}
and the proof is completed by Lemma \ref{Lemma6}. 
\end{proof}

\subsubsection{Main Estimate\label{sub:Main-estimates}}
\begin{prop}
\label{Main Replacement Lemma} For any $f\in\mbox{\ensuremath{\mathscr{P}}}_{N}$
and $0<\epsilon<\frac{1}{4}$, we have 
\begin{align}
\frac{1}{N^{2}}\sum_{i\neq j}\left|\int_{G_{N}}f(x)\frac{\chi_{\epsilon}(x_{j}-x_{i})}{2\epsilon}dx-\frac{1}{2}\left(\int_{F_{ij}}f_{ij}(x)dS_{ij}(x)+f_{ji}(x)dS_{ji}(x)\right)\right|\nonumber \\
\le C\left[1+\left(\frac{\mathcal{D}_{N}(f)}{N}\right)^{\frac{7}{8}}\right]\left(\epsilon^{\frac{1}{4}}+\frac{1}{\sqrt{N}}\right).\label{eq:ttyas}
\end{align}
\end{prop}
\begin{proof}
Note first that the presence of absolute values in the summation prevents
us from applying Green's formula in the form of Lemma \ref{lem: GREEN THM}.
Instead, we define a function $g_{\epsilon}$ on $[0,\,1]$ as
\[
g_{\epsilon}(x)=\begin{cases}
\frac{x}{2\epsilon}-\frac{1}{2} & \text{for\,\,\,}0\le x\le\epsilon\\
0 & \text{for\,\,}\,\epsilon\le x\le1-\epsilon\\
\frac{x-1}{2\epsilon}+\frac{1}{2} & \text{for\,\,}\,1-\epsilon\le x\le1.
\end{cases}
\]
and then consider a vector field $\mathbf{V}(x)=g_{\epsilon}(x_{i}-x_{j})e_{j}$
consisting of only one direction. Green's formula for this vector
field is 
\begin{equation}
\int_{G_{N}}\nabla_{j}f(x)g_{\epsilon}(x_{i}-x_{j})dx=\int_{G_{N}}f(x)\frac{1}{2\epsilon}\chi_{\epsilon}(x_{i}-x_{j})dx+J_{1}+J_{2}\label{eq:1st imsi}
\end{equation}
because $g_{\epsilon}'(x)=\frac{1}{2\epsilon}\chi_{\epsilon}(x)$
where $J_{1}$ and $J_{2}$ are the boundary terms to be explained
below. The boundary terms for this particular Green's formula are
\begin{equation}
\sum_{p\neq q}\left[\int_{F_{pq}}f_{pq}(x)\left\langle \mathbf{V}(x),\,e_{p}-e_{q}\right\rangle dS_{pq}(x)\right]\label{eq: kcc}
\end{equation}
for which the summands are non-zero only if $p$ or $q$ is $j$.
Now, let $J_{1}$ be the sum of two summands in (\ref{eq: kcc}) with
$(p,\,q)=(i,\,j)$ or $(j,\,i)$ and $J_{2}$ be the sum of all the
others. Namely, 
\begin{align*}
J_{1} & =\int_{F_{ji}}f_{ji}(x)g_{\epsilon}(x_{i}-x_{j})dS_{ji}(x)-\int_{F_{ij}}f_{ij}(x)g_{\epsilon}(x_{i}-x_{j})dS_{ij}(x)\\
J_{2} & =\sum_{k:k\neq i,\,j}\left[\int_{F_{jk}}f_{jk}(x)g_{\epsilon}(x_{i}-x_{j})dS_{jk}(x)-\int_{F_{kj}}f_{kj}(x)g_{\epsilon}(x_{i}-x_{j})dS_{kj}(x)\right]
\end{align*}
Note that $g_{\epsilon}(x_{i}-x_{j})=\frac{1}{2}$ on $F_{ij}$ and
$-\frac{1}{2}$ on $F_{ji}$ and hence 
\begin{align}
J_{1} & =-\frac{1}{2}\int_{F_{ji}}f_{ji}(x)dS_{ji}(x)-\frac{1}{2}\int_{F_{ij}}f_{ij}(x)dS_{ij}(x).\nonumber \\
 & =-\frac{1}{2}\left(\int_{F_{ij}}f_{ij}(x)dS_{ij}(x)+f_{ji}(x)dS_{ji}(x)\right).\label{eq:J111}
\end{align}
For $J_{2}$, we know that $g_{\epsilon}(x_{i}-x_{j})$ has same value
on $F_{jk}$ and $F_{kj}$ and therefore 
\begin{equation}
J_{2}=\sum_{k:k\neq i,\,j}\left[\int_{F_{jk}}(f_{jk}(x)-f_{kj}(x))g_{\epsilon}(x_{i}-x_{j})dS_{jk}(x)\right].\label{eq:J222}
\end{equation}

By combining (\ref{eq:J111}), (\ref{eq:J222}) with (\ref{eq:1st imsi})
we can bound the LHS of (\ref{eq:ttyas}) by $J_{3}+J_{4}$ where
\begin{align*}
J_{3} & =\frac{1}{N^{2}}\sum_{i\neq j}\int_{G_{N}}\left|\nabla_{j}f(x)\right|\left|g_{\epsilon}(x_{i}-x_{j})\right|dx\\
J_{4} & =\frac{1}{N^{2}}\sum_{i\neq j}\sum_{k:k\neq i,\,j}\int_{F_{jk}}\left|f_{jk}(x)-f_{kj}(x)\right|\left|g_{\epsilon}(x_{i}-x_{j})\right|dS_{jk}(x)
\end{align*}
Since $\left|g_{\epsilon}(\cdot)\right|\le\frac{1}{2}\chi_{\epsilon}(\cdot)$,
we have 
\begin{align}
J_{3} & \le\frac{1}{2N^{2}}\sum_{i=1}^{N}\int_{G_{N}}\left|\nabla_{i}f(x)\right|M_{\epsilon,i}(x)dx\label{eq:ttya1}\\
J_{4} & \le\frac{1}{2N^{2}}\sum_{i\neq j}\int_{F_{jk}}\left|f_{ij}(x)-f_{ij}(x)\right|M_{\epsilon,i}(x)dS_{ij}(x).
\end{align}
The proof is completed by Lemmas \ref{Lemma4} and \ref{Lemma7}. 
\end{proof}

\subsection{Proof of Replacement Lemma\label{sub:Proof-of-Superexponential}}

In this subsection, we provide the proof of Theorem \ref{thm: REPLACEMENT_colored}
based on Proposition \ref{Main Replacement Lemma} and the classic
technique developed by Donsker and Varadhan \cite{DV}. Their method
is only available for the process $\mathbb{P}_{N}$ which is sufficiently
close to the equilibrium process $\mathbb{P}_{N}^{eq}$ in the sense
that $\left\Vert \log\frac{d\mathbb{P}_{N}}{d\mathbb{P}_{N}^{eq}}\right\Vert _{L^{\infty}(\mathbb{T})}=O(N)$.
Unfortunately, this condition does not hold not only for our model,
but also for the general interacting particle system of diffusion
type. For example, if we start deterministically, our process $\mathbb{P}_{N}$
is even \textit{orthogonal} to $\mathbb{P}_{N}^{eq}$ starting from
the invariant measure $dx$. We solve this issue by using a symmetrization
procedure. However, this procedure is only possible for the time slot
$[\eta,\,T]$ for some $\eta>0$. Thus, we have to establish the replacement
lemma on the interval $[0,\,\eta]$ in an independent manner. Let
us examine this procedure more closely, by dividing the Theorem \ref{thm: REPLACEMENT_colored}
into the following two propositions. 
\begin{prop}
\label{prop: normal time regime}For any $\eta,\,\delta>0$, $\eta\le t_{1}<t_{2}\le T$
and two colors $c_{1}\neq c_{2}$,
\begin{equation}
\limsup_{\epsilon\rightarrow0}\limsup_{N\rightarrow\infty}\frac{1}{N}\log\mathbb{P}_{N}\left[\mathbf{C}_{N}^{c_{1},\,c_{2}}(t_{1},\,t_{2};\epsilon,\,\delta)\right]=-\infty.\label{eq: ntrg}
\end{equation}

\end{prop}

\begin{prop}
\label{prop: smal time regime}For any $\delta>0$ and two colors
$c_{1}\neq c_{2}$,
\begin{align}
 & \limsup_{\eta\rightarrow0}\limsup_{\epsilon\rightarrow0}\limsup_{N\rightarrow\infty}\label{eq:strg}\\
 & \frac{1}{N}\log\mathbb{P}_{N}\left[\frac{1}{N}\sum_{i\in I_{c_{1}}^{N}}\left|\int_{0}^{\eta}\rho_{\epsilon,i}^{(c_{2})}(x^{N}(t))dt-A_{i,c_{2}}^{N}(\eta)\right|>\delta\right]=-\infty.\nonumber 
\end{align}

\end{prop}
In this paper, estimates in the form of Propositions \ref{prop: normal time regime}
and \ref{prop: smal time regime} are referred to as by \textit{normal
time regime} and \textit{small time regime,} respectively. These respective
regimes require different approaches. Dichotomies of this nature frequently
occur in our work.

\subsubsection{Normal Time Regime}

The key procedure on which to base the proof of Proposition \ref{prop: normal time regime}
is symmetrization. Our interacting particle system is defined as a
probability measure $\mathbb{P}_{N}$ on $C([0,\,T],\,\mathbb{T}^{N})$
with the initial profile $f_{N}^{0}(dx)$; then, the density profile
at time $t>0$ denoted by $f_{N}(t,\,x)$ satisfies the forward equation
\begin{equation}
\frac{\partial f_{N}}{\partial t}(t,\,x)=\frac{1}{2}\Delta f_{N}(t,\,x)+\sum_{i\neq j}\mathfrak{U}_{ij}^{\lambda}f_{N}(t,\,x)\delta^{+}(x_{j}-x_{i})\label{eq:forward}
\end{equation}
where $\mathfrak{U}_{ij}^{\lambda}f$ is as defined in (\ref{eq:Uij}).
The process with the initial density $dx$ is the equilibrium process,
which we denote by $\mathbb{P}_{N}^{eq}$. 

Now, we define some intermediate processes. Let $\mathfrak{P}_{N}$
be the set of all permutations of $[N]$ and let $\sigma(x)=(x_{\sigma(1)},\,x_{\sigma(2)},\,\cdots,\,x_{\sigma(N)})$
for $\sigma\in\mathfrak{P}_{N}$ and $x\in\mathbb{T}^{N}$. Then we
can consider a process starting from
\begin{equation}
\bar{f}_{N}^{0}(dx)=\frac{1}{N!}\sum_{\sigma\in\mathfrak{P}_{N}}f_{N}^{0}(d\sigma(x))\label{eq:init_unc}
\end{equation}
with the same interacting mechanism. We denote this process by $\bar{\mathbb{P}}_{N}$.
Finally, we define another initial profile 
\begin{equation}
\bar{f}_{N}^{0,\,color}(dx)=\frac{1}{|\mathfrak{C}_{N}|}\sum_{\sigma\in\mathfrak{C}_{N}}f_{N}^{0}(\sigma(x))\label{eq:init_col}
\end{equation}
where $\mathfrak{C}_{N}\subset\mathfrak{P}_{N}$ is the set of all
permutations with $I_{1}^{N},\,I_{2}^{N},\,\cdots,\,I_{m}^{N}$ (cf.
Section 1.1.2) as their invariant sets. Then, let $\mathbb{\bar{P}}_{N}^{color}$
be the process with the initial profile $\bar{f}_{N}^{0,\,color}(dx)$
with the same type of interactions. 
\begin{lem}
\label{lem: symmetrization lemma}Let $\mathbf{E}_{N}$ be an event
on $C([0,\,T],\,\mathbb{T}^{N})$ which only depends on sub-path $\{x(s):\eta\le s\le T\}$.
Furthermore, if the event $\mathbf{E}_{N}$ is invariant under permutations
in the sense that $\left\{ x(\cdot)\in\mathbf{E}_{N}\right\} =\left\{ \sigma(x(\cdot))\in\mathbf{E}_{N}\right\} $
for all $\sigma\in\mathfrak{P}_{N}$, then we have
\begin{equation}
\mathbb{P}_{N}\left[\mathbf{E}_{N}\right]\le\left(\frac{C}{\sqrt{\eta}}\right)^{N}\mathbb{P}_{N}^{eq}\left[\mathbf{E}_{N}\right]\label{eq: symm_uncolr}
\end{equation}
for some universal constant $C$. Furthermore, if the event $\mathbf{E}_{N}$
is only invariant under the permutations among the same color in the
sense that $\left\{ x(\cdot)\in\mathbf{E}_{N}\right\} =\left\{ \sigma(x(\cdot))\in\mathbf{E}_{N}\right\} $
for all $\sigma\in\mathfrak{C}_{N}$, then 
\begin{equation}
\mathbb{P}_{N}\left[\mathbf{E}_{N}\right]\le\left(\frac{Cm}{\sqrt{\eta}}\right)^{N}\mathbb{P}_{N}^{eq}\left[\mathbf{E}_{N}\right]\label{eq:symm_color}
\end{equation}
where $m$ is the number of colors. \end{lem}
\begin{proof}
First of all, the marginal density of the process $\bar{\mathbb{P}}_{N}$
is
\[
\bar{f}_{N}(t,\,x)=\frac{1}{N!}\sum_{\sigma\in\mathfrak{P}_{N}}f_{N}(t,\,\sigma(x))
\]
and hence, we can deduce from (\ref{eq:forward}) that $\bar{f}_{N}(t,\,x)$
is the solution of the heat equation $\partial_{t}\bar{f}_{N}=\frac{1}{2}\Delta\bar{f}_{N}$
with initial condition (\ref{eq:init_unc}). Therefore, we have bound
of the form 
\begin{equation}
\left\Vert \bar{f}_{N}(t,\,\cdot)\right\Vert _{\infty}\le\left(\frac{C}{\sqrt{t}}\right)^{N}\le\left(\frac{C}{\sqrt{\eta}}\right)^{N}\label{eq:univ bound}
\end{equation}
for a $t\ge\eta$ with a (universal) constant $C$.\footnote{It is easy to check since $\bar{f}_{N}(t,\,x)=\bar{f}_{N}^{0}*p_{N}^{t}(x)$
where $p_{N}^{t}(x)$, the heat kernel on $\mathbb{T}$, is given
by $\left(\frac{1}{\sqrt{2\pi t}}\right)^{N}\sum_{n\in\mathbb{Z}}\exp\left\{ \frac{(x+n)^{2}}{2t}\right\} $. } Note that if $\mathbf{E}_{N}$ is invariant under all permutations
then $\mathbb{P}_{N}\left[\mathbf{E}_{N}\right]=\bar{\mathbb{P}}_{N}\left[\mathbf{E}_{N}\right]$.
Moreover, since $\mathbf{E}_{N}$ only depends on the path after time
$\eta$, we have 
\[
\bar{\mathbb{P}}_{N}\left[\mathbf{E}_{N}\right]\le\left(\frac{C}{\sqrt{\eta}}\right)^{N}\mathbb{P}_{N}^{eq}\left[\mathbf{E}_{N}\right]
\]
by (\ref{eq:univ bound}) and therefore we can derive (\ref{eq: symm_uncolr}). 

For (\ref{eq:symm_color}), note first that the marginal density profile
of $\mathbb{\bar{P}}_{N}^{color}$ at time $t$ is
\[
\bar{f}_{N}^{color}(t,\,x)=\frac{1}{|\mathfrak{C}_{N}|}\sum_{\sigma\in\mathfrak{C}_{N}}f_{N}(t,\,\sigma(x)).
\]
Since $|\mathfrak{C}_{N}|=N_{1}!N_{2}!\cdots N_{m}!$ where $N_{c}=\left|I_{c}^{N}\right|$,
we can obtain
\[
\frac{\bar{f}_{N}^{color}(t,\,x)}{\bar{f}_{N}(t,\,x)}=\frac{\frac{1}{|\mathfrak{C}_{N}|}\sum_{\sigma\in\mathfrak{C}_{N}}f_{N}(t,\,\sigma(x))}{\frac{1}{N!}\sum_{\sigma\in\mathfrak{P}_{N}}f_{N}(t,\,\sigma(x))}\le\frac{N!}{N_{1}!N_{2}!\cdots N_{m}!}\le m^{N}
\]
and thus $\left\Vert \bar{f}_{N}^{color}(t,\,\cdot)\right\Vert _{\infty}\le\left(Cm/\sqrt{\eta}\right)^{N}$
for $t\ge\eta$ from (\ref{eq:univ bound}). Therefore, we can derive
(\ref{eq:symm_color}) in a similar way. 
\end{proof}
If $\eta\le t_{1}<t_{2}\le T$, then the event $\mathbf{C}_{N}^{c_{1},\,c_{2}}(t_{1},\,t_{2};\epsilon,\,\delta)$
satisfies the conditions of second part of previous lemma and thus
\begin{equation}
\mathbb{P}{}_{N}\left[\mathbf{C}_{N}^{c_{1},\,c_{2}}(t_{1},\,t_{2};\epsilon,\,\delta)\right]\le\left(\frac{Cm}{\sqrt{\eta}}\right)^{N}\mathbb{P}_{N}^{eq}\left[\mathbf{C}_{N}^{c_{1},\,c_{2}}(t_{1},\,t_{2};\epsilon,\,\delta)\right].\label{eq:symmetrization2-1}
\end{equation}
Consequently, we can reduce Proposition \ref{prop: normal time regime}
into the following equilibrium estimate. 
\begin{prop}
\label{prop: eq_normal time regime}For any $\eta,\,\delta>0$, $\eta\le t_{1}<t_{2}\le T$
and two colors $c_{1}\neq c_{2}$, 
\begin{equation}
\limsup_{\epsilon\rightarrow0}\limsup_{N\rightarrow\infty}\frac{1}{N}\log\mathbb{P}_{N}^{eq}\left[\mathbf{C}_{N}^{c_{1},\,c_{2}}(t_{1},\,t_{2};\epsilon,\,\delta)\right]=-\infty.\label{eq:eq esss}
\end{equation}
\end{prop}
\begin{proof}
By Chebyshev's inequality, 
\begin{align}
 & \frac{1}{N}\log\mathbb{P}_{N}^{eq}\left[\mathbf{C}_{N}^{c_{1},\,c_{2}}(t_{1},\,t_{2};\epsilon,\,\delta)\right]\label{eq:sual}\\
 & \le-a\delta+\frac{1}{N}\log\mathbb{E}_{N}^{eq}\exp\left\{ a\sum_{i\in I_{c_{1}}^{N}}\left|\int_{t_{1}}^{t_{2}}\rho_{\epsilon,i}^{(c_{2})}(x^{N}(t))dt-\left(A_{i,c_{2}}^{N}(t_{2})-A_{i,c_{2}}^{N}(t_{1})\right)\right|\right\} \nonumber 
\end{align}
for any $a>0$ where $\mathbb{E}_{N}^{eq}$ is the expectation with
respect to $\mathbb{P}_{N}^{eq}$. Note that 
\begin{align}
 & \mathbb{E}_{N}^{eq}\exp\left\{ a\sum_{i\in I_{c_{1}}^{N}}\left|\int_{t_{1}}^{t_{2}}\rho_{\epsilon,i}^{(c_{2})}(x^{N}(t))dt-\left(A_{i,c_{2}}^{N}(t_{2})-A_{i,c_{2}}^{N}(t_{1})\right)\right|\right\} \label{eq:suan}\\
 & \le\sum_{\mathfrak{e}_{i}=\pm1\,\forall i}\mathbb{E}_{N}^{eq}\exp\left\{ a\sum_{i\in I_{c_{1}}^{N}}\mathfrak{e}_{i}\left[\int_{t_{1}}^{t_{2}}\rho_{\epsilon,i}^{(c_{2})}(x^{N}(t))dt-\left(A_{i,c_{2}}^{N}(t_{2})-A_{i,c_{2}}^{N}(t_{1})\right)\right]\right\} .\nonumber 
\end{align}
and let us investigate each summand of the last line. By Feynman-Kac's
formula, 
\begin{align}
 & \mathbb{E}_{N}^{eq}\exp\left\{ a\sum_{i\in I_{c_{1}}^{N}}\mathfrak{e}_{i}\left[\int_{t_{1}}^{t_{2}}\rho_{\epsilon,i}^{(c_{2})}(x^{N}(t))dt-\left(A_{i,c_{2}}^{N}(t_{2})-A_{i,c_{2}}^{N}(t_{1})\right)\right]\right\} \label{eq:suabn}\\
 & \le\exp\{(t_{2}-t_{1})\lambda_{N,\epsilon,a}\}\nonumber 
\end{align}
where $\lambda_{N,\epsilon,a}$ is the largest eigenvalue of the operator
\[
\mathscr{L}_{N}+\frac{a}{N}\sum_{i\in I_{c_{1}}^{N}}\left[\mathfrak{e}_{i}\sum_{j\in I_{c_{2}}^{N}}\left\{ \frac{\chi_{\epsilon}(x_{j}-x_{i})}{2\epsilon}-\delta^{+}(x_{j}-x_{i})-\delta^{+}(x_{i}-x_{j})\right\} \right]
\]
on $\mathcal{D}(\mathscr{L}_{N})$. The variational formula for $\lambda_{N,\epsilon,a}$
is 
\begin{align*}
\sup_{f\in\mathscr{P}_{N}\cap\mathcal{D}(\mathscr{L}_{N})}\Biggl\{\frac{a}{N}\sum_{i\in I_{c_{1}}^{N},\,j\in I_{c_{2}}^{N}}\mathfrak{e}_{i}\Biggl[ & \int_{G_{N}}f(x)\frac{\chi_{\epsilon}(x_{j}-x_{i})}{2\epsilon}dx\\
-\frac{1}{2}\Biggl(\int_{F_{ij}}f_{ij}(x) & dS_{ij}(x)+\int_{F_{ji}}f_{ji}(x)dS_{ji}(x)\Biggr)\Biggr]-\mathcal{D}_{N}(f)\Biggr\}.
\end{align*}
By Proposition \ref{Main Replacement Lemma}, the expression inside
sup (and thus $\lambda_{N,\epsilon,a}$) can be bounded by
\begin{eqnarray*}
N\left[Ca\left\{ 1+\left(\frac{\mathcal{D}_{N}(f)}{N}\right)^{\frac{7}{8}}\right\} \left(\epsilon^{\frac{1}{4}}+\frac{1}{\sqrt{N}}\right)-\frac{\mathcal{D}_{N}(f)}{N}\right]\\
\le C'N\left[a\left(\epsilon^{\frac{1}{4}}+\frac{1}{\sqrt{N}}\right)+a^{8}\left(\epsilon^{\frac{1}{4}}+\frac{1}{\sqrt{N}}\right)^{8}\right]
\end{eqnarray*}
where $C,\,C'$ are proper constants. Therefore, by (\ref{eq:suabn}),
(\ref{eq:suan}) is bounded by 
\begin{equation}
2^{N}\exp\{CN(a\epsilon^{\frac{1}{4}}+a^{8}\epsilon^{2}+o_{N}(1))(t_{2}-t_{1})\}\label{eq:traiw}
\end{equation}
Finally, by (\ref{eq:sual}) and (\ref{eq:traiw}), 
\[
\frac{1}{N}\log\mathbb{P}_{N}^{eq}\left[\mathbf{C}_{N}^{c_{1},\,c_{2}}(t_{1},\,t_{2};\epsilon,\,\delta)\right]\le-a\delta+\log2+C(a\epsilon^{\frac{1}{4}}+a^{8}\epsilon^{2})(t_{2}-t_{1})+o_{N}(1)
\]
and therefore 
\[
\limsup_{\epsilon\rightarrow0}\limsup_{N\rightarrow\infty}\frac{1}{N}\log\mathbb{P}_{N}^{eq}\left[\mathbf{C}_{N}^{c_{1},\,c_{2}}(t_{1},\,t_{2};\epsilon,\,\delta)\right]\le-a\delta+\log2.
\]
Since $a>0$ is arbitrary, we are done. \end{proof}
\begin{rem}
Previous argument combining Chebyshev's inequality, Feynman-Kac's
formula and the variational formula for the maximal eigenvalue, in
the context of interacting particle system, has been originally introduced
by \cite{DV} and also explained thoroughly in Chapter 10 of \cite{KL}.
This method will be used frequently and implicitly in the remaining
part of the current article.
\end{rem}

\subsubsection{Small Time Regime}

We now prove Proposition \ref{prop: smal time regime}. By Chebyshev's
inequality, 
\begin{align*}
 & \frac{1}{N}\log\mathbb{P}_{N}\left[\frac{1}{N}\sum_{i\in I_{c_{1}}^{N}}\left|\int_{0}^{\eta}\rho_{\epsilon,i}^{(c_{2})}(x^{N}(t))dt-A_{i,c_{2}}^{N}(\eta)\right|>\delta\right]\\
 & \le-a\delta+\frac{1}{N}\log\mathbb{E}_{N}\exp\left\{ a\sum_{i\in I_{c_{1}}^{N}}\left|\int_{0}^{\eta}\rho_{\epsilon,i}^{(c_{2})}(x^{N}(t))dt-A_{i,c_{2}}^{N}(\eta)\right|\right\} \\
 & \le-a\delta+\frac{1}{N}\log\mathbb{E}_{N}\exp\left\{ a\sum_{i=1}^{N}\int_{0}^{\eta}\rho_{\epsilon,i}(x^{N}(t))dt+aNA^{N}(\eta)\right\} 
\end{align*}
or any $a>0$ where $\mathbb{E}_{N}$ is the expectation with respect
to $\mathbb{P}_{N}$. Accordingly, we only need to establish the following
estimates. 
\begin{prop}
\label{prop: smal time regime-1}For any $a>0$,\textup{
\begin{align}
 & \limsup_{\eta\rightarrow0}\limsup_{N\rightarrow\infty}\frac{1}{N}\log\mathbb{E}_{N}\exp\left\{ aNA^{N}(\eta)\right\} \le0\label{eq:huy1-1}\\
 & \limsup_{\eta\rightarrow0}\limsup_{\epsilon\rightarrow0}\limsup_{N\rightarrow\infty}\frac{1}{N}\log\mathbb{E}_{N}\exp\left\{ a\sum_{i=1}^{N}\int_{0}^{\eta}\rho_{\epsilon,i}\left(x^{N}(t)\right)dt\right\} \le0\label{eq: huy-1}
\end{align}
} \end{prop}
\begin{proof}
This proposition is a special case in which the labels of particles
(and thus the interaction) play no role. Therefore, for this proposition
we can temporarily assume that $x_{1}^{N}(t),\,x_{2}^{N}(t),\,\cdots,\,x_{N}^{N}(t)$
move by the way of independent Brownian motions. For $x\in\mathbb{T}^{N}$,
let us define $G_{N}(x)=\frac{a}{N}\sum_{i\neq j}g(x_{i}-x_{j})$
where $g(x)=\frac{x(1-x)}{2}$ is a continuous function on $\mathbb{T}$.
Then, by Tanaka's formula,
\begin{align*}
 & G_{N}(x^{N}(\eta))-G_{N}(x^{N}(0))+\frac{a\eta(N-1)}{2}-\frac{2a}{N}\sum_{i\neq j}A_{ij}^{N}(\eta)\\
 & =\frac{a}{N}\sum_{i=1}^{N}\int_{0}^{\eta}\left[\sum_{j:j\neq i}g'(x_{i}^{N}(s)-x_{j}^{N}(s))\right]dx_{i}^{N}(s)
\end{align*}
and thus by Girsanov's theorem, 
\begin{align}
 & \mathbb{E}_{N}\exp\left\{ 2aNA^{N}(\eta)-[G_{N}(x^{N}(\eta))-G_{N}(x^{N}(0))]-\Lambda_{N}(\eta)\right\} \label{eq: hhy}\\
 & =\exp\frac{a\eta(N-1)}{2}\nonumber 
\end{align}
where
\begin{equation}
\Lambda_{N}(\eta)=\frac{a^{2}}{2N^{2}}\sum_{i=1}^{N}\int_{0}^{\eta}\left(\sum_{j:j\neq i}g'(x_{i}^{N}(s)-x_{j}^{N}(s))\right)^{2}ds\le\frac{a^{2}\eta N}{8}\label{eq:hhu}
\end{equation}
since $|g'(x)|\le\frac{1}{2}$. By (\ref{eq: hhy}) and (\ref{eq:hhu})
we can obtain 
\begin{align}
 & \mathbb{E}_{N}\exp\left\{ 2aNA^{N}(\eta)-\left[G_{N}(x^{N}(\eta))-G_{N}(x^{N}(0))\right]\right\} \label{eq:c_sy1}\\
 & \le\exp\left\{ \frac{a\eta N}{2}+\frac{a^{2}\eta N}{8}\right\} .\nonumber 
\end{align}
Moreover, by the mean value theorem, 
\begin{align}
\mathbb{E}_{N}\exp\left\{ G_{N}(x^{N}(\eta))-G_{N}(x^{N}(0))\right\}  & \le\mathbb{E}_{N}\left[\exp\left\{ a\sum_{i=1}^{N}\left|x_{i}^{N}(\eta)-x_{i}^{N}(0)\right|\right\} \right]\nonumber \\
 & \le\left\{ \left(1+a\sqrt{\frac{2\eta}{\pi}}\right)\exp\frac{a^{2}\eta}{2}\right\} ^{N}\label{eq:c_sy2}
\end{align}
since we have assumed that $x_{i}^{N}(t),\,1\le i\le N$ are independent
Brownian motions. By (\ref{eq:c_sy1}), (\ref{eq:c_sy2}) and Cauchy-Schwarz's
inequality, 
\begin{equation}
\mathbb{E}_{N}\exp\left\{ aNA^{N}(\eta)\right\} \le\left(1+a\sqrt{\frac{2\eta}{\pi}}\right)^{\frac{N}{2}}\exp\left\{ \frac{a\eta N}{4}+\frac{5a^{2}\eta N}{16}\right\} \label{eq:Estimate of Total Local Time}
\end{equation}
and we proved (\ref{eq:huy1-1}). 

For (\ref{eq: huy-1}), we define $p_{\epsilon}(x)=\frac{1}{2\epsilon}u_{\epsilon}(x)$
where $u_{\epsilon}$ is the function defined in Lemma \ref{Lemma6}
and thus $p_{\epsilon}''(x)=\frac{1}{2\epsilon}\chi_{\epsilon}(x)$.
Define $H_{N}(x)=\frac{4a}{N}\sum_{i\neq j}p_{\epsilon}(x_{i}-x_{j})$
and apply Tanaka's formula such that
\begin{align*}
 & H_{N}(x^{N}(\eta))-H_{N}(x^{N}(0))-2a\left[\int_{0}^{\eta}\sum_{i=1}^{N}\rho_{\epsilon,i}\left(x^{N}(t)\right)dt-NA^{N}(\eta)\right]\\
 & =\frac{4a}{N}\sum_{i=1}^{N}\int_{0}^{\eta}\left[\sum_{j:j\neq i}p_{\epsilon}'(x_{i}^{N}(s)-x_{j}^{N}(s))\right]dx_{i}^{N}(s).
\end{align*}
Note that we have $\left|p_{\epsilon}'(x)\right|\le\frac{1}{2}$ and
therefore we can deduce
\begin{align*}
 & \limsup_{\eta\rightarrow0}\limsup_{\epsilon\rightarrow0}\limsup_{N\rightarrow\infty}\\
 & \frac{1}{N}\log\mathbb{E}_{N}\exp\left\{ a\left(\int_{0}^{\eta}\sum_{i=1}^{N}\rho_{\epsilon,i}\left(x^{N}(t)\right)dt-NA^{N}(\eta)\right)\right\} \le0
\end{align*}
for all $a>0$ by the exactly identical way to the previous step.
Thus, we can conclude (\ref{eq: huy-1}) as well. 
\end{proof}

\section{Exponential Tightness}

In this section, we establish the exponential tightness of $\bigl\{\widetilde{\mathbb{Q}}_{N}\bigr\}_{N=1}^{\infty}$,
which can be deduced from the following result. 
\begin{thm}
\label{thm: SUPER EXP TIGHTNESS}For any $\epsilon,\,\alpha>0$, 
\[
\limsup_{\delta\rightarrow0}\limsup_{N\rightarrow\infty}\frac{1}{N}\log\mathbb{P}_{N}\left[\left|\left\{ i:\sup_{\substack{0\le s,\,t\le T\\
|s-t|\le\delta
}
}\left|x_{i}^{N}(t)-x_{i}^{N}(s)\right|\ge\epsilon\right\} \right|\ge N\alpha\right]=-\infty.
\]

\end{thm}
Before proving Theorem \ref{thm: SUPER EXP TIGHTNESS}, we briefly
explain the reason for the exponential tightness of $\bigl\{\widetilde{\mathbb{Q}}_{N}\bigr\}_{N=1}^{\infty}$
being a corollary of this theorem. We can prove exponential tightness
of $\bigl\{\widetilde{\mathbb{Q}}_{N}\bigr\}_{N=1}^{\infty}$ by showing
\[
\lim_{\delta\rightarrow0}\limsup_{N\rightarrow\infty}\frac{1}{N}\log\mathbb{P}_{N}\left[\sup_{\substack{0\le s,\,t\le T\\
|s-t|\le\delta
}
}\frac{1}{N}\left|\sum_{c=1}^{m}\sum_{i\in I_{c}^{N}}\left(J_{c}(x_{i}^{N}(t))-J_{c}(x_{i}^{N}(s))\right)\right|\ge\epsilon\right]=-\infty
\]
for any $(J_{1},\,J_{2},\,\cdots,\,J_{m})\in C(\mathbb{T})^{m}$ and
$\epsilon>0$. It is obvious that this estimate is a direct consequence
of 
\[
\lim_{\delta\rightarrow0}\limsup_{N\rightarrow\infty}\frac{1}{N}\log\mathbb{P}_{N}\left[\sup_{\substack{0\le s,\,t\le T\\
|s-t|\le\delta
}
}\frac{1}{N}\sum_{i\in I_{c}^{N}}\left|x_{i}^{N}(t)-x_{i}^{N}(s)\right|\ge\epsilon\right]=-\infty
\]
for each $c$. We can deduce this estimate from Theorem \ref{thm: SUPER EXP TIGHTNESS}
since 
\begin{align*}
 & \mathbb{P}_{N}\left[\sup_{\substack{0\le s,\,t\le T,\,|s-t|\le\delta}
}\frac{1}{N}\sum_{i\in I_{c}^{N}}\left|x_{i}^{N}(t)-x_{i}^{N}(s)\right|\ge\epsilon\right]\\
 & \le\mathbb{P}_{N}\left[\left|\left\{ i:\sup_{\substack{0\le s,\,t\le T,\,|s-t|\le\delta}
}\left|x_{i}^{N}(t)-x_{i}^{N}(s)\right|\ge\frac{\epsilon}{2}\right\} \right|\ge\frac{N\epsilon}{2}\right].
\end{align*}

We return now to Theorem \ref{thm: SUPER EXP TIGHTNESS}. The basic
strategy is to divide the estimate into the normal and small time
regimes as before. To carry this out, we first observe that 
\[
\left\{ i:\sup_{\substack{0\le s,\,t\le T,\,|s-t|\le\delta}
}\left|x_{i}^{N}(t)-x_{i}^{N}(s)\right|\ge\epsilon\right\} \subset S_{\frac{\epsilon}{2},\delta}([0,\,\eta])\cup S_{\frac{\epsilon}{2},\delta}([\eta,\,T])
\]
where
\begin{align}
S_{\epsilon,\delta}([\eta,\,T]) & =\left\{ i:\sup_{\substack{\eta\le s,\,t\le T,\,|s-t|\le\delta}
}\left|x_{i}^{N}(t)-x_{i}^{N}(s)\right|\ge\epsilon\right\} \label{eq:Long}\\
S_{\epsilon,\delta}([0,\,\eta]) & =\left\{ i:\sup_{\substack{0\le s,\,t\le\eta,\,|s-t|\le\delta}
}\left|x_{i}^{N}(t)-x_{i}^{N}(s)\right|\ge\epsilon\right\} \label{eq:short}
\end{align}
for all $\eta>0$. Consequently, Theorem \ref{thm: SUPER EXP TIGHTNESS}
can be separated into the following propositions. The first one is
the normal time regime type of estimate. 
\begin{prop}[Normal time regime]
\label{prop: normal time reg}For any $\eta,\,\epsilon,\,\alpha>0$
\begin{equation}
\limsup_{\delta\rightarrow0}\limsup_{N\rightarrow\infty}\frac{1}{N}\log\mathbb{P}_{N}\left[\left|S_{\epsilon,\delta}([\eta,\,T])\right|\ge N\alpha\right]=-\infty.\label{eq:regular time regime}
\end{equation}

\end{prop}
For the small time regime, we have 
\[
S_{\epsilon,\delta}([0,\,\eta])\subset\left\{ i:\sup_{\substack{0\le t\le\eta}
}\left|x_{i}^{N}(t)-x_{i}^{N}(0)\right|\ge\frac{\epsilon}{2}\right\} 
\]
 and hence it is enough to prove the following estimate. 
\begin{prop}[Small time regime]
\label{prop:(Small-time-regime)}For any $\epsilon,\,\alpha>0$ 
\[
\limsup_{\eta\rightarrow0}\limsup_{N\rightarrow\infty}\frac{1}{N}\log\mathbb{P}_{N}\left[\left|\left\{ i:\sup_{\substack{0\le t\le\eta}
}\left|x_{i}^{N}(t)-x_{i}^{N}(0)\right|\ge\epsilon\right\} \right|\ge N\alpha\right]=-\infty.
\]

\end{prop}
For the normal time regime, it is possible to transfer the estimate
to that of the equilibrium process by using Lemma \ref{lem: symmetrization lemma}.
Then, we can apply the well known methodology (\textit{e.g.}, \cite{R})
based on Garsia-Rumsey-Rodemich's inequality to the equilibrium estimate
by making a small adjustment. However, for the small time regime,
we cannot send the estimate to the equilibrium and therefore have
to adopt a different approach.

\subsection{Normal Time Regime}

By setting $f(x)=x_{i}$ in (\ref{eq:tanaka}) and (\ref{eq:martingale}),
we obtain 
\begin{equation}
x_{i}^{N}(t)=\beta_{i}(t)+\tilde{A}_{i}^{N}(t).\label{x_i as brown+local}
\end{equation}
where
\begin{equation}
\tilde{A}_{i}^{N}(t)=\sum_{j:j\neq i}\left[A_{ij}^{N}(t)-A_{ji}^{N}(t)\right]\label{eq:Loc1}
\end{equation}
which can be regarded as the difference between the left and right
collision times for particle $x_{i}^{N}(\cdot)$ and also measures
the deviation of the lifted particle $x_{i}^{N}(t)$ from the underlying
Brownian motion $\beta_{i}(t)$. In contrast to the averaged local
times in (\ref{eq:Loc2}), (\ref{eq:Loc3}) or (\ref{eq:Loc4}), the
behavior of $\tilde{A}_{i}^{N}(t)$ is unacceptably noisy. Thus, we
now present a way to control this object. 

By (\ref{x_i as brown+local}), the estimate (\ref{eq:regular time regime})
can be divided into
\begin{align}
 & \limsup_{\delta\rightarrow0}\limsup_{N\rightarrow\infty}\label{eq:(REG_BR)}\\
 & \frac{1}{N}\log\mathbb{P}_{N}\left[\left|\left\{ i:\sup_{\substack{\eta\le s,\,t\le T,\,|s-t|\le\delta}
}\left|\beta_{i}(t)-\beta_{i}(s)\right|\ge\epsilon\right\} \right|\ge N\alpha\right]=-\infty,\nonumber \\
 & \limsup_{\delta\rightarrow0}\limsup_{N\rightarrow\infty}\label{eq:(REG_INC)}\\
 & \frac{1}{N}\log\mathbb{P}_{N}\left[\left|\left\{ i:\sup_{\substack{\eta\le s,\,t\le T,\,|s-t|\le\delta}
}\left|\widetilde{A}_{i}^{N}(t)-\widetilde{A}_{i}^{N}(s)\right|\ge\epsilon\right\} \right|\ge N\alpha\right]=-\infty.\nonumber 
\end{align}
First of all, (\ref{eq:(REG_BR)}) is standard because $\beta_{i}$'s
are independent Brownian motion. The main challenge is (\ref{eq:(REG_INC)}).
Since the event inside the bracket of (\ref{eq:(REG_INC)}) is invariant
under the permutation of labels, we can apply Lemma \ref{lem: symmetrization lemma}
to send the estimate to the equilibrium as following proposition. 
\begin{prop}
\label{prop: normal time reg-1} For any $\epsilon,\,\alpha>0$ and
$T\ge1$,
\[
\limsup_{\delta\rightarrow0}\limsup_{N\rightarrow\infty}\frac{1}{N}\log\mathbb{P}_{N}^{eq}\left[\left|\left\{ i:\sup_{\substack{0\le s,\,t\le T\\
|s-t|\le\delta
}
}\left|\widetilde{A}_{i}^{N}(t)-\widetilde{A}_{i}^{N}(s)\right|\ge\epsilon\right\} \right|\ge N\alpha\right]=-\infty.
\]
\end{prop}
\begin{rem}
We expanded the time window from $[\eta,\,T]$ to $[0,\,T]$ to reduce
unnecessary notational complexity. We also assumed $T\ge1$ without
loss of generality.
\end{rem}

\subsubsection{Garcia-Rumsey-Rodemich's Inequality}

For $\phi\in C([0,\,T],\,\mathbb{R})$, let us define 
\[
S_{T}(\phi)=\sup_{0\le\delta\le\frac{1}{2}}\sup_{\substack{0\le s,\,t\le T\\
|s-t|\le\delta
}
}\frac{\left|\phi(t)-\phi(s)\right|}{\sqrt[4]{\delta}\log\frac{1}{\delta}}
\]
then we have the following result. 
\begin{lem}
\label{lem: GRR LEMMA}For $T\ge1$, we have
\[
S_{T}(\phi)\le C_{1}+C_{2}\log\int_{0}^{T}\int_{0}^{T}\exp\left\{ \left|\frac{\phi(t)-\phi(s)}{\sqrt[4]{t-s}}\right|\right\} dtds
\]
for some positive constants $C_{1},\,C_{2}$.\end{lem}
\begin{proof}
Define $p(x)=x^{\frac{1}{4}}$, $\Psi(x)=e^{|x|}-1$ and $M=\int_{0}^{T}\int_{0}^{T}\exp\left\{ \left|\frac{\phi(t)-\phi(s)}{\sqrt[4]{t-s}}\right|\right\} dtds$
so that $M\ge T^{2}\ge1$. For $|t-s|\le\delta$, by Garcia-Rumsey-Rodemich's
inequality (\textit{cf.} Section 1.3 of \cite{V5}), 
\begin{align*}
\left|\phi(t)-\phi(s)\right| & \le8\int_{0}^{\left|t-s\right|}\log\left\{ 1+\frac{4\left(M-T^{2}\right)}{u^{2}}\right\} dp(u)\\
 & \le2\int_{0}^{\delta}u^{-\frac{3}{4}}\log\left(M+\frac{4M}{u^{2}}\right)du\\
 & =8\delta^{\frac{1}{4}}\log M+2\int_{0}^{\delta}u^{-\frac{3}{4}}\log\left(1+\frac{4}{u^{2}}\right)du.
\end{align*}
Therefore, the proof is completed since we have 
\[
\int_{0}^{\delta}u^{-\frac{3}{4}}\log\left(1+\frac{4}{u^{2}}\right)du<\int_{0}^{\delta}u^{-\frac{3}{4}}\left(2+2\log\frac{1}{u}\right)du=40\delta^{\frac{1}{4}}+8\delta^{\frac{1}{4}}\log\frac{1}{\delta}.
\]

\end{proof}

\subsubsection{Proof of Proposition \ref{prop: normal time regime}}

We return now to Proposition \ref{prop: normal time regime}. For
$\delta\le\frac{1}{2}$, we have 
\begin{align*}
 & \mathbb{P}_{N}^{eq}\left[\left|\left\{ i:\sup_{\substack{0\le s,\,t\le T,\,|s-t|\le\delta}
}\left|\widetilde{A}_{i}^{N}(t)-\widetilde{A}_{i}^{N}(s)\right|\ge\epsilon\right\} \right|\ge N\alpha\right]\\
 & \le\mathbb{P}_{N}^{eq}\left[\sum_{i=1}^{N}\sup_{\substack{0\le s,\,t\le T,\,|s-t|\le\delta}
}\left|\widetilde{A}_{i}^{N}(t)-\widetilde{A}_{i}^{N}(s)\right|\ge N\alpha\epsilon\right]\\
 & \le\mathbb{P}_{N}^{eq}\left[\sum_{i=1}^{N}S_{T}\left(\widetilde{A}_{i}^{N}\right)\ge\frac{N\alpha\epsilon}{\sqrt[4]{\delta}\log\frac{1}{\delta}}\right]\\
 & \le\exp\left\{ -\frac{N\alpha\epsilon}{C_{2}\sqrt[4]{\delta}\log\frac{1}{\delta}}\right\} \mathbb{E}_{N}^{eq}\left[\exp\left\{ \frac{1}{C_{2}}\sum_{i=1}^{N}S_{T}\left(\widetilde{A}_{i}^{N}\right)\right\} \right]
\end{align*}
where $C_{2}$ is the constant from Lemma \ref{lem: GRR LEMMA}. Therefore
it suffices to show 
\begin{equation}
\limsup_{N\rightarrow\infty}\frac{1}{N}\log\mathbb{E}_{N}^{eq}\left[\exp\left\{ \frac{1}{C_{2}}\sum_{i=1}^{N}S_{T}\left(\widetilde{A}_{i}^{N}\right)\right\} \right]\le C\label{eq:CSY grr}
\end{equation}
where $C$ is a constant which does not depend on $\delta$. By Lemma
\ref{lem: GRR LEMMA}, 
\begin{align}
 & \mathbb{E}_{N}^{eq}\exp\left\{ \frac{1}{C_{2}}\sum_{i=1}^{N}S_{T}\left(\widetilde{A}_{i}^{N}\right)\right\} \label{eq: Tight_FIN}\\
 & \le\mathbb{E}_{N}^{eq}\left[e^{\frac{C_{1}}{C_{2}}N}\prod_{i=1}^{N}\int_{0}^{T}\int_{0}^{T}\exp\left\{ \frac{\left|\widetilde{A}_{i}^{N}(t)-\widetilde{A}_{i}^{N}(s)\right|}{\sqrt[4]{t-s}}\right\} dtds\right]\nonumber \\
 & =e^{\frac{C_{1}}{C_{2}}N}\int_{0}^{T}\cdots\int_{0}^{T}\mathbb{E}_{N}^{eq}\exp\left\{ \sum_{i=1}^{N}\frac{\left|\widetilde{A}_{i}^{N}(t_{i})-\widetilde{A}_{i}^{N}(s_{i})\right|}{\sqrt[4]{t_{i}-s_{i}}}\right\} dt_{1}ds_{1}\cdots dt_{N}ds_{N}\nonumber \\
 & \le e^{\frac{C_{1}}{C_{2}}N}\sum_{\mathfrak{e}_{i}=\pm1,\,\forall i}\int_{0}^{T}\cdots\int_{0}^{T}\mathbb{E}_{N}^{eq}\exp\left\{ \sum_{i=1}^{N}\mathfrak{e}_{i}\frac{\widetilde{A}_{i}^{N}(t_{i})-\widetilde{A}_{i}^{N}(s_{i})}{\sqrt[4]{t_{i}-s_{i}}}\right\} dt_{1}\cdots ds_{N}.\nonumber 
\end{align}
We will prove the following lemma in the next subsection. 
\begin{lem}
\label{lem: LLLEma}For any $\alpha_{i}$ and $0\le s_{i}<t_{i}\le T$,
\begin{equation}
\mathbb{E}_{N}^{eq}\left[\exp\left\{ \sum_{i=1}^{N}\alpha_{i}\left(\widetilde{A}_{i}^{N}(t_{i})-\widetilde{A}_{i}^{N}(s_{i})\right)\right\} \right]\le\exp\left\{ C\sum_{i=1}^{N}(\alpha_{i}^{2}+\alpha_{i}^{4})(t_{i}-s_{i})\right\} \label{eq: TIGHT_final object}
\end{equation}
where $C$ is a constant only depending on $T$. 
\end{lem}
By assuming this lemma, we can bound (\ref{eq: Tight_FIN}) by
\begin{eqnarray*}
 &  & e^{\frac{C_{1}}{C_{2}}N}\sum_{\mathfrak{e}_{i}=\pm1\,\forall i}\int_{0}^{T}\int_{0}^{T}\cdots\int_{0}^{T}\exp\left\{ C\sum_{i=1}^{N}\left(\sqrt{t_{i}-s_{i}}+1\right)\right\} dt_{1}ds_{1}\cdots dt_{N}ds_{N}\\
 &  & \le e^{\frac{C_{1}}{C_{2}}N}2^{N}T^{2N}e^{C(\sqrt{T}+1)N}
\end{eqnarray*}
and hence (\ref{eq:CSY grr}) is proven.

\subsubsection{Proof of Lemma \ref{lem: LLLEma}}

The final step to prove Proposition \ref{prop: normal time reg-1}
is Lemma \ref{lem: LLLEma}. We prove this lemma by a series of estimates. 
\begin{lem}
\label{lem: step1}For any \textup{$\alpha_{i}$ and $0\le s_{i}<t_{i}\le T$,}
\begin{equation}
\mathbb{E}_{N}^{eq}\exp\left\{ \sum_{i=1}^{N}\left[\alpha_{i}\left(\widetilde{A}_{i}^{N}(t_{i})-\widetilde{A}_{i}^{N}(s_{i})\right)-\frac{\alpha_{i}^{2}}{\lambda}\left(A_{i}^{N}(t_{i})-A_{i}^{N}(s_{i})\right)\right]\right\} \le1.\label{eq: TGHT1}
\end{equation}
\end{lem}
\begin{proof}
Let us define $V(t,\,x)=\sum_{i=1}^{N}\mathds{1}_{[s_{i,}t_{i}]}(t)V_{i}(x)$
where
\[
V_{i}(x)=\sum_{j:j\neq i}\left[\alpha_{i}\left(\delta_{+}(x_{i}-x_{j})-\delta_{+}(x_{j}-x_{i})\right)-\frac{\alpha_{i}^{2}}{\lambda N}\left(\delta_{+}(x_{i}-x_{j})+\delta_{+}(x_{j}-x_{i})\right)\right].
\]
Note that we can rewrite (\ref{eq: TGHT1}) as
\begin{equation}
\mathbb{E}_{N}^{eq}\exp\left\{ \int_{s_{i}}^{t_{i}}V(t,\,x^{N}(t))dt\right\} \le1.\label{eq: okoo}
\end{equation}
Now, as in the proof of Proposition \ref{prop: eq_normal time regime},
we can obtain
\begin{align}
 & \mathbb{E}_{N}^{eq}\exp\left\{ \int_{s_{i}}^{t_{i}}V\left(t,\,x^{N}(t)\right)dt\right\} \label{eq:VAR}\\
 & \le\exp\left\{ \int_{s_{i}}^{t_{i}}\sup_{f\in\mathscr{P}_{N}\cap\mathcal{D}(\mathscr{L}_{N})}\left\{ \int_{G_{N}}V(t,\,x)f(x)dx-\mathcal{D}_{N}(f)\right\} \right\} \nonumber 
\end{align}
by Feynman-Kac's formula and the variational formula for the largest
eigenvalue of $\mathscr{L}_{N}+V$. Note that we can bound $\int_{G_{N}}V(t,\,x)f(x)dx-\mathcal{D}_{N}(f)$
by
\[
\sum_{i\neq j}\int_{F_{ij}}\left[\alpha_{i}\left|f_{ij}-f_{ji}\right|-\frac{\alpha_{i}^{2}}{\lambda N}(f_{ij}+f_{ji})-\frac{\lambda N}{2}\left(\sqrt{f_{ij}}-\sqrt{f_{ji}}\right)^{2}\right](x)dS_{ij}(x).
\]
It is not difficult to check the last expression is non-positive because
of the elementary inequality
\[
x^{2}(a+b)+\frac{1}{2}y^{2}\left(\sqrt{a}-\sqrt{b}\right)^{2}\ge xy\left|a-b\right|
\]
for $a,\,b\ge0.$ Thus, the RHS of (\ref{eq:VAR}) is bounded by $1$
and hence (\ref{eq: okoo}) holds.
\end{proof}
The next estimate is a stronger version of Lemma \ref{Lemma2}. 
\begin{lem}
\label{lem: est }For any $1\le i\le N$ and $f\in\mbox{\ensuremath{\mathscr{P}}}_{N}$,
we have 
\[
\sum_{j:j\neq i}\int_{F_{ij}}(f_{ij}(x)+f_{ji}(x))dS_{ij}(x)\le2N+\sqrt{32N\mathcal{D}_{N}(f)}.
\]
\end{lem}
\begin{proof}
We introduce a function $\sigma_{k}^{(i)}(x)$ on $G_{N}$ for $k\neq i$
by
\[
\sigma_{k}^{(i)}(x)=\sum_{j=1}^{N}1_{[0,\,x_{k}-x_{i}]}(x_{j}-x_{i})
\]
which counts the number of particles between $x_{i}$ and $x_{k}$
in the clockwise sense. We remark here that this function also appeared
in \cite{G1} to estimate $A_{i}^{N}(t)$. We normalize $\sigma_{k}^{(i)}(x)$
by $c_{k}^{(i)}(x)=\sigma_{k}^{(i)}(x)-\frac{N+2}{2}$ and set $c_{i}^{(i)}(x)=0$
for simplicity. We can define a piecewise constant vector field $\mathbf{C}_{i}(x)=\sum_{k=1}^{N}c_{k}^{(i)}(x)e_{k}$
and apply Green's formula (\ref{eq:green}). First note that $\nabla\cdot\mathbf{\mathbf{C}}_{i}=0$
and thus we only need to concern about boundary terms. On $F_{kl}$
with $k,\,l\neq i$, we have 
\[
\left\langle \mathbf{C}_{i}(x),\,e_{k}-e_{l}\right\rangle =c_{k}^{(i)}(x)-c_{l}^{(i)}(x)=-1
\]
because $x_{l}=x_{k}+0$ on $F_{kl}$ and hence $\sigma_{l}^{(i)}(x)=\sigma_{k}^{(i)}(x)+1$.
For the boundary $F_{ji}$,
\[
\left\langle \mathbf{C}_{i}(x),\,e_{j}-e_{i}\right\rangle =c_{j}^{(i)}(x)-c_{i}^{(i)}(x)=\frac{N-2}{2}
\]
because $\sigma_{j}^{(i)}(x)=N$ on $F_{ji}$ and $c_{i}^{(i)}(x)=0$.
Similarly, on $F_{ij}$, 
\[
\left\langle \mathbf{C}_{i}(x),\,e_{i}-e_{j}\right\rangle =c_{i}^{(i)}(x)-c_{j}^{(i)}(x)=\frac{N-2}{2}.
\]
We now apply Green's formula with the vector field $\mathbf{C}_{i}(x)$:
\begin{align*}
\sum_{k=1}^{N} & \int_{G_{N}}\left[\nabla_{k}f(x)\right]c_{k}(x)dx\\
= & -\sum_{\substack{k,\,l:k\neq l,\,k,l\neq i}
}\int_{F_{kl}}f_{kl}(x)dS_{kl}(x)\\
 & +\frac{N-2}{2}\sum_{j:j\neq i}\left[\int_{F_{ij}}f_{ij}(x)dS_{ij}(x)+\int_{F_{ij}}f_{ji}(x)dS_{ji}(x)\right]\\
= & -\frac{1}{2}\sum_{u\neq v}\int_{F_{uv}}(f_{uv}+f_{vu})(x)dS_{uv}(x)+\frac{N}{2}\sum_{j:j\neq i}\int_{F_{ij}}(f_{ij}+f_{ji})(x)dS_{ij}(x).
\end{align*}
Since we have bound $|c_{k}(x)|\le\frac{N-2}{2}<\frac{N}{2}$, we
can derive 
\begin{align*}
 & \sum_{j:j\neq i}\int_{F_{ij}}(f_{ij}+f_{ji})(x)dS_{ij}(x)\\
 & \le\frac{1}{N}\sum_{u\neq v}\int_{F_{uv}}(f_{uv}+f_{vu})(x)dS_{uv}(x)+\sum_{k=1}^{N}\int_{G_{N}}\left|\nabla_{k}f(x)\right|dx.
\end{align*}
Note here that the RHS is bounded by $2N+2\sqrt{8N\mathcal{D}_{N}(f)}$
due to Lemmas \ref{Lemma1} and \ref{Lemma2}\end{proof}
\begin{lem}
\label{lem: step2}For any \textup{$\alpha_{i}\ge0$ }and $0\le s_{i}<t_{i}\le T$
\[
\mathbb{E}_{N}^{eq}\exp\left\{ \sum_{i=1}^{N}\alpha_{i}\left(A_{i}^{N}(t_{i})-A_{i}^{N}(s_{i})\right)\right\} \le\exp\left\{ \sum_{i=1}^{N}8(\alpha_{i}+\alpha_{i}^{2})(t_{i}-s_{i})\right\} .
\]
\end{lem}
\begin{proof}
It suffices to prove 
\begin{equation}
\mathbb{E}_{N}^{eq}\exp\left\{ \int_{s_{i}}^{t_{i}}V_{i}(x^{N}(t))dt\right\} \le\exp\left\{ 8N(\alpha+\alpha^{2})(t_{i}-s_{i})\right\} \label{eq:obtv}
\end{equation}
for any $i$ and $\alpha\ge0$, where $V_{i}(x)=\alpha\sum_{j:j\neq i}\left[\delta_{+}(x_{i}-x_{j})+\delta_{+}(x_{j}-x_{i})\right].$
Note that the LHS of (\ref{eq:obtv}) is bounded above by
\[
\exp\left\{ (t_{i}-s_{i})\sup_{f\in\mathscr{P}_{N}\cap\mathscr{D}(\mathscr{L}_{N})}\left\{ \alpha\sum_{j:j\neq i}\int_{F_{ij}}(f_{ij}(x)+f_{ji}(x))dS_{ij}(x)-\mathcal{D}_{N}(f)\right\} \right\} 
\]
as before. Finally, by Lemma \ref{lem: est },
\begin{align*}
 & \alpha\sum_{j:j\neq i}\int_{F_{ij}}(f_{ij}(x)+f_{ji}(x))dS_{ij}(x)-\mathcal{D}_{N}(f)\\
 & \le\alpha\left(2N+\sqrt{32N\mathcal{D}_{N}(f)}\right)-\mathcal{D}_{N}(f)\\
 & \le8N(\alpha+\alpha^{2})
\end{align*}
and we are done. 
\end{proof}
Consequently, we are able to prove Lemma \ref{lem: LLLEma} by Lemmas
\ref{lem: step1}, \ref{lem: step2} and Cauchy-Schwarz's inequality.

\subsection{\label{sub:Small-Time-Regime}Small Time Regime}

In this subsection, we provide a detailed proof of Proposition \ref{prop:(Small-time-regime)}.
The small time regime differs from the normal time regime in that
$\tilde{A}_{i}^{N}(t)$ on $t\in[0,\,\eta]$ cannot be properly controlled
if the process does not start from the neighborhood of the equilibrium.
Therefore, it is not possible to work directly with $x_{i}^{N}(\cdot)$
as in the normal time regime. Instead, we introduce an intermediate
process $z_{i}^{N}(\cdot)$ where
\begin{equation}
z_{i}^{N}(t)=x_{i}^{N}(t)+\frac{1}{N(\lambda+1)}\sum_{j:j\neq i}\nu(x_{j}^{N}(t)-x_{i}^{N}(t))\,\,;\,\,1\le i\le N\label{eq:z(t)}
\end{equation}
where $\nu(x)=x$ on $[0,\,1].$ These adjusted processes were introduced
in \cite{G1} and were turned out to be martingales with respect to
the same filtration with $x^{N}(t)$. More precisely, we can prove
that 
\begin{equation}
z_{i}^{N}(t)-z_{i}^{N}(0)=\tilde{\beta}_{i}^{N}(t)+\frac{1}{(\lambda+1)N}\widetilde{M}_{i}^{N}(t)\,\,;\,\,1\le i\le N\label{eq:zzz}
\end{equation}
where
\begin{align*}
\tilde{\beta}_{i}^{N}(t) & =\frac{N\lambda+1}{N(\lambda+1)}\beta_{i}(t)+\frac{1}{N(\lambda+1)}\sum_{k:k\neq i}\beta_{k}(t)\\
\widetilde{M}_{i}^{N}(t) & =\sum_{k:k\neq i}(M_{ki}^{N}(t)-M_{ik}^{N}(t)).
\end{align*}
For the details, see Proposition 2 of \cite{G1}.

We first develop the exponential tightness of $\left\{ z_{i}^{N}(t)\right\} _{i=1}^{N}$
as an intermediate step, which can be formulated as follows. 
\begin{prop}
\label{prop: tightness of z_i}For any $\epsilon,\,\alpha>0$, 
\[
\limsup_{\eta\rightarrow0}\limsup_{N\rightarrow\infty}\frac{1}{N}\log\mathbb{P}_{N}\left[\left|\left\{ i:\sup_{\substack{0\le t\le\eta}
}\left|z_{i}^{N}(t)-z_{i}^{N}(0)\right|\ge\epsilon\right\} \right|\ge N\alpha\right]=-\infty.
\]
\end{prop}
\begin{proof}
By (\ref{eq:zzz}), it is enough to show that 
\begin{align}
\limsup_{\eta\rightarrow0}\limsup_{N\rightarrow\infty}\frac{1}{N}\log\mathbb{P}_{N}\left[\left|\left\{ i:\sup_{\substack{0\le t\le\eta}
}\left|\tilde{\beta}_{i}^{N}(t)\right|\ge\epsilon\right\} \right|\ge N\alpha\right] & =-\infty\label{eq:brw pt}\\
\limsup_{\eta\rightarrow0}\limsup_{N\rightarrow\infty}\frac{1}{N}\log\mathbb{P}_{N}\left[\left|\left\{ i:\sup_{\substack{0\le t\le\eta}
}\left|\tilde{M}_{i}^{N}(t)\right|\ge N\epsilon\right\} \right|\ge N\alpha\right] & =-\infty\label{eq:jump pt}
\end{align}
respectively. Let us first consider the (\ref{eq:brw pt}). We can
easily bound this as 
\[
\mathbb{P}_{N}\left[\left|\left\{ i:\sup_{\substack{0\le t\le\eta}
}\left|\tilde{\beta}_{i}^{N}(t)\right|\ge\epsilon\right\} \right|\ge N\alpha\right]\le\mathbb{P}_{N}\left[\sum_{i=1}^{N}\sup_{\substack{0\le t\le\eta}
}\left|\beta_{i}(t)\right|\ge N\alpha\epsilon\right].
\]
Since $\left\{ \beta_{i}(t)\right\} _{i=1}^{N}$ are independent Brownian
motions, it is easy to check that this probability is super-exponentially
small.

The next step is (\ref{eq:jump pt}). We define two adapted processes
\[
\zeta{}_{i}^{+}(s)=\mathds{1}_{\sup_{0\le t\le s}\left|\widetilde{M}_{i}^{N}(t)\right|\ge N\epsilon}\,\,\,\mbox{and}\,\,\,\zeta_{i}^{-}(s)=\mathds{1}_{\sup_{0\le t\le s}\left|\widetilde{M}_{i}^{N}(t)\right|<N\epsilon}
\]
and then we can rewrite (\ref{eq:jump pt}) as 
\begin{equation}
\limsup_{\eta\rightarrow0}\limsup_{N\rightarrow\infty}\frac{1}{N}\log\mathbb{P}_{N}\left[\sum_{i=1}^{N}\zeta{}_{i}^{+}(\eta)\ge N\alpha\right]=-\infty.\label{eq:jump pt 2}
\end{equation}
By Chebyshev's inequality it suffices to show that 
\begin{equation}
\limsup_{\eta\rightarrow0}\limsup_{N\rightarrow\infty}\frac{1}{N}\log\mathbb{E}_{N}\exp\left\{ a\sum_{i=1}^{N}\zeta{}_{i}^{+}(\eta)\right\} \le C\label{eq:hsr2}
\end{equation}
for all $a\ge0$ where $C$ is a constant does not depend on $a$.
We shall prove (\ref{eq:hsr2}) with $C=\log2$. 

First we prove 
\begin{equation}
\zeta{}_{i}^{+}(\eta)\le\frac{1}{N\epsilon}\left|\int_{0}^{\eta}\zeta_{i}^{-}(s)d\widetilde{M}_{i}^{N}(s)\right|.\label{eq: hsr1}
\end{equation}
To see this, we only need to concern about the case where $\zeta{}_{i}^{+}(\eta)=1$.
Define 
\[
\eta_{0}=\inf\left\{ t:\left|\widetilde{M}_{i}^{N}(t)\right|\ge N\epsilon\right\} 
\]
and then $\eta_{0}\le\eta$. Thus, 
\[
\left|\int_{0}^{\eta}\zeta_{i}^{-}(s)d\widetilde{M}_{i}^{N}(s)\right|=\left|\int_{0}^{\eta_{0}}1\cdot d\widetilde{M}_{i}^{N}(s)\right|=\left|\widetilde{M}_{i}^{N}(\eta_{0})\right|\ge N\epsilon
\]
due to the right-continuity of the jump process and we proved (\ref{eq: hsr1}).
We return now to (\ref{eq:hsr2}). By (\ref{eq: hsr1}), 
\begin{align}
\mathbb{E}_{N}\exp\left\{ a\sum_{i=1}^{N}\zeta{}_{i}^{+}(\eta)\right\}  & \le\mathbb{E}_{N}\exp\left\{ \frac{a}{N\epsilon}\sum_{i=1}^{N}\left|\int_{0}^{\eta}\zeta_{i}^{-}(s)d\widetilde{M}_{i}^{N}(s)\right|\right\} \nonumber \\
 & \le\sum_{\substack{\mathfrak{e}_{i}=\pm1,\,\forall i}
}\mathbb{E}_{N}\exp\left\{ \frac{a}{N\epsilon}\sum_{i=1}^{N}\int_{0}^{\eta}\mathfrak{e}_{i}\zeta_{i}^{-}(s)d\widetilde{M}_{i}^{N}(s)\right\} .\label{eq: hsr3}
\end{align}
Note that we can rearrange each summand in (\ref{eq: hsr3}) in a
way that 
\begin{align}
 & \mathbb{E}_{N}\exp\left\{ \frac{a}{N\epsilon}\sum_{i=1}^{N}\int_{0}^{\eta}\mathfrak{e}_{i}\zeta_{i}^{-}(s)d\widetilde{M}_{i}^{N}(s)\right\} \nonumber \\
 & =\mathbb{E}_{N}\exp\left\{ \frac{a}{N\epsilon}\sum_{1\le i\neq k\le N}\int_{0}^{\eta}\mathfrak{e}_{i}\zeta_{i}^{-}(s)d[M_{ik}^{N}(s)-M_{ki}^{N}(s)]\right\} \nonumber \\
 & =\mathbb{E}_{N}\exp\left\{ \sum_{1\le i\neq k\le N}\int_{0}^{\eta}\frac{au_{ik}(s)}{N\epsilon}dM_{ik}^{N}(s)\right\} \label{eq:hsr5}
\end{align}
where $u_{ik}(s)=\mathfrak{e}_{i}\zeta_{i}^{-}(s)-\mathfrak{e}_{k}\zeta_{k}^{-}(s)$.
Then, since each $M_{ik}^{N}(t)$ is the compensated Poison process
with rate $\lambda NA_{ik}^{N}(t)$, 
\begin{align*}
\mathbb{E}_{N}\exp\sum_{i,k=1}^{N}\Biggl[ & \int_{0}^{\eta}\frac{2au_{ik}(s)}{N\epsilon}dM_{ik}^{N}(s)\\
 & -\int_{0}^{\eta}\left(\exp\frac{2au_{ik}(s)}{N\epsilon}-\frac{2au_{ik}(s)}{N\epsilon}-1\right)\lambda NdA_{ik}^{N}(s)\Biggl]=1.
\end{align*}
Therefore, by Cauchy-Schwarz's inequality, (\ref{eq:hsr5}) is bounded
by
\[
\mathbb{E}_{N}\left[\exp\sum_{i,k=1}^{N}\int_{0}^{\eta}\left(\exp\frac{2au_{ik}(s)}{N\epsilon}-\frac{2au_{ik}(s)}{N\epsilon}-1\right)\lambda NdA_{ik}^{N}(s)\right]^{\frac{1}{2}}.
\]
For sufficiently large $N$, we can bound the last expression by 
\begin{align*}
 & \mathbb{E}_{N}\left[\exp\sum_{i,k=1}^{N}\int_{0}^{\eta}\left\{ \frac{2au_{ik}(s)}{N\epsilon}\right\} ^{2}\lambda NdA_{ik}^{N}(s)\right]^{\frac{1}{2}}\\
 & \le\mathbb{E}_{N}\left[\exp\sum_{i,k=1}^{N}\int_{0}^{\eta}\frac{16a^{2}\lambda}{N\epsilon^{2}}dA_{ik}^{N}(s)\right]^{\frac{1}{2}}\\
 & =\mathbb{E}_{N}\left[\exp\frac{16a^{2}\lambda}{\epsilon^{2}}NA^{N}(\eta)\right]^{\frac{1}{2}}
\end{align*}
since $\left|u_{ik}(s)\right|\le2$. These series of estimates enable
us to bound 
\[
\mathbb{E}_{N}\exp\left\{ a\sum_{i=1}^{N}\zeta{}_{i}^{+}(\eta)\right\} \le2^{N}\mathbb{E}_{N}\left[\exp\left\{ \frac{16a^{2}\lambda}{\epsilon^{2}}NA^{N}(\eta)\right\} \right]^{\frac{1}{2}}.
\]
Now the proof is completed by Proposition \ref{prop: smal time regime-1}.
\end{proof}
Now we prove the tightness of $\{x_{i}^{N}(t)\}_{N=1}^{\infty}$ by
starting from that of $\{z_{i}^{N}(t)\}_{N=1}^{\infty}$. The methodology
for this step was developed in Proposition 3 of \cite{G1} for a fixed
$i$. The situation here is slightly different but we can still burrow
the core idea. 
\begin{proof}[Proof of Proposition \ref{prop:(Small-time-regime)}]
We start by defining two stopping times $\tau_{i,\epsilon}^{+}$
and $\tau_{i,\epsilon}^{-}$ as
\begin{align*}
\tau_{i,\epsilon}^{+} & =\inf\left\{ t\,:\,x_{i}^{N}(t)-x_{i}^{N}(0)\ge\epsilon\right\} \\
\tau_{i,\epsilon}^{-} & =\inf\left\{ t\,:\,x_{i}^{N}(t)-x_{i}^{N}(0)\le-\epsilon\right\} 
\end{align*}
for each $i$ and then 
\[
\left\{ \sup_{\substack{0\le t\le\eta}
}\left|x_{i}^{N}(t)-x_{i}^{N}(0)\right|\ge\epsilon\right\} =\left\{ \tau_{i,\epsilon}^{+}\le\eta\right\} \cup\left\{ \tau_{i,\epsilon}^{-}\le\eta\right\} .
\]
Thus, it suffices to show 
\begin{equation}
\limsup_{\eta\rightarrow0}\limsup_{N\rightarrow\infty}\frac{1}{N}\log\mathbb{P}_{N}\left[\left|\left\{ i\,:\,\tau_{i,\epsilon}^{+}\le\eta\right\} \right|\ge N\alpha\right]=-\infty\label{eq:csy112}
\end{equation}
since $\tau_{i,\epsilon}^{-}$ can be handled by the exactly same
manner. Let us define
\[
u_{i}^{N}(t)=\frac{1}{N(\lambda+1)}\sum_{j=1}^{N}\nu\left(x_{j}^{N}(t)-x_{i}^{N}(t)\right)
\]
so that $z_{i}^{N}(t)=x_{i}^{N}(t)+u_{i}^{N}(t)$. Then, 
\begin{align*}
 & \left\{ \tau_{i,\epsilon}^{+}\le\eta\right\} \\
 & =\left\{ \tau_{i,\epsilon}^{+}\le\eta,\,\left|z_{i}^{N}(\tau_{i,\epsilon}^{+})-z_{i}^{N}(0)\right|\le\kappa\epsilon\right\} \cup\left\{ \tau_{i,\epsilon}^{+}\le\eta,\,\left|z_{i}^{N}(\tau_{i,\epsilon}^{+})-z_{i}^{N}(0)\right|>\kappa\epsilon\right\} \\
 & \subset\left\{ \tau_{i,\epsilon}^{+}\le\eta,\,\left|\epsilon+u_{i}^{N}(\tau_{i,\epsilon}^{+})-u_{i}^{N}(0)\right|\le\kappa\epsilon\right\} \cup\left\{ \sup_{0\le t\le\eta}\left|z_{i}^{N}(t)-z_{i}^{N}(0)\right|>\kappa\epsilon\right\} 
\end{align*}
for any $\kappa>0$. Note that the second set is super-exponentially
negligible by Proposition \ref{prop: tightness of z_i}. For the first
set, we have 
\begin{align}
 & \left\{ \tau_{i,\epsilon}^{+}\le\eta,\,\left|\epsilon+u_{i}^{N}(\tau_{i,\epsilon}^{+})-u_{i}^{N}(0)\right|\le\kappa\epsilon\right\} \nonumber \\
 & \subset\left\{ \tau_{i,\epsilon}^{+}\le\eta,\,u_{i}^{N}(\tau_{i,\epsilon}^{+})-u_{i}^{N}(0)\le(\kappa-1)\epsilon\right\} .\label{eq:alsli}
\end{align}
Take $\phi_{\epsilon}\in C^{\infty}(\mathbb{T})$ satisfying $\nu(x)\le\phi_{\epsilon}(x)\le\nu(x-\epsilon)+(1+\kappa)\epsilon$
so that\footnote{The existence of such a function is proved in Proposition 4 of \cite{G1}}
\begin{align*}
 & u_{i}^{N}(\tau_{i,\epsilon}^{+})-u_{i}^{N}(0)\\
 & =\frac{1}{N(\lambda+1)}\sum_{j=1}^{N}\left[\nu\left(x_{j}^{N}(\tau_{i,\epsilon}^{+})-x_{i}^{N}(\tau_{i,\epsilon}^{+})\right)-\nu\left(x_{j}^{N}(0)-x_{i}^{N}(0)\right)\right]\\
 & =\frac{1}{N(\lambda+1)}\sum_{j=1}^{N}\left[\nu\left(x_{j}^{N}(\tau_{i,\epsilon}^{+})-x_{i}^{N}(0)-\epsilon\right)-\nu\left(x_{j}^{N}(0)-x_{i}^{N}(0)\right)\right]\\
 & \ge\frac{1}{N(\lambda+1)}\sum_{j=1}^{N}\left[\phi_{\epsilon}\left(x_{j}^{N}(\tau_{i,\epsilon}^{+})-x_{i}^{N}(0)\right)-\phi_{\epsilon}\left(x_{j}^{N}(0)-x_{i}^{N}(0)\right)-(1+\kappa)\epsilon\right].
\end{align*}
Thus, $u_{i}^{N}(\tau_{i,\epsilon}^{+})-u_{i}^{N}(0)\le(\kappa-1)\epsilon$
implies 
\begin{align*}
 & \frac{1}{N(\lambda+1)}\sum_{j=1}^{N}\left[\phi_{\epsilon}(x_{j}^{N}(\tau_{i,\epsilon}^{+})-x_{i}^{N}(0))-\phi_{\epsilon}(x_{j}^{N}(0)-x_{i}^{N}(0))\right]\\
 & \le(\kappa-1)\epsilon+\frac{\kappa+1}{\lambda+1}\epsilon:=-\frac{\gamma}{\lambda+1}\epsilon.
\end{align*}
We choose $\kappa$ small enough so that $\gamma>0$. Then, the RHS
of (\ref{eq:alsli}) is a subset of 
\begin{equation}
\left\{ \sup_{0\le t\le\eta}\left|\frac{1}{N}\sum_{j=1}^{N}\left[\phi_{\epsilon}\left(x_{j}^{N}(t)-x_{i}^{N}(0)\right)-\phi_{\epsilon}\left(x_{j}^{N}(0)-x_{i}^{N}(0)\right)\right]\right|\ge\gamma\epsilon\right\} .\label{eq:csy312}
\end{equation}
By Ito's formula,
\begin{align*}
 & \left|\frac{1}{N}\sum_{j=1}^{N}\left[\phi_{\epsilon}\left(x_{j}^{N}(t)-x_{i}^{N}(0)\right)-\phi_{\epsilon}\left(x_{j}^{N}(0)-x_{i}^{N}(0)\right)\right]\right|\\
 & \le\left|\frac{1}{2N}\sum_{j=1}^{N}\int_{0}^{t}\phi_{\epsilon}^{''}\left(x_{j}^{N}(s)-x_{i}^{N}(0)\right)ds\right|+\left|\frac{1}{N}\sum_{j=1}^{N}\int_{0}^{t}\phi_{\epsilon}^{'}\left(x_{j}^{N}(s)-x_{i}^{N}(0)\right)d\beta_{j}(s)\right|\\
 & \le\frac{\eta}{2}\left\Vert \phi_{\epsilon}^{''}\right\Vert _{\infty}+\left|\frac{1}{N}\sum_{i=1}^{N}\int_{0}^{t}\phi_{\epsilon}^{'}\left(x_{j}^{N}(s)-x_{i}^{N}(0)\right)d\beta_{j}(s)\right|
\end{align*}
for $t\le\eta.$ Note that the local time does not appear since the
expression (\ref{eq:csy312}) is totally symmetric with the function
$\phi_{\epsilon}(\cdot-x_{i}^{N}(0))$. Thus, (\ref{eq:csy312}) is
a subset of 
\begin{equation}
\left\{ \sup_{0\le t\le\eta}\left|\frac{1}{\sqrt{N}}\sum_{j=1}^{N}\int_{0}^{t}\phi_{\epsilon}^{'}\left(x_{j}^{N}(s)-x_{i}^{N}(0)\right)d\beta_{j}(s)\right|>\frac{\gamma\epsilon}{2}\sqrt{N}\right\} \label{eq:csy341}
\end{equation}
for sufficiently small $\eta$. We now regard $\frac{1}{\sqrt{N}}\sum_{j=1}^{N}\int_{0}^{t}\phi_{\epsilon}^{'}\left(x_{j}^{N}(s)-x_{i}^{N}(0)\right)d\beta_{j}(s)$
as a time change of Brownian motion 
\[
B_{i}\left(\frac{1}{N}\sum_{j=1}^{N}\int_{0}^{t}\phi_{\epsilon}^{'}(x_{j}^{N}(s)-x_{i}^{N}(0))^{2}ds\right)
\]
where $B_{i}(\cdot)$ is a Brownian motion starting from $0$ under
$\mathbb{P}_{N}$. Since 
\[
\frac{1}{N}\sum_{j=1}^{N}\int_{0}^{t}\phi_{\epsilon}^{'}(x_{j}^{N}(s)-x_{i}^{N}(0))^{2}ds\le\left\Vert \phi_{\epsilon}^{'}\right\Vert _{\infty}^{2}\eta:=C_{\epsilon}\eta
\]
the event (\ref{eq:csy341}) is a subset of $\left\{ \sup_{0\le t\le C_{\epsilon}\eta}\left|B_{i}(t)\right|>\frac{\gamma\epsilon}{2}\sqrt{N}\right\} .$

Finally, it suffices to show 
\[
\limsup_{\eta\rightarrow0}\limsup_{N\rightarrow\infty}\frac{1}{N}\log\mathbb{P}_{N}\left[\left|\left\{ i:\sup_{0\le t\le C_{\epsilon}\eta}\left|B_{i}(t)\right|>\frac{\gamma\epsilon}{2}\sqrt{N}\right\} \right|\ge N\alpha\right]=-\infty
\]
to complete the proof. However, this is obvious since we have a trivial
bound 
\begin{align*}
 & \mathbb{P}_{N}\left[\left|\left\{ i:\sup_{0\le t\le C_{\epsilon}\eta}\left|B_{i}(t)\right|>\frac{\gamma\epsilon}{2}\sqrt{N}\right\} \right|\ge N\alpha\right]\\
 & \le\sum_{i=1}^{N}\mathbb{P}_{N}\left[\sup_{0\le t\le C_{\epsilon}\eta}\left|B_{i}(t)\right|>\frac{\gamma\epsilon}{2}\sqrt{N}\right]
\end{align*}
for large enough $N$, and then by a property of the Brownian motion,
\[
\mathbb{P}_{N}\left[\sup_{0\le t\le C_{\epsilon}\eta}\left|B_{i}(t)\right|>\frac{\gamma\epsilon}{2}\sqrt{N}\right]\le\frac{8\sqrt{C_{\epsilon}\eta}}{\gamma\epsilon\sqrt{2\pi N}}\exp\left\{ -\frac{\left(\frac{\gamma\epsilon}{2}\sqrt{N}\right)^{2}}{2C_{\epsilon}\eta}\right\} .
\]

\end{proof}

\section{Diffusion of Colors}

\subsection{Introduction\label{sub:Introduction} }

In this section, we develop the LDP for the empirical density of colors.
We recall from Section \ref{sub:Colored-Process} that the empirical
density for colors $\{\tilde{\mu}^{N}(t):0\le t\le T\}$ is defined
as (\ref{eq:emp color}) which can be regarded as a Markov process
$\widetilde{\mathbb{Q}}_{N}$ on $C([0,\,T],\,\mathscr{M}(\mathbb{T})^{m})$.
First, we state the hydrodynamical limit theory of $\{\widetilde{\mathbb{Q}}_{N}\}_{N=1}^{\infty}$. 
\begin{thm}
\label{thm: HDL collor}Suppose that Assumption 2 is satisfied with
the uncolored initial measure $\rho^{0}(x)dx$ for some bounded $\rho^{0}(\cdot)$.
Then $\{\widetilde{\mathbb{Q}}_{N}\}_{N=1}^{\infty}$ converges weakly
to $\widetilde{\mathbb{Q}}_{\infty}$ which is a Dirac mass concentrating
on the single-valued trajectory\textup{
\[
\{\tilde{\rho}(t,\,x)dx=(\rho_{1}(t,\,x)dx,\,\rho_{2}(t,\,x)dx,\,\cdots,\,\rho_{m}(t,\,x)dx)^{\dagger}:\,0\le t\le T\}
\]
}where $\tilde{\rho}(t,\,x)$ is the unique weak solution of the partial
differential equation\textup{
\begin{equation}
\frac{\partial\tilde{\rho}}{\partial t}=\frac{1}{2}\nabla\cdot\left[D(\tilde{\rho})\nabla\tilde{\rho}\right]\label{eq:diffusion of color}
\end{equation}
}with initial condition $\tilde{\rho}^{0}(x)dx$ where the diffusion
matrix $D(\tilde{\rho})$ is given by (\ref{eq:diffusion mat}). \end{thm}
\begin{rem}
An extension of this hydrodynamic limit result is obtained in \cite{IS}
which consider a system of two mechanically different types of particles.
\end{rem}
As we already observed in (\ref{eq:single color PDE}) and (\ref{eq: lln color}),
each component of equation (\ref{eq:diffusion of color}) can be written
as 
\begin{equation}
\partial_{t}\rho_{c}=\frac{1}{2}\nabla\left[\frac{\lambda}{\lambda+\rho}\nabla\rho_{c}+\frac{\nabla\rho}{\lambda+\rho}\rho_{c}\right]\,\,;\,\,c=1,\,2,\,\cdots,\,m.\label{eq:sce}
\end{equation}
where $\rho=\sum\rho_{c}$. Since $\rho$ is the solution of the heat
equation $\partial_{t}\rho=\frac{1}{2}\Delta\rho$ with the initial
condition $\rho^{0}(dx)=\sum_{c=1}^{m}\rho_{c}^{0}(dx)$, (\ref{eq:sce})
is just a usual linear parabolic equation with smooth coefficients.
Thus, the uniqueness of the solution is immediate\footnote{This uniqueness is also a direct consequence of Theorem \ref{thm: Uniqueness Theorem } }
and consequently Theorem \ref{thm: HDL collor} is a corollary of
Theorem \ref{thm : HDL for Driven System}. The final comment regarding
Theorem 4.1 is that we did not need assumptions on $f_{N}^{0}(dx)$
more than what is stated in Assumption 2 because our estimates in
Sections 2 and 3 did not impose any further conditions due to our
careful analysis on the small time regime. 

The next step is the large deviation theory for $\{\widetilde{\mathbb{Q}}_{N}\}_{N=1}^{\infty}$.
In order to concentrate on the large deviation of dynamic evolution
and simplify the arguments regarding the initial deviation, we assume
the following throughout Sections 4 and 5 in addition to Assumption
2. \\
\\
\textbf{Assumption 3. }\textit{The initial configuration of particles
is i.i.d. with a bounded probability density function $\rho^{0}(x)$
on $\mathbb{T}$. }
\begin{rem}
Reading our proof carefully reveals that the LDP is still valid under
many general initial configurations, e.g., deterministic configuration.
However, the boundedness assumption on $\rho^{0}(x)$ is essential,
especially when we establish the compactness property of the rate
function in Lemma \ref{prop: Estimates based on bound on rate function}.
\end{rem}
We now state the LDP for $\{\widetilde{\mathbb{Q}}_{N}\}_{N=1}^{\infty}$
under Assumptions 2 and 3. First of all, for each color $c$, $\mathscr{M}_{c}^{0}(\mathbb{T})\subset\mathscr{M}(\mathbb{T})$
is defined by 
\[
\mathscr{M}_{c}^{0}(\mathbb{T})=\left\{ r(x)dx:\int_{\mathbb{T}}r(x)dx=\bar{\rho}_{c},\,r(x)\ge0\right\} 
\]
where $\bar{\rho}_{c}$ is given by (\ref{eq:aver col c}). Then we
will show in Section 4.2 that the LDP rate is infinite outside $C([0,\,T],\,\mathscr{M}_{color}^{0}(\mathbb{T}))$
where 
\[
\mathscr{M}_{color}^{0}(\mathbb{T})=\prod_{c=1}^{m}\mathscr{M}_{c}^{0}(\mathbb{T}).
\]
In the domain $C([0,\,T],\,\mathscr{M}_{color}^{0}(\mathbb{T}))$,
the rate function is $I_{color}^{m}(\cdot)=I_{init}^{m}(\cdot)+I_{dyn}^{m}(\cdot)$
where the dynamic rate function $I_{dyn}^{m}(\cdot)$ is defined by
\begin{equation}
I_{dyn}^{m}(\tilde{\rho}(\cdot,\,x)dx)=\frac{1}{2}\int_{0}^{T}\left\Vert \frac{\partial\tilde{\rho}}{\partial t}-\frac{1}{2}\nabla\cdot\left[D(\tilde{\rho})\nabla\tilde{\rho}\right]\right\Vert _{-1,\,A(\tilde{\rho})}^{2}dt.\label{eq: dynamic rate function}
\end{equation}
for $\tilde{\rho}(\cdot,\,x)dx\in C([0,\,T],\,\mathscr{M}_{color}^{0}(\mathbb{T}))$
where $A(\tilde{\rho})$ is defined by (\ref{eq:A and chi}). Here,
it is necessary to clarify exactly what is meant by the RHS of (\ref{eq: dynamic rate function}).
The $H_{-1,\,A}$ norm can be explained by the variational formula
(\textit{cf.} (2.24) of \cite{QRV}) 
\begin{align}
\sup_{\phi}\Biggl\{ & \int_{\mathbb{T}}\phi^{\dagger}\tilde{\rho}(T,\,x)dx-\int_{\mathbb{T}}\phi^{\dagger}\tilde{\rho}(0,\,x)dx\label{eq: dynamic rate function _var form}\\
 & +\int_{0}^{T}\int_{\mathbb{T}}\left[-\frac{\partial\phi}{\partial t}^{\dagger}\tilde{\rho}+\frac{1}{2}\nabla\phi^{\dagger}D(\tilde{\rho})\nabla\tilde{\rho}-\frac{1}{2}\nabla\phi^{\dagger}A(\tilde{\rho})\nabla\phi\right](t,\,x)dxdt\Biggl\}\nonumber 
\end{align}
where the supremum is taken over $\phi\in C^{\infty}([0,\,T]\times\mathbb{T}^{m})$.
However, the last expression is still not well-defined as it stands
because it involves $\nabla\tilde{\rho},$ which might not exist.
Thus, our starting point should be to obtain a reasonable explanation
of (\ref{eq: dynamic rate function _var form}) and the basic properties
of this rate function, e.g., the compactness and lower semicontinuity.
Section \ref{sub:Rate-Function1} is devoted to this project. Then,
we will establish the large deviation upper and lower bounds in Sections
\ref{sub:Upper-Bound} and \ref{sub:Lower-Bound}, respectively.

\subsection{Rate Function\label{sub:Rate-Function1}}

\subsubsection{Well-definedness of Variational Formula (\ref{eq: dynamic rate function _var form})}

To define the rate function in the sense of (\ref{eq: dynamic rate function _var form}),
we need some \textit{a priori} regularity result for $\tilde{\rho}(\cdot,\,\cdot)$,
as well as energy estimates on the domain of the rate function. Recall
here that $\tilde{\rho}=(\rho_{1},\,\rho_{2},\,\cdots,\,\rho_{m})^{\dagger}$
is the $m$-dimensional vector of the density of colors and $\rho=\sum_{c=1}^{m}\rho_{c}$
denotes the total density. To begin with, let us define a set $\mathscr{D}_{color}^{m}\subset C([0,\,T],\,\mathscr{M}_{color}^{0}(\mathbb{T}))$
that consists of $\tilde{\rho}(t,\,x)dx$ which is weakly differentiable
in $x$ with the energy estimate 
\begin{equation}
\int_{0}^{T}\int_{\mathbb{T}}\left(\nabla\tilde{\rho}^{\dagger}\chi A\chi(\tilde{\rho})\nabla\tilde{\rho}\right)(t,\,x)dxdt<\infty\label{eq: aprori est}
\end{equation}
and satisfies the finite initial entropy condition 
\begin{equation}
\int_{\mathbb{T}}\rho(0,\,x)\log\rho(0,\,x)dx<\infty.\label{eq:finiteness of entropy}
\end{equation}
Note here that we can compute $\nabla\tilde{\rho}^{\dagger}\chi A\chi(\tilde{\rho})\nabla\tilde{\rho}$
explicitly as
\begin{equation}
\nabla\tilde{\rho}^{\dagger}\chi A\chi(\tilde{\rho})\nabla\tilde{\rho}=\frac{(\nabla\rho)^{2}}{\lambda+\rho}+\sum_{c=1}^{m}\frac{(\nabla\rho_{c})^{2}}{(\lambda+\rho)\rho_{c}}\label{eq: explicity form of chi A chi}
\end{equation}
and accordingly, (\ref{eq: aprori est}) is equivalent to 
\[
\begin{cases}
\hat{I}(\rho)=\int_{0}^{T}\int_{\mathbb{T}}\frac{(\nabla\rho)^{2}}{\rho}(t,\,x)dxdt<\infty & \mbox{and},\\
\hat{I}_{c}(\tilde{\rho})=\int_{0}^{T}\int_{\mathbb{T}}\frac{(\nabla\rho_{c})^{2}}{(\lambda+\rho)\rho_{c}}(t,\,x)dxdt<\infty & 1\le c\le m.
\end{cases}
\]
In particular, the finiteness of $\hat{I}(\rho)$ implied by (\ref{eq: aprori est})
implies that $\rho\in L_{2}([0,\,T]\times\mathbb{T})$ or equivalently
$\tilde{\rho}\in L_{2}([0,\,T]\times\mathbb{T}^{m})$ as follows.
\begin{lem}
\label{lem: L_2 boundness-1}Suppose that $\rho$ is weakly differentiable
and satisfies $\hat{I}(\rho)<\infty$. Then $\rho\in L_{2}(0,\,T,\,L_{\infty}(\mathbb{T}))$
and the $L_{2}$ norm is bounded by $2(\hat{I}(\rho)+T)$. In particular,
$\tilde{\rho}(\cdot,\,x)dx\in\mathscr{D}_{color}^{m}$ implies that
$\tilde{\rho}\in L_{2}([0,\,T]\times\mathbb{T}^{m})$.\end{lem}
\begin{proof}
Let $\phi_{\epsilon}$ be the heat kernel on $\mathbb{T}$ at time
$\epsilon^{2}$ and $\rho_{\epsilon}=\rho*\phi_{\epsilon}$. Then,
\[
\rho_{\epsilon}(t,\,x)-\rho_{\epsilon}(t,\,y)=\int_{[x,y]}\nabla\rho_{\epsilon}(t,\,z)dz\le\sqrt{\int_{\mathbb{T}}\frac{(\nabla\rho_{\epsilon})^{2}}{\rho_{\epsilon}}(t,\,z)dz}
\]
by Cauchy-Schwarz's inequality. Integrating the above expression against
$y$ give us 
\begin{align*}
\sup_{x\in\mathbb{T}}\rho_{\epsilon}(t,\,x) & \le\sqrt{\int_{\mathbb{T}}\frac{(\nabla\rho_{\epsilon})^{2}}{\rho_{\epsilon}}(t,\,z)dz}+1\le\sqrt{\int_{\mathbb{T}}\left(\frac{(\nabla\rho)^{2}}{\rho}\right)_{\epsilon}(t,\,z)dz}+1\\
 & =\sqrt{\int_{\mathbb{T}}\frac{(\nabla\rho)^{2}}{\rho}(t,\,z)dz}+1.
\end{align*}
This implies $\int_{0}^{T}\sup_{x\in\mathbb{T}}\rho_{\epsilon}^{2}(t,\,x)\le2(\hat{I}(\rho)+T)$
and we can obtain the desired result by taking $\epsilon\rightarrow0$. \end{proof}
\begin{rem}
\label{RMK4.3}Henceforth, for any function $f$ on $\mathbb{T}$,
$f_{\epsilon}$ denotes $f*\phi_{\epsilon}$ where $\phi_{\epsilon}$
is the heat kernel on $\mathbb{T}$ at time $\epsilon^{2}$. 
\end{rem}
The next two lemmas prove that the domain of the rate function is
included in $\mathscr{D}_{color}^{m}$.
\begin{lem}
For any $G\in C^{0,1}([0,\,T]\times\mathbb{T})$,
\begin{equation}
\limsup_{N\rightarrow\infty}\frac{1}{N}\log\mathbb{E}_{N}\exp\left\{ \int_{0}^{T}\sum_{i=1}^{N}\left[\nabla G(t,\,x_{i}^{N}(t))-2G^{2}(t,\,x_{i}^{N}(t))\right]dt\right\} \le0\label{eq: apriori energy estimate 1}
\end{equation}
and for each color $c$,
\begin{align}
 & \limsup_{\epsilon\rightarrow0}\limsup_{N\rightarrow\infty}\label{eq: apriori energy estimate 2}\\
 & \frac{1}{N}\log\mathbb{E}_{N}\exp\left\{ \int_{0}^{T}\sum_{i\in I_{c}^{N}}\left[\nabla G(t,\,x_{i}^{N}(t))-12G^{2}(t,\,x_{i}^{N}(t))\left(1+\frac{\rho_{i,\epsilon}(x^{N}(t))}{\lambda}\right)\right]dt\right\} \nonumber \\
 & \le m\left\Vert \rho_{0}\right\Vert _{\infty}.\nonumber 
\end{align}
\end{lem}
\begin{proof}
For (\ref{eq: apriori energy estimate 1}), the expression is symmetric
and thus we can consider the model as non-interacting case. Then,
\begin{align}
 & \log\mathbb{E}_{N}\left[\exp\left\{ \int_{0}^{T}\sum_{i=1}^{N}\left[\nabla G(t,\,x_{i}^{N}(t))-2G^{2}(t,\,x_{i}^{N}(t))\right]dt\right\} \right]\nonumber \\
 & =\sum_{i=1}^{N}\log\mathbb{E}_{N}\left[\exp\left\{ \int_{0}^{T}\left[\nabla G(t,\,x_{i}^{N}(t))-2G^{2}(t,\,x_{i}^{N}(t))\right]dt\right\} \right].\label{eq:bmtm}
\end{align}
If $\beta$ is a standard Brownian motion under $P$, then by  Feynman-Kac
formula and the variational formula for the largest eigenvalue,
\begin{align*}
 & \log\mathbb{E}^{P}\left[\exp\left\{ \int_{0}^{T}\left[\nabla G(t,\,\beta(t))-2G^{2}(t,\,\beta(t))\right]dt\right\} \right]\\
 & \le\int_{0}^{T}\sup_{h\in C^{\infty}(\mathbb{T}),\,\int_{\mathbb{T}}h(y)dy=1,\,h\ge0}\left\{ \int_{\mathbb{T}}h(y)(\nabla G-2G^{2})(t,\,y)-\frac{1}{8}\frac{\left(\nabla h(y)\right)^{2}}{h(y)}dy\right\} dt\\
 & =\int_{0}^{T}\sup_{h\in C^{\infty}(\mathbb{T}),\,\int_{\mathbb{T}}h(y)dy=1,\,h\ge0}\left\{ -\frac{1}{8}\int_{\mathbb{T}}h(y)\left[4G(t,\,y)+\frac{\nabla h(y)}{h(y)}\right]^{2}dy\right\} dt\\
 & \le0.
\end{align*}
Therefore, (\ref{eq:bmtm}) is non-positive.

For (\ref{eq: apriori energy estimate 2}), we should take the interaction
into account and this requires us to consider additional $\rho_{i,\epsilon}(x^{N}(t))$
part. In the spirit of Lemma \ref{lem: symmetrization lemma}, we
can replace $\mathbb{E}_{N}$ in (\ref{eq: apriori energy estimate 2})
by $\bar{\mathbb{E}}_{N}^{color}$. Then, the estimate with $\bar{\mathbb{E}}_{N}^{color}$
is equivalent to the one with $\mathbb{E}_{N}^{eq}$ since 
\[
\left\Vert \frac{d\mathbb{P}_{N}^{eq}}{d\bar{\mathbb{P}}_{N}^{color}}\right\Vert _{\infty}\le\left\Vert \frac{d\bar{\mathbb{P}}_{N}}{d\bar{\mathbb{P}}_{N}^{color}}\right\Vert _{\infty}\left\Vert \frac{d\mathbb{P}_{N}^{eq}}{d\bar{\mathbb{P}}_{N}}\right\Vert _{\infty}\le m^{N}\left\Vert \rho_{0}\right\Vert _{\infty}^{N}
\]
by Assumption 3. Consequently, we can substitute $\mathbb{E}_{N}$
in (\ref{eq: apriori energy estimate 2}) by $\mathbb{E}_{N}^{eq}$
and the price of this substitution is $m\left\Vert \rho_{0}\right\Vert _{\infty}$.
For the equilibrium estimate, by the standard argument as before,
\begin{align*}
\frac{1}{N}\log\mathbb{E}_{N}^{eq}\exp\left\{ \int_{0}^{T}\sum_{i\in I_{c}^{N}}\left[\nabla G(t,\,x_{i}^{N}(t))-12G^{2}(t,\,x_{i}^{N}(t))\left(1+\frac{\rho_{i,\epsilon}(x^{N}(t))}{\lambda}\right)\right]dt\right\} \\
\le\int_{0}^{T}\sup_{f\in\mathscr{P}_{N}}\Biggl\{\frac{1}{N}\sum_{i\in I_{c}^{N}}\int_{G_{N}}f(x)\left[\nabla G(t,\,x_{i})-12G^{2}(t,\,x_{i})\left(1+\frac{\rho_{i,\epsilon}(x))}{\lambda}\right)\right]dx\\
-\frac{\mathcal{D}_{N}(f)}{N}\Biggl\} & .
\end{align*}
Now we apply Green's formula (\ref{eq:green}) with the vector field
$\sum_{i\in I_{c}^{N}}G(t,\,x_{i})e_{i}$ such that 
\begin{equation}
\sum_{i\in I_{c}^{N}}\int_{G_{N}}f(x)\nabla G(t,\,x_{i})dx=U_{1}+U_{2}\label{eq:U10}
\end{equation}
where
\begin{align*}
U_{1} & =\sum_{i\in I_{c}^{N}}\int_{G_{N}}\nabla_{i}f(x)G(t,\,x_{i})dx\\
U_{2} & =\sum_{i\in I_{c}^{N}}\sum_{k:k\neq i}\int_{F_{ik}}G(t,\,x_{i})(f_{ki}(x)-f_{ik}(x))dS_{ik}(x).
\end{align*}
We can estimate $U_{1}$ as
\begin{align}
U_{1} & \le\sum_{i\in I_{c}^{N}}\int_{G_{N}}\left[12f(x)G^{2}(t,\,x_{i})+\frac{1}{48}\frac{(\nabla_{i}f(x))^{2}}{f(x)}\right]dx\label{eq:U11}\\
 & \le12\sum_{i\in I_{c}^{N}}\int_{G_{N}}f(x)G^{2}(t,\,x_{i})dx+\frac{1}{6}\mathcal{D}_{N}(f).\nonumber 
\end{align}
For $U_{2}$, we can apply Proposition \ref{Main Replacement Lemma}
such that 
\begin{align}
 & U_{2}\label{eq: U22}\\
 & \le\sum_{i\in I_{c}^{N},\,k\neq i}\int_{F_{ik}}\frac{3}{\lambda N}(\sqrt{f_{ik}}+\sqrt{f_{ki}})^{2}(x)G^{2}(t,\,x_{i})+\frac{\lambda N}{12}(\sqrt{f_{ik}}-\sqrt{f_{ki}})^{2}(x)dS_{ik}(x)\nonumber \\
 & \le\frac{6}{\lambda N}\sum_{i\in I_{c}^{N},\,k\neq i}\int_{F_{ik}}(f_{ik}+f_{ki})(x)G^{2}(t,\,x_{i})dS_{ik}(x)+\frac{1}{6}\mathcal{D}_{N}(f)\nonumber \\
 & \le\frac{12}{\lambda}\sum_{i\in I_{c}^{N}}\int_{G_{N}}f(x)\rho_{\epsilon,i}(x)G^{2}(t,\,x_{i})dx+CN\epsilon^{\frac{1}{4}}\left(1+\left(\frac{\mathcal{D}_{N}(f)}{N}\right)^{\frac{7}{8}}\right)+\frac{1}{6}\mathcal{D}_{N}(f)\nonumber 
\end{align}
where $C$ is a constant only depends on $G$. Thus, we have by (\ref{eq:U11})
and (\ref{eq: U22}) that 
\begin{align*}
 & \frac{1}{N}\sum_{i\in I_{c}^{N}}\int_{G_{N}}f(x)\left[\nabla G(t,\,x_{i})-12G^{2}(t,\,x_{i})\left(1+\frac{\rho_{i,\epsilon}(x)}{\lambda}\right)\right]dx-\frac{\mathcal{D}_{N}(f)}{N}\\
 & \le C\epsilon^{\frac{1}{4}}\left(1+\left(\frac{\mathcal{D}_{N}(f)}{N}\right)^{\frac{7}{8}}\right)-\frac{2}{3}\frac{\mathcal{D}_{N}(f)}{N}\le C'(\epsilon^{\frac{1}{4}}+\epsilon^{2}).
\end{align*}
Thus, the proof is completed. 
\end{proof}
Based on the previous lemma, we can restrict the domain of rate function
to $\mathscr{D}_{color}^{m}$. Note that $\{\widetilde{\mathbb{Q}}_{N}\}_{N=1}^{\infty}$
satisfies the LDP due to Bryc's inverse Varadhan Lemma(\textit{cf.}
Theorem 4.4.2 of \cite{DZ}) because of the exponential tightness
result of Section 3. Hence, we will temporarily denote the rate function
by $\bar{I}^{m}(\cdot)$ in the next lemma, since we do not know exact
form of the rate function at this stage. 
\begin{lem}
\label{lem: Domain of Rate function}If $\bar{I}^{m}(\tilde{\mu}(\cdot))<\infty$,
then $\tilde{\mu}(\cdot)=(\mu_{1},\,\mu_{2},\,\cdots,\,\mu_{m})^{\dagger}\in\mathscr{D}_{color}^{m}$. \end{lem}
\begin{proof}
First of all, we know from the Theorem 3.1 of \cite{KO} that $\bar{I}^{m}(\tilde{\mu}(\cdot))<\infty$
only if $\tilde{\mu}(t)$ is absolute continuous with respect to the
Lebesgue measure for all $t$. Let us define a functional $\Xi_{G}$
on $C([0,\,T],\,\mathscr{M}(\mathbb{\mathbb{T}}))$ by 
\begin{equation}
\Xi_{G}(\tilde{\mu}(\cdot))=\int_{0}^{T}dt\int_{\mathbb{T}}\left[\nabla G(t,\,x)-2G^{2}(t,\,x)\right]\mu(t,\,dx)\label{eq: JJJY1}
\end{equation}
where $G\in C^{0,1}([0,\,T]\times\mathbb{T})$ and $\mu=\sum_{c=1}^{m}\mu_{c}$.
Then (\ref{eq: apriori energy estimate 1}) can be rewritten as 
\[
\limsup_{N\rightarrow\infty}\frac{1}{N}\log\mathbb{E}^{\widetilde{\mathbb{Q}}_{N}}\left[\exp\left\{ N\Xi_{G}(\tilde{\mu}(\cdot))\right\} \right]\le0
\]
and therefore by Varadhan's Lemma, $\Xi_{G}(\tilde{\mu}(\cdot))\le\bar{I}^{m}(\tilde{\mu}(\cdot))$.
In particular, if $\tilde{\mu}(\cdot)=\tilde{\rho}(\cdot,\,x)dx$
satisfies $\bar{I}^{m}(\tilde{\mu}(\cdot))<\infty$, then we have
\[
\int_{0}^{T}\int_{\mathbb{T}}\nabla G(t,\,x)\rho(t,\,x)dxdt\le\bar{I}^{m}(\tilde{\mu}(\cdot))+2\int_{0}^{T}\int_{\mathbb{T}}G^{2}(t,\,x)\rho(t,\,x)dxdt.
\]
and therefore 
\begin{equation}
\int_{0}^{T}\int_{\mathbb{T}}\nabla G(t,\,x)\rho(t,\,x)dxdt\le C\sqrt{\int_{0}^{T}\int_{\mathbb{T}}G^{2}(t,\,x)\rho(t,\,x)dxdt}\label{eq: Boundedness 12}
\end{equation}
for some constant $C\ge0$ which does not depend on $G$. If we define
the inner product $\left\langle \cdot\,,\,\cdot\right\rangle _{\rho}$
on $C([0,\,T]\times\mathbb{T})$ by 
\[
\left\langle F_{1},\,F_{2}\right\rangle _{\rho}=\int_{0}^{T}\int_{\mathbb{T}}F_{1}F_{2}\rho dxdt
\]
and let $L_{\rho}^{2}([0,\,T]\times\mathbb{T})$ be the Hilbert space
by taking completion and equivalent class. Then (\ref{eq: Boundedness 12})
implies that the functional $l(G)=\int_{0}^{T}\int_{\mathbb{T}}\rho\nabla Gdxdt$
is a bounded linear functional on $C^{0,1}([0,\,T]\times\mathbb{T})\subset L_{\rho}^{2}([0,\,T]\times\mathbb{T}).$
By Hahn-Banach's theorem, we can extend $l(\cdot)$ to $L_{\rho}^{2}$
and then Riesz representation theorem gives us a function $H\in L_{\rho}^{2}$
such that $l(G)=\left\langle G,\,H\right\rangle _{\rho}.$ Therefore,
$\rho$ is weakly differentiable with $\nabla\rho:=-H\rho$. Moreover,
since $H\in L_{\rho}^{2}([0,\,T]\times\mathbb{T})$ we obtain
\[
\int_{0}^{T}\int_{\mathbb{T}}H^{2}(t,\,x)\rho(t,\,x)dxdt=\int_{0}^{T}\int_{\mathbb{T}}\frac{(\nabla\rho)^{2}}{\rho}dxdt<\infty.
\]
This proves the finiteness of $\hat{I}(\rho)$. 

Likewise, we can derive from (\ref{eq: apriori energy estimate 2})
that 
\[
\limsup_{\epsilon\rightarrow0}\int_{0}^{T}\int_{\mathbb{T}}\left(\nabla G-6G^{2}\left(1+\frac{\rho*\iota_{\epsilon}}{\lambda}\right)\right)\rho_{c}dxdt\le\bar{I}^{m}(\tilde{\rho})+m||\rho_{0}||_{\infty}.
\]
Since $\rho\in L_{2}([0,\,T]\times\mathbb{T})$ by the finiteness
of $\hat{I}(\rho)$ and Lemma (\ref{lem: L_2 boundness-1}), the LHS
is 
\[
\int_{0}^{T}\int_{\mathbb{T}}\left(\nabla G-\frac{6(\lambda+\rho)G^{2}}{\lambda}\right)\rho_{c}dxdt
\]
and the RHS is independent with $G$. Therefore, we can repeat the
previous argument to prove that $\rho_{c}$ is weakly differentiable
and $\hat{I}_{c}(\tilde{\rho})<\infty$.

Finally, if $\bar{I}^{m}(\tilde{\mu})<\infty$ then by Sanov's theorem
$\int_{\mathbb{T}}\rho(0,\,x)\log\frac{\rho(0,\,x)}{\rho_{0}(x)}dx<\infty$
and this implies the finiteness of the entropy (\ref{eq:finiteness of entropy}).
This finishes the proof. 
\end{proof}
Now, we are ready to explain the dynamic rate function in the sense
of (\ref{eq: dynamic rate function _var form}) for $\tilde{\rho}(t,\,x)dx\in\mathscr{D}_{color}^{m}$.
The only part which is not well-defined in (\ref{eq: dynamic rate function _var form})
is 
\[
\int_{0}^{T}\int_{\mathbb{T}}\nabla\phi^{\dagger}D(\tilde{\rho})\nabla\tilde{\rho}dxdt.
\]
First observe that the $c$th element of $D(\tilde{\rho})\nabla\tilde{\rho}$
is $\frac{\lambda}{\lambda+\rho}\nabla\rho_{c}+\frac{\rho_{c}}{\lambda+\rho}\nabla\rho$,
and hence it is enough to show the finiteness of 
\[
\int_{0}^{T}\int_{\mathbb{T}}\left|\frac{\nabla\rho_{c}}{\lambda+\rho}\right|dxdt\text{ \,\ and\, }\int_{0}^{T}\int_{\mathbb{T}}\left|\frac{\rho_{c}}{\lambda+\rho}\nabla\rho\right|dxdt
\]
for $\tilde{\rho}(t,\,x)dx\in\mathscr{D}_{color}^{m}$. The first
one is bounded by
\[
\left[\int_{0}^{T}\int_{\mathbb{T}}\frac{(\nabla\rho_{c})^{2}}{(\lambda+\rho)\rho_{c}}dxdt\,\int_{0}^{T}\int_{\mathbb{T}}\frac{\rho_{c}}{\lambda+\rho}dxdt\right]^{\frac{1}{2}}
\]
and therefore finite for $\tilde{\rho}(\cdot,\,x)dx\in\mathscr{D}_{color}^{m}$.
The second one is bounded by
\[
\left[\int_{0}^{T}\int_{\mathbb{T}}\frac{(\nabla\rho)^{2}}{\rho}dxdt\,\int_{0}^{T}\int_{\mathbb{T}}\frac{\rho_{c}^{2}\rho}{(\lambda+\rho)^{2}}dxdt\right]^{\frac{1}{2}}
\]
which is also finite since $\frac{\rho_{c}^{2}\rho}{(\lambda+\rho)^{2}}\le\rho$.
Therefore, for $\tilde{\mu}(\cdot)=$$\tilde{\rho}(\cdot,\,x)dx\in\mathscr{D}_{color}^{m}$,
we can define $I_{dyn}^{m}(\tilde{\rho}(\cdot,\,x)dx)$ through the
variational formula (\ref{eq: dynamic rate function _var form}).
We finally set $I_{dyn}^{m}(\tilde{\mu}(\cdot))=\infty$ for $\tilde{\mu}(\cdot)\notin\mathscr{D}_{color}^{m}$. 
\begin{rem}
Henceforth, we write $\tilde{\rho}\in\mathscr{D}_{color}^{m}$ and
 $I_{dyn}^{m}(\tilde{\rho})$, instead of $\tilde{\rho}(\cdot,\,x)dx\in\mathscr{D}_{color}^{m}$
and $I_{dyn}^{m}(\tilde{\rho}(\cdot,\,x)dx)$, respectively, for simplicity. 
\end{rem}

\subsubsection{Lower Semicontinuity}

The next step is to establish the lower semicontinuity of the functional
$I_{color}^{m}(\cdot)$ or equivalently $I_{dyn}^{m}(\cdot)$. To
carry this out, we start from a compactness result. 
\begin{lem}
\label{prop: Estimates based on bound on rate function}Suppose that
$\tilde{\rho}\in\mathscr{D}_{color}^{m}$ and $I_{color}^{m}(\tilde{\rho})<\infty$.
Then, we have
\begin{align}
\int_{0}^{T}\int_{\mathbb{T}}\nabla\tilde{\rho}^{\dagger}\chi A\chi(\tilde{\rho})\nabla\tilde{\rho}\,dxdt & \le C(1+I_{color}^{m}(\tilde{\rho}))\label{eq:ESTIMATE 2}
\end{align}
for some constant $C$. \end{lem}
\begin{proof}
For $g\in C^{\infty}([0,\,T]\times\mathbb{T}^{m})$ with $\int_{\mathbb{T}}g(t,\,x)dx=0$
for all $t\in[0,\,T]$, we can consider a semi-norm $||g||_{\mathscr{H}_{-1}(A(\tilde{\rho}))}^{2}=\int_{0}^{T}||g||_{-1,A(\tilde{\rho})}^{2}$
and by taking completion and equivalence class, we obtain $\mathscr{H}_{-1}$
space. For $h\in C^{\infty}([0,\,T]\times\mathbb{T}^{m})$, we have
another semi-norm $||h||_{\mathscr{H}_{1}(A(\tilde{\rho}))}^{2}=\int_{0}^{T}\int_{\mathbb{T}}\nabla h^{\dagger}A(\tilde{\rho})\nabla hdxdt$
and we can obtain $\mathscr{H}_{1}$ space in a similar manner. These
two spaces are dual each other and hence, for $g\in\mathscr{H}_{-1}$
and $h\in\mathscr{H}_{1}$, the integral $\int_{0}^{T}\int_{\mathbb{T}}g(t,\,x)h(t,\,x)dxdt$
is well-defined and satisfies Cauchy-Schwarz's inequality
\[
\int_{0}^{T}\int_{\mathbb{T}}g(t,\,x)h(t,\,x)dxdt\le||g||_{\mathscr{H}_{-1}(A(\tilde{\rho}))}||h||_{\mathscr{H}_{1}(A(\tilde{\rho}))}.
\]
Now, we write 
\begin{align*}
\log\tilde{\rho} & =(\log\rho_{1},\,\log\rho_{2},\,\cdots,\,\log\rho_{m})^{\dagger}\\
G & =\partial_{t}\tilde{\rho}-\frac{1}{2}\nabla\left[A\chi(\tilde{\rho})\nabla\tilde{\rho}\right]
\end{align*}
then $G,\,\nabla\left[A\chi(\tilde{\rho})\nabla\tilde{\rho}\right]\in\mathscr{H}_{-1}$
and $\log\tilde{\rho}\in\mathscr{H}_{1}$ where the norms can be easily
computed by the variational formula given in (2.24) of \cite{QRV}
such that 
\begin{align}
||G||_{\mathscr{H}_{-1}(A(\tilde{\rho}))}^{2} & =I_{dyn}^{m}(\tilde{\rho})\label{eq:emmrtt}\\
||\nabla\left[A\chi(\tilde{\rho})\nabla\tilde{\rho}\right]||_{\mathscr{H}_{-1}(A(\tilde{\rho}))}^{2} & =\int_{0}^{T}\int_{\mathbb{T}}\nabla\tilde{\rho}^{\dagger}\chi A\chi(\tilde{\rho})\nabla\tilde{\rho}dxdt\label{eq: emmt}\\
||\log\tilde{\rho}||_{\mathscr{H}_{1}(A(\tilde{\rho}))}^{2} & =\int_{0}^{T}\int_{\mathbb{T}}\nabla\tilde{\rho}^{\dagger}\chi A\chi(\tilde{\rho})\nabla\tilde{\rho}dxdt.\label{eq:emmrst}
\end{align}
Now, let us consider the entropy functional 
\[
H_{t}(\tilde{\rho})=\sum_{c=1}^{m}\int_{\mathbb{T}}\rho_{c}(t,\,x)\log\rho_{c}(t,\,x)dx
\]
for $\tilde{\rho}\in\mathscr{D}_{color}^{m}$. Then, by (\ref{eq:emmrtt})
and (\ref{eq:emmrst}), 
\begin{align*}
H_{T}(\tilde{\rho})-H_{0}(\tilde{\rho}) & =\int_{0}^{T}\int_{\mathbb{T}}\left(\log\tilde{\rho}\right)^{\dagger}\left(\frac{1}{2}\nabla\cdot\left[A\chi(\tilde{\rho})\nabla\tilde{\rho}\right]+G\right)dxdt\\
 & =-\frac{1}{2}\int_{0}^{T}\int_{\mathbb{T}}\nabla\tilde{\rho}^{\dagger}\chi A\chi(\tilde{\rho})\nabla\tilde{\rho}+\int_{0}^{T}\int_{\mathbb{T}}\left(\log\tilde{\rho}\right)^{\dagger}Gdxdt\\
 & \le-\frac{1}{2}\int_{0}^{T}\int_{\mathbb{T}}\nabla\tilde{\rho}^{\dagger}\chi A\chi(\tilde{\rho})\nabla\tilde{\rho}+||\log\tilde{\rho}||_{\mathscr{H}_{1}(A(\tilde{\rho}))}||G||_{\mathscr{H}_{-1}(A(\tilde{\rho}))}\\
 & \le-\frac{1}{4}\int_{0}^{T}\int_{\mathbb{T}}\nabla\tilde{\rho}^{\dagger}\chi A\chi(\tilde{\rho})\nabla\tilde{\rho}+I_{dyn}^{m}(\tilde{\rho}).
\end{align*}
Note that 
\[
H_{0}(\tilde{\rho})\le\int_{\mathbb{T}}\rho(0,\,x)\log\rho(0,\,x)dx\le\int_{\mathbb{T}}\rho(0,\,x)\log\frac{\rho(0,\,x)}{\rho_{0}(x)}dx+\log||\rho_{0}||_{\infty}
\]
and $\int_{\mathbb{T}}\rho(0,\,x)\log\frac{\rho(0,\,x)}{\rho_{0}(x)}dx$
is the large deviation rate for $\mu^{N}(0)$ and therefore bounded
by $I_{init}^{m}(\tilde{\rho})$ by the contraction principle. This
proves (\ref{eq:ESTIMATE 2}).
\end{proof}
Establishing the lower semicontinuity or $I_{color}^{m}(\cdot)$ requires
a few convergence results. The following lemma implies that the weak
convergence can be combined with some energy estimates to obtain the
strong convergence. This lemma is motivated by Lemma 4.2 of \cite{QRV}
but our formulation and proof are differ slightly in that we do not
have \textit{a priori} boundedness of the density. 
\begin{lem}
\label{APProximation Lemma} Suppose that $\left\{ f_{N}(\cdot,x)dx\right\} _{N=1}^{\infty}\subset C([0,\,T],\,\mathscr{M}(\mathbb{T}))$
satisfies 
\begin{align}
f_{N}(\cdot,\,x)dx\rightharpoonup f(\cdot,\,x)dx & \,\,\,\,\,\text{weakly in }C([0,\,T],\,\mathscr{M}(\mathbb{T}))\label{eq:nu1}\\
\int_{\mathbb{T}}f_{N}(t,\,x)dx=\int_{\mathbb{T}}f(t,\,x)dx=\bar{f} & \,\,\,\,\,\text{for all \ensuremath{t}},\,N\label{eq:nu2}\\
\int_{0}^{T}\int_{\mathbb{T}}\frac{(\nabla f_{N})^{2}}{\alpha_{N}f_{N}}dxdt\le C & \,\,\,\,\,\text{for all }N\label{eq:Low Con 1}
\end{align}
where positive functions $\alpha_{N}(\cdot,\,\cdot)$ satisfies
\begin{equation}
\int_{\mathbb{T}}\alpha_{N}(t,\,x)\le M\label{eq:Low Con 2}
\end{equation}
uniformly in $t,\,N$ for some $M>0$. Then, $f_{N}\rightarrow f$
strongly in $L_{1}([0,\,T]\times\mathbb{T})$. Moreover, if $f_{N},\,f\in L_{2}([0,\,T]\times\mathbb{T})$
and $\alpha_{N}=1$ for all $N$ then $f_{N}\rightarrow f$ strongly
in $L_{2}([0,\,T]\times\mathbb{T})$.\end{lem}
\begin{proof}
For the first part, we first recall the notation of Remark \ref{RMK4.3}
and then it suffices to show 
\[
\lim_{\epsilon\rightarrow0}\limsup_{N\rightarrow\infty}\int_{0}^{T}\int_{\mathbb{T}}\left|f_{N}-\left(f_{N}\right)_{\epsilon}\right|+\left|\left(f_{N}\right)_{\epsilon}-f_{\epsilon}\right|+\left|f_{\epsilon}-f\right|dxdt=0.
\]
For the second term in this limit, observe first that $\left(f_{N}\right)_{\epsilon}\rightarrow f_{\epsilon}$
pointwise as $N\rightarrow\infty$ due to the weak convergence and
\[
\int_{0}^{T}\int_{\mathbb{T}}\left|\left(f_{N}\right)_{\epsilon}\right|dxdt=\int_{0}^{T}\int_{\mathbb{T}}\left|f_{\epsilon}\right|dxdt=T\bar{f}.
\]
Therefore we can apply Scheffe's Theorem to check the desired convergence.
The third term obviously tends to $0$ as $\epsilon\rightarrow0$
and therefore it is enough to show 
\begin{equation}
\lim_{\epsilon\rightarrow0}\sup_{N}\int_{0}^{T}\int_{\mathbb{T}}\left|f_{N}-\left(f_{N}\right)_{\epsilon}\right|dxdt=0.\label{eq: APOP1}
\end{equation}
By Cauchy-Schwarz's inequality we can bound 
\begin{eqnarray}
\int_{0}^{T}\int_{\mathbb{T}}\left|f_{N}-\left(f_{N}\right)_{\epsilon}\right| & \le & \sqrt{4T\bar{f}\int_{0}^{T}\int_{\mathbb{T}}\left(\sqrt{f_{N}}-\sqrt{\left(f_{N}\right)_{\epsilon}}\right)^{2}dx}.\label{eq: cskal}
\end{eqnarray}
Now we can bound $\int_{\mathbb{T}}\left(\sqrt{f_{N}}-\sqrt{\left(f_{N}\right)_{\epsilon}}\right)^{2}dx$
by
\begin{align}
 & \int_{\mathbb{T}}\int_{\mathbb{T}}\left(\sqrt{f_{N}(t,\,x)}-\sqrt{f_{N}(t,\,x+y)}\right)^{2}\phi_{\epsilon}(y)dydx\label{eq: asdfkkk}\\
 & \le\int_{\mathbb{T}}\int_{\mathbb{T}}\left(\int_{\mathbb{T}}1_{[x,\,x+y]}(z)\frac{\nabla f_{N}(t,\,z)}{2\sqrt{f_{N}(z)}}dz\right)^{2}\phi_{\epsilon}(y)dydx\nonumber \\
 & \le\frac{1}{4}\int_{\mathbb{T}}\int_{\mathbb{T}}\left(\int_{\mathbb{T}}\frac{\left|\nabla f_{N}\right|^{2}}{\alpha_{N}f_{N}}(t,\,z)dz\right)\left(\int_{\mathbb{T}}\mathds{1}_{[x,\,x+y]}(w)a_{N}(t,\,w)dw\right)\phi_{\epsilon}(y)dydx\nonumber \\
 & \le\frac{1}{4}\left(\int_{\mathbb{T}}\frac{\left|\nabla f_{N}\right|^{2}}{\alpha_{N}f_{N}}(t,\,z)dz\right)\int_{\mathbb{T}}\int_{\mathbb{T}}\int_{\mathbb{T}}\alpha_{N}(t,\,w)\mathds{1}_{[w-y,\,w]}(x)\phi_{\epsilon}(y)dxdydw\nonumber \\
 & =\frac{1}{4}\left(\int_{\mathbb{T}}\frac{\left|\nabla f_{N}\right|^{2}}{\alpha_{N}f_{N}}(t,\,z)dz\right)\left(\int_{\mathbb{T}}\alpha_{N}(t,\,w)dw\right)\left(\int_{\mathbb{T}}y\phi_{\epsilon}(y)dy\right).\nonumber 
\end{align}
By (\ref{eq:Low Con 1}), (\ref{eq:Low Con 2}) and (\ref{eq: cskal})
we obtain 
\[
\int_{0}^{T}\int_{\mathbb{T}}\left|f_{N}-\left(f_{N}\right)_{\epsilon}\right|\le\sqrt{CMT\bar{f}\int_{\mathbb{T}}y\phi_{\epsilon}(y)dy}
\]
which completes the proof of the first part. 

For the second part, the property $\alpha_{N}=1$ enable us to enhance
the calculations of (\ref{eq: asdfkkk}) in a way that
\begin{align*}
 & \int_{\mathbb{T}}\left|f_{N}-\left(f_{N}\right)_{\epsilon}\right|^{2}dxdt\\
 & \le\int_{\mathbb{T}}\int_{\mathbb{T}}\left(f_{N}(t,\,x)-f_{N}(t,\,x+y)\right)^{2}\phi_{\epsilon}(y)dydx\\
 & \le\frac{1}{4}\int_{\mathbb{T}}\int_{\mathbb{T}}\left(\int_{\mathbb{T}}1_{[x,\,x+y]}(z)\frac{\left|\nabla f_{N}(t,\,z)\right|^{2}}{f_{N}(t,\,z)}dz\right)\left(\int_{\mathbb{T}}f_{N}(t,\,w)dw\right)\phi_{\epsilon}(y)dydx\\
 & =\frac{\bar{f}}{4}\int_{\mathbb{T}}\int_{\mathbb{T}}\int_{\mathbb{T}}\frac{\left|\nabla f_{N}(t,\,z)\right|^{2}}{f_{N}(t,\,z)}1_{[z-y,\,z]}(x)\phi_{\epsilon}(y)dxdydw\\
 & =\frac{\bar{f}}{4}\int_{\mathbb{T}}\frac{\left|\nabla f_{N}(t,\,z)\right|^{2}}{f_{N}(t,\,z)}dz\int_{\mathbb{T}}y\phi_{\epsilon}(y)dy
\end{align*}
and we are done. 
\end{proof}
The following lemma is a summary of elementary convergence results
which are useful in our context. 
\begin{lem}
\label{lem: frequently used}Let $\left\{ f_{N}\right\} _{N=1}^{\infty}$,
$\left\{ g_{N}\right\} _{N=1}^{\infty}$ be sequences of functions
on $[0,\,T]\times\mathbb{T}$. 
\begin{enumerate}
\item \noindent If $f_{N}\rightarrow f$, $g_{N}\rightarrow g$ strongly
in $L_{1}$ and $\left\Vert f_{N}\right\Vert _{L_{\infty}}<C$ for
all $N$, then $f_{N}g_{N}\rightarrow fg$ strongly in $L_{1}$. 
\item \noindent If $f_{N}\rightarrow f$ strongly in $L_{2}$ and $g{}_{N}\rightharpoonup g$
weakly in $L_{2}$ then $f_{N}g_{N}\rightharpoonup fg$ weakly in
$L_{1}$.
\item \noindent Assuming that $f_{N},\,f$ are weakly differentiable and
$g_{N},\,g>0$ for all $N$. If $f_{N}\rightharpoonup f$ weakly in
$L_{1}$, $g_{N}\rightarrow g$ strongly in $L_{2}$ and $\left\{ \nabla f_{N}/g_{N}\right\} _{N=1}^{\infty}$
is uniformly bounded in $L_{2}$, then $\nabla f_{N}/g_{N}\rightharpoonup\nabla f/g$
weakly in $L_{2}$. 
\end{enumerate}
\end{lem}
\begin{proof}
(1) For any $M>0$, 
\begin{align*}
\int_{0}^{T}\int_{\mathbb{T}}\left|f_{N}g_{N}-fg\right|dxdt & \le\int_{0}^{T}\int_{\mathbb{T}}\left|f_{N}\right|\left|g_{N}-g\right|+\left|g\right|\left|f_{N}-f\right|dxdt\\
 & \le\int_{0}^{T}\int_{\mathbb{T}}C\left|g_{N}-g\right|+M\left|f_{N}-f\right|+2C\left|g\right|1_{|g|>M}dxdt
\end{align*}
then we can send $N\rightarrow\infty$ and then $M\rightarrow\infty$
to obtain the desired result. 

$ $

\noindent (2) For any bounded function $U$,
\begin{align*}
 & \left|\int_{0}^{T}\int_{\mathbb{T}}\left(Uf_{N}g_{N}-Ufg\right)dxdt\right|\\
 & \le\int_{0}^{T}\int_{\mathbb{T}}\left|U\right|\left|g_{N}\right|\left|f_{N}-f\right|dxdt+\left|\int_{0}^{T}\int_{\mathbb{T}}\left(Ufg_{N}-Ufg\right)dxdt\right|
\end{align*}
and since $\left\{ g_{N}\right\} _{N=1}^{\infty}$ is uniformly bounded
in $L_{2}$ the first term converges to $0$. $Uf$ is a $L_{2}$
function and therefore the second term goes to 0 as well. 

$ $

\noindent (3) For any subsequence of $\{\nabla f_{N}/g_{N}\}_{N=1}^{\infty}$,
we can take a further subsequence which converges weakly in $L_{2}$
to some $u$. Then it suffices to show $u=\frac{\nabla f}{g}$ almost
surely. To this end, without loss of generality, we assume $\frac{\nabla f_{N}}{g_{N}}\rightharpoonup u$
weakly in $L_{2}$ instead of its subsequence. Then, $\nabla f_{N}\rightharpoonup gu$
weakly in $L_{1}$ by (2). However, for any smooth function $v$,
\begin{align*}
\int_{0}^{T}\int_{\mathbb{T}}v(gu)dxdt & =\lim_{N\rightarrow\infty}\int_{0}^{T}\int_{\mathbb{T}}v\nabla f_{N}dxdt=\lim_{N\rightarrow\infty}-\int_{0}^{T}\int_{\mathbb{T}}f_{N}\nabla vdxdt\\
 & =-\int_{0}^{T}\int_{\mathbb{T}}f\nabla vdxdt=\int_{0}^{T}\int_{\mathbb{T}}v\nabla fdxdt
\end{align*}
and therefore we obtain $\nabla f=gu$. 
\end{proof}
Now we are ready to prove the lower semicontinuity of the rate function. 
\begin{thm}
\label{thm: LSC}The functional $I_{color}^{m}(\cdot)$ is lower semicontinuous. \end{thm}
\begin{proof}
It suffices to show that if $\tilde{\rho}^{(k)}(t,\,x)dx\rightharpoonup\tilde{\rho}(t,\,x)dx$
weakly in $C([0,\,T],\,\mathscr{M}(\mathbb{T}))$ and $I_{color}^{m}(\tilde{\rho}^{(k)})\le M$
for all $k$ then $I_{color}^{m}(\tilde{\rho})\le M$. Since we already
assumed the initial LDP as in Assumption 2, it is enough to consider
the dynamic part. 

We start by considering a functional 
\begin{align}
\Lambda_{\phi}(\tilde{\rho})= & \int_{\mathbb{T}}\phi^{\dagger}\tilde{\rho}(T,\,x)dx-\int_{\mathbb{T}}\phi^{\dagger}\tilde{\rho}(0,\,x)dx\nonumber \\
 & +\int_{0}^{T}\int_{\mathbb{T}}\left[-\frac{\partial\phi}{\partial t}^{\dagger}\tilde{\rho}+\frac{1}{2}\nabla\phi^{\dagger}D(\tilde{\rho})\nabla\tilde{\rho}-\frac{1}{2}\nabla\phi^{\dagger}A(\tilde{\rho})\nabla\phi\right]dxdt\Biggl\}\label{eq: naham}
\end{align}
on $\mathscr{D}_{color}^{m}$, then 
\[
I_{dyn}^{m}(\tilde{\rho})=\sup_{\phi\in C^{\infty}([0,\,T]\times\mathbb{T})}\Lambda_{\phi}(\tilde{\rho})
\]
and hence it is enough to show $\lim_{k\rightarrow\infty}\Lambda_{\phi}(\tilde{\rho}^{(k)})=\Lambda_{\phi}(\tilde{\rho})$.
The convergences of the first three terms in (\ref{eq: naham}) are
direct from the weak convergence of $\tilde{\rho}^{(k)}$ and therefore
it suffices to show 
\begin{eqnarray}
 & D(\tilde{\rho}^{(k)})\nabla\tilde{\rho}^{(k)}\rightharpoonup D(\tilde{\rho})\nabla\tilde{\rho} & \,\,\,\,\text{weakly in \ensuremath{L_{1}([0,\,T]\times\mathbb{T})}}\label{eq: KSJJ2}\\
 & A(\tilde{\rho}^{(k)})\rightarrow A(\tilde{\rho}) & \,\,\,\,\text{strongly in \ensuremath{L_{1}([0,\,T]\times\mathbb{T}).}}\label{eq: KSJJ1}
\end{eqnarray}
First note that the uniform boundedness of $I_{color}^{m}(\tilde{\rho}^{(k)})$
and Lemma \ref{prop: Estimates based on bound on rate function} together
imply
\begin{equation}
\int_{0}^{T}\int_{\mathbb{T}}\frac{(\nabla\rho^{(k)})^{2}}{\rho^{(k)}}<M'\text{\,\,\ and\,\,}\int_{0}^{T}\int_{\mathbb{T}}\frac{\left(\nabla\rho_{c}^{(k)}\right)^{2}}{(\lambda+\rho^{(k)})\rho_{c}^{(k)}}<M'\text{\,,\,\ \ensuremath{\forall}\ensuremath{c}\,}\label{eq:unfm bd}
\end{equation}
for some $M'$. Thus, $\left\{ \rho_{c}^{(k)}\right\} _{k=1}^{\infty}$
satisfies the conditions of Lemma \ref{APProximation Lemma} with
$\alpha_{c}^{(k)}=\lambda+\rho^{(k)}$, and therefore convergence
of $\rho_{c}^{(k)}\rightarrow\rho_{c}$ is strong in $L_{1}$. Moreover,
$\left\{ \rho^{(k)}\right\} _{k=1}^{\infty}$ satisfies the conditions
of the second part of Lemma \ref{APProximation Lemma} because of
Lemma \ref{lem: L_2 boundness-1} and hence $\rho^{(k)}\rightarrow\rho$
strongly in $L_{2}$.

To show (\ref{eq: KSJJ2}), first note that the $c$th element of
$D(\tilde{\rho}^{(k)})\nabla\tilde{\rho}^{(k)}$ is $\frac{\lambda\nabla\rho_{c}^{(k)}}{\lambda+\rho^{(k)}}+\frac{\rho_{c}^{(k)}}{\lambda+\rho^{(k)}}\nabla\rho^{(k)}$
and therefore it suffices to show that 
\begin{eqnarray}
 & \frac{\nabla\rho_{c}^{(k)}}{\lambda+\rho^{(k)}}\rightharpoonup\frac{\nabla\rho_{c}}{\lambda+\rho} & \,\,\,\,\text{weakly in \ensuremath{L_{1}([0,\,T]\times\mathbb{T})}}\label{eq: KSJJ3}\\
 & \frac{\rho_{c}^{(k)}}{\lambda+\rho^{(k)}}\nabla\rho^{(k)}\rightharpoonup\frac{\rho_{c}}{\lambda+\rho}\nabla\rho & \,\,\,\,\text{weakly in \ensuremath{L_{1}([0,\,T]\times\mathbb{T})}}\label{eq:KSJJ4}
\end{eqnarray}
for each $c$. Note that (\ref{eq: KSJJ3}) follows directly from
(\ref{eq:unfm bd}) and (3) of Lemma \ref{lem: frequently used}.
For (\ref{eq:KSJJ4}), by the same argument as before, we can show
$\frac{\nabla\rho^{(k)}}{\sqrt{\lambda+\rho^{(k)}}}\rightharpoonup\frac{\nabla\rho}{\sqrt{\lambda+\rho}}$
weakly in $L_{2}$ and it is also easy to check that $\frac{\rho_{c}^{(k)}}{\sqrt{\lambda+\rho^{(k)}}}\rightarrow\frac{\rho_{c}}{\sqrt{\lambda+\rho}}$
strongly in $L_{2}$. Thus, by (2) of Lemma \ref{lem: frequently used},
we can prove (\ref{eq:KSJJ4}). 

To prove (\ref{eq: KSJJ1}), we need to show
\begin{eqnarray}
 & \frac{\lambda+\rho_{c}^{(k)}}{\lambda+\rho^{(k)}}\rho_{c}^{(k)}\rightarrow\frac{\lambda+\rho_{c}}{\lambda+\rho}\rho_{c} & \,\,\,\,\text{strongly in \ensuremath{L_{1}\left([0,\,T]\times\mathbb{T}\right)}}\label{eq:skjs}\\
 & \frac{\rho_{c}^{(k)}}{\lambda+\rho^{(k)}}\rho_{c'}^{(k)}\rightarrow\frac{\rho_{c}}{\lambda+\rho}\rho_{c'} & \,\,\,\,\text{strongly in \ensuremath{L_{1}\left([0,\,T]\times\mathbb{T}\right)}}\label{eq:sksj}
\end{eqnarray}
for each $c,\,c'$. Since $\frac{\lambda+\rho_{c}^{(k)}}{\lambda+\rho^{(k)}}\rightarrow\frac{\lambda+\rho_{c}}{\lambda+\rho}$
and $\frac{\rho_{c}^{(k)}}{\lambda+\rho^{(k)}}\rightarrow\frac{\rho_{c}}{\lambda+\rho}$
strongly in $L_{1}$ and bounded by $1$, we can prove (\ref{eq:skjs}),
(\ref{eq:sksj}) by (1) of Lemma \ref{lem: frequently used}.
\end{proof}

\subsection{Upper Bound\label{sub:Upper-Bound}}

In this section, we establish the LDP upper bound for $\bigl\{\widetilde{\mathbb{Q}}_{N}\bigr\}_{N=1}^{\infty}$
with the rate function $I_{color}^{m}(\cdot)$. The upper bound is
usually based on the exponential martingale with a mean of $1$ and
the martingale should be suitably chosen such that it can be approximated
by the density fields $\tilde{\mu}^{N}(\cdot)$. An exponential martingale
such as this can be built by first using $z_{i}^{N}(t)$ (\textit{cf.}
(\ref{eq:z(t)})) in a way such that 
\begin{equation}
\frac{1}{N}\sum_{c=1}^{m}\sum_{i\in I_{c}^{N}}\int_{0}^{T}g_{x}^{(c)}(t,\,z_{i}^{N}(t))dz_{i}^{N}(t).\label{eq: hyu3}
\end{equation}
where $\tilde{g}=(g^{(1)},\,g^{(2)},\,\cdots,\,g^{(m)})^{\dagger}\in C^{1,2}([0,\,T]\times\mathbb{T})^{m}$.
This martingale can be reinterpreted as 
\begin{align}
\frac{1}{N}\sum_{c=1}^{m}\sum_{i\in I_{c}^{N}}\Biggl[g^{(c)}(T,\,z_{i}^{N}(T))-g^{(c)}(0,\,z_{i}^{N}(0))-\int_{0}^{T}g_{t}^{(c)}(t,\,z_{i}^{N}(t))dt\label{eq: ord_martingale 1}\\
-\frac{1}{2}\int_{0}^{T}g_{xx}^{(c)}(t,\,z_{i}^{N}(t))d\left\langle z_{i}^{N},\,z_{i}^{N}\right\rangle _{t} & \Biggr]\nonumber 
\end{align}
according to Ito's formula. As we commented in Section \ref{sub:Small-Time-Regime},
we can represent $z_{i}^{N}(t)-z_{i}^{N}(0)$ as 
\begin{equation}
\frac{N\lambda+1}{N(\lambda+1)}\beta_{i}(t)+\frac{1}{N(\lambda+1)}\sum_{k:k\neq i}\beta_{k}(t)+\frac{1}{N(\lambda+1)}\sum_{k:k\neq i}(M_{ki}^{N}(t)-M_{ik}^{N}(t))\label{eq: z(t) as martingale}
\end{equation}
and therefore the quadratic variation $d\left\langle z_{i}^{N},\,z_{i}^{N}\right\rangle _{t}$
is 
\begin{equation}
\left[\frac{\lambda^{2}}{(\lambda+1)^{2}}+\frac{2\lambda+1}{N(\lambda+1)^{2}}\right]dt+\frac{\lambda}{(\lambda+1)^{2}}dA_{i}^{N}(t)\label{eq: QV of z(t)}
\end{equation}
Note that, although this expression relates to the local time, we
can replace it by the local density by using Theorem \ref{thm: REPLACEMENT_colored}.
However, even after that, we still have a problem in (\ref{eq: ord_martingale 1}).
Broadly stated, we have a nuisance term relating to $\rho_{t}$ in
the final stage that should not have appeared. The strategy for eliminating
this term is to add another martingale 
\begin{equation}
\frac{1}{N}\sum_{i=1}^{N}\int_{0}^{T}J_{x}(t,\,x_{i}^{N}(t))d\beta_{i}(t)\label{eq: hyu4}
\end{equation}
to (\ref{eq: hyu3}) with a suitably chosen $J\in C^{1,2}([0,\,T]\times\mathbb{T})$,
which also has an alternative representation 
\begin{equation}
\frac{1}{N}\sum_{i=1}^{N}\left[J(T,\,x_{i}^{N}(T))-J(0,\,x_{i}^{N}(0))-\int_{0}^{T}\left(J_{t}+\frac{1}{2}J_{xx}\right)(t,\,x_{i}^{N}(t))dt\right]\label{eq: ord_martingale 2}
\end{equation}
according to Ito's formula. 

We now start the proof of the upper bound by defining a martingale
$M_{N}(\tilde{g},\,J)$ for $\tilde{g}\in C^{1,2}([0,\,T]\times\mathbb{T}^{m})$
and $J\in C^{1,2}([0,\,T]\times\mathbb{T})$ by 
\begin{align*}
 & M_{N}(\tilde{g},\,J)\\
 & =\int_{0}^{T}\sum_{c=1}^{m}\sum_{i\in I_{c}^{N}}g_{x}^{(c)}(t,\,z_{i}^{N}(t))dz_{i}^{N}(t)+\int_{0}^{T}\sum_{i=1}^{N}J_{x}(t,\,x_{i}^{N}(t))d\beta_{i}(t)\\
 & =\int_{0}^{T}\sum_{c=1}^{m}\sum_{i\in I_{c}^{N}}\left[\frac{\lambda}{\lambda+1}g_{x}^{(c)}(t,\,z_{i}^{N}(t))+\frac{1}{\lambda+1}G_{N}(t,\,x^{N}(t))+J_{x}(t,\,x_{i}^{N}(t))\right]d\beta_{i}\\
 & \,\,\,\,\,\,\,+\sum_{\substack{1\le c_{1}<c_{2}\le m}
}\sum_{\substack{i\in I_{c_{1}}^{N},\,j\in I_{c_{2}}^{N}}
}\int_{0}^{T}\frac{\mathbf{g}_{x}^{(c_{1},\,c_{2})}(t,\,z_{i}^{N}(t))}{N(\lambda+1)}\left(dM_{ij}^{N}(t)-dM_{ji}^{N}(t)\right)
\end{align*}
by (\ref{eq: z(t) as martingale}) where
\begin{align*}
G_{N}(t,\,x^{N}(t)) & =\frac{1}{N}\sum_{c=1}^{m}\sum_{i\in I_{c}^{N}}g^{(c)}(t,\,z_{i}^{N}(t))\\
\mathbf{g}_{x}^{(c_{1},\,c_{2})}(t,\,z_{i}^{N}(t)) & =g_{x}^{(c_{1})}(t,\,z_{i}^{N}(t))-g_{x}^{(c_{2})}(t,\,z_{i}^{N}(t)).
\end{align*}
The next object to be characterized is $A_{N}(\tilde{g},\,J)$ which
must satisfy 
\begin{equation}
\mathbb{E}_{N}\exp\Bigl\{ M_{N}(\tilde{g},\,J)-A_{N}(\tilde{g},\,J)\Bigr\}=1.\label{eq: exponential martingale}
\end{equation}
We can find such an $A_{N}(\tilde{g},\,J)$ by 
\begin{align*}
 & A_{N}(\tilde{g},\,J)\\
 & =\frac{1}{2}\int_{0}^{T}\sum_{c=1}^{m}\sum_{i\in I_{c}^{N}}\left[\frac{\lambda}{\lambda+1}g_{x}^{(c)}(t,\,z_{i}^{N}(t))+\frac{1}{\lambda+1}G_{N}(t,\,x^{N}(t))+J_{x}(t,\,x_{i}^{N}(t))\right]^{2}dt\\
 & \,\,\,\,\,\,+\lambda N\sum_{\substack{1\le c_{1}<c_{2}\le m}
}\,\sum_{i\in I_{c_{1}}^{N},\,j\in I_{c_{2}}^{N}}\int_{0}^{T}U\left(\frac{\mathbf{g}_{x}^{(c_{1},\,c_{2})}(t,\,z_{i}^{N}(t))}{N(\lambda+1)}\right)\left(dA_{ij}^{N}(t)+dA_{ji}^{N}(t)\right)\\
 & =\frac{1}{2}\int_{0}^{T}\sum_{c=1}^{m}\sum_{i\in I_{c}^{N}}\left[\frac{\lambda}{\lambda+1}g_{x}^{(c)}(t,\,z_{i}^{N}(t))+\frac{1}{\lambda+1}G_{N}(t,\,x^{N}(t))+J_{x}(t,\,x_{i}^{N}(t))\right]^{2}dt\\
 & \,\,\,\,\,\,+\frac{\lambda}{2N(\lambda+1)^{2}}\sum_{\substack{1\le c_{1}<c_{2}\le m}
}\,\sum_{i\in I_{c_{1}}^{N}}\int_{0}^{T}\left(\mathbf{g}_{x}^{(c_{1},\,c_{2})}(t,\,z_{i}^{N}(t))\right)^{2}dA_{i,c_{2}}^{N}(t)+O_{N}(1)
\end{align*}
where $U(x)=e^{x}-x-1\sim\frac{x^{2}}{2}$. Note that the error term
is $O_{N}(1)$ because of (\ref{eq:Estimate of Total Local Time}). 

The next step is to approximate $M_{N}(\tilde{g},\,J)$ and $A_{N}(\tilde{g},\,J)$
by a density field of $\tilde{\mu}^{N}(\cdot)$. To carry this program
out, we define a set $\mathscr{B}_{N,\epsilon,\delta}\subset C([0,\,T],\,\mathbb{T}^{N})$
such that $x(\cdot)\in\mathscr{B}_{N,\epsilon,\delta}$ if and only
if 
\begin{equation}
\left|\int_{0}^{T}\widetilde{V}_{N,\epsilon}^{\mathbf{g}}(t,\,x(t))dt\right|<\delta\,\,\,\text{and}\,\,\,\left|\int_{0}^{T}V_{N,\epsilon}^{(c)}(t,\,x(t))dt\right|<\delta\,\,\,\text{ for }c=1,\,2,\,\cdots,\,m\label{eq: etm1}
\end{equation}
where
\begin{align*}
V_{N,\epsilon}^{(c)}(t,\,x)= & \frac{1}{N^{2}}\sum_{i=1}^{N}g_{xx}^{(c)}(t,\,z_{i})\sum_{j:j\neq i}\left[\frac{1}{2\epsilon}\chi_{\epsilon}(x_{j}-x_{i})-\left(\delta^{+}(x_{j}-x_{i})+\delta^{+}(x_{i}-x_{j})\right)\right]\\
\widetilde{V}_{N,\epsilon}^{\mathbf{g}}(t,\,x)= & \frac{1}{N^{2}}\sum_{\substack{1\le c_{1}<c_{2}\le m}
}\,\sum_{i\in I_{c_{1}}^{N},\,j\in I_{c_{2}}^{N}}\Biggl\{\left(\mathbf{g}_{x}^{(c_{1},\,c_{2})}(t,\,z_{i})\right)^{2}\\
 & \,\,\,\,\,\,\,\,\,\,\,\,\,\,\,\,\,\,\,\,\,\,\,\,\,\,\,\,\,\,\,\,\,\,\,\,\,\,\,\,\,\,\,\,\,\,\,\,\,\times\left[\frac{1}{2\epsilon}\chi_{\epsilon}(x_{j}-x_{i})-\left(\delta^{+}(x_{j}-x_{i})+\delta^{+}(x_{i}-x_{j})\right)\right]\Biggl\}.
\end{align*}
Recall from Theorem \ref{thm: REPLACEMENT_colored} that 
\begin{equation}
\limsup_{\epsilon\rightarrow0}\limsup_{N\rightarrow\infty}\frac{1}{N}\mathbb{P}_{N}\left[\mathscr{B}_{N,\epsilon,\delta}^{c}\right]=-\infty.\label{eq:SUPPEP}
\end{equation}
for any $\delta>0$. Now, we can approximate $M_{N}(\tilde{g},\,J)$
and $A_{N}(\tilde{g},\,J)$ by the density field for $x^{N}(\cdot)\in\mathscr{B}_{N,\epsilon,\delta}$.
More precisely, (\ref{eq: ord_martingale 1}) and (\ref{eq: ord_martingale 2})
imply 
\begin{equation}
M_{N}(\tilde{g},\,J)=N\left[\Phi_{\epsilon,\tilde{g},J}(\tilde{\mu}^{N}(\cdot))+O(\delta)\right]\label{eq:exponet 1}
\end{equation}
for $x^{N}(\cdot)\in\mathscr{B}_{N,\epsilon,\delta}$ where the functional
$\Phi_{\epsilon,\tilde{g},J}(\cdot)$ on $C([0,\,T],\,\mathscr{M}(\mathbb{T})^{m})$
is defined as 
\begin{align}
\Phi_{\epsilon,\tilde{g},J}(\tilde{\pi}_{\cdot})= & \left\langle \tilde{\pi}_{T},\,\tilde{g}(T,\,F_{\pi}(T,\,x))\right\rangle -\left\langle \tilde{\pi}_{0},\,\tilde{g}(0,\,F_{\pi}(0,\,x))\right\rangle \nonumber \\
 & -\int_{0}^{T}\left\langle \tilde{\pi}_{t},\,\left(\tilde{g_{t}}+\frac{\lambda(\lambda+\left(\pi_{t}*\iota_{\epsilon}\right)(x)}{2(\lambda+1)^{2}}\tilde{g}_{xx}\right)(t,\,F_{\pi}(t,\,x))\right\rangle dt\nonumber \\
 & +\left\langle \pi_{T},\,J(T,\,x)\right\rangle -\left\langle \pi_{0},\,J(0,\,x)\right\rangle -\int_{0}^{T}\left\langle \pi_{t},\,\left(J_{t}+\frac{1}{2}J_{xx}\right)\left(t,\,x\right)\right\rangle dt\label{eq: APPPP1}
\end{align}
for $\tilde{\pi}_{\cdot}=(\pi_{\cdot}^{1},\,\pi_{\cdot}^{2},\,\cdots,\,\pi_{\cdot}^{m})^{\dagger}$,
$\pi_{t}=\sum_{c=1}^{m}\pi_{t}^{c}$ and 
\[
F_{\pi}(t,\,x)=x+\frac{1}{\lambda+1}\left\langle \pi_{t}(dy),\,\nu(y-x)\right\rangle 
\]
Note that we used (\ref{eq: etm1}) to replace the local time by the
local density $\pi_{t}*\iota_{\epsilon}$. Similarly, we can obtain\textit{
}
\begin{equation}
A_{N}(\tilde{g},\,J)=N\left[\Psi_{\epsilon,\tilde{g},J}(\tilde{\mu}^{N}(\cdot))+O(\delta)+O\left(\frac{1}{N}\right)\right]\label{eq:exponet 2}
\end{equation}
for $x^{N}(\cdot)\in\mathscr{B}_{N,\epsilon,\delta}$ where
\begin{align*}
\Psi_{\epsilon,\tilde{g},J}(\tilde{\pi}_{\cdot})= & \,\frac{\lambda^{2}}{2(\lambda+1)^{2}}\int_{0}^{T}\left\langle \tilde{\pi}_{t},\,\tilde{g}_{x}^{2}\left(t,\,F_{\pi}(t,\,x)\right)\right\rangle dt\\
 & +\frac{\lambda}{\lambda+1}\int_{0}^{T}\left\langle \tilde{\pi}_{t},\,\left(J_{x}(t,\,x)+K_{\tilde{\pi},\tilde{g}}(t)\right)\tilde{g}_{x}\left(t,\,F_{\pi}(t,\,x)\right)\right\rangle dt\\
 & +\frac{1}{2}\int_{0}^{T}\left\langle \tilde{\pi}_{t},\,\left(J_{x}(t,\,x)+K_{\tilde{\pi},\tilde{g}}(t)\right)^{2}\right\rangle dt\\
 & +\frac{\lambda}{2(\lambda+1)^{2}}\int_{0}^{T}\left\langle \tilde{\pi}_{t},\,\left(\pi_{t}*\iota_{\epsilon}\right)(x)\,\tilde{g}_{x}^{2}(t,\,F_{\pi}(t,\,x))\right\rangle dt\\
 & -\frac{\lambda}{2(\lambda+1)^{2}}\int_{0}^{T}\left\langle \tilde{\pi}_{t},\,L_{\tilde{\pi},\tilde{g},\epsilon}(t,\,x)\,\tilde{g}_{x}(t,\,F_{\pi}(t,\,x))\right\rangle dt
\end{align*}
with
\begin{align}
\tilde{g}_{x}^{2}(t,\,x)= & \left(\tilde{g}_{x}^{(1)}(t,\,x)^{2},\,\tilde{g}_{x}^{(2)}(t,\,x)^{2},\,\cdots,\,\tilde{g}_{x}^{(m)}(t,\,x)^{2}\right)^{\dagger}\nonumber \\
K_{\tilde{\pi},\tilde{g}}(t)= & \frac{1}{\lambda+1}\left\langle \tilde{\pi}_{t},\,\tilde{g}_{x}(t,\,F_{\pi}(t,\,x))\right\rangle \label{eq:APPPP2}\\
L_{\tilde{\pi},\tilde{g},\epsilon}(t,\,x)= & \left(\tilde{\pi}_{t}*\iota_{\epsilon}\right)(t,\,x)\cdot\tilde{g}_{x}\left(t,\,F_{\pi}(t,\,x)\right).\nonumber 
\end{align}
We can combine (\ref{eq: exponential martingale}), (\ref{eq:exponet 1})
and (\ref{eq:exponet 2}) so that 
\begin{align}
 & \frac{1}{N}\log\mathbb{E}_{N}\left[\exp\left\{ N\left[\Phi_{\epsilon,\tilde{g},J}(\tilde{\mu}^{N}(\cdot))-\Psi_{\epsilon,\tilde{g},J}(\tilde{\mu}^{N}(\cdot))\right]\right\} \cdot\mathds{1}_{\mathscr{B}_{N,\epsilon,\delta}}\right]\label{eq:exponent 3}\\
 & =O(\delta)+O\left(\frac{1}{N}\right)\nonumber 
\end{align}

Now we are ready to establish the large deviation upper bound for
compact sets by the standard method (e.g., Chapter 10 of \cite{KL}).
For any open set $\mathscr{O}\subset C([0,\,T],\,\mathscr{M}(\mathbb{T})^{m})$,
\begin{align*}
 & \limsup_{N\rightarrow\infty}\frac{1}{N}\log\widetilde{\mathbb{Q}}_{N}\left[\mathscr{O}\right]\\
 & \le\max\biggl\{\limsup_{N\rightarrow\infty}\frac{1}{N}\log\mathbb{P}_{N}\left[\left\{ \tilde{\mu}^{N}(\cdot)\in\mathscr{O}\right\} \cap\mathscr{B}_{N,\epsilon,\delta}\right],\,\limsup_{N\rightarrow\infty}\frac{1}{N}\log\mathbb{P}_{N}\left[\mathscr{B}_{N,\epsilon,\delta}^{c}\right]\biggr\}
\end{align*}
and by Chebyshev's inequality with (\ref{eq:exponent 3}), 
\begin{align*}
 & \frac{1}{N}\log\mathbb{P}_{N}\left[\left\{ \tilde{\mu}^{N}\in\mathscr{O}\right\} \cap\mathscr{B}_{N,\epsilon,\delta}\right]\\
 & \le\frac{1}{N}\log\mathbb{E}_{N}\left[\exp\left\{ N\left[\Phi_{\epsilon,\tilde{g},J}(\tilde{\mu}^{N}(\cdot))-\Psi_{\epsilon,\tilde{g},J}(\tilde{\mu}^{N}(\cdot))\right]\right\} \cdot\ensuremath{\mathds{1}}_{\mathscr{B}_{N,\epsilon,\delta}}\right]\\
 & \,\,\,\,\,\,-\inf_{\tilde{\pi}_{\cdot}\in\mathscr{O}}\left\{ \Phi_{\epsilon,\tilde{g},J}(\tilde{\pi}_{\cdot})-\Psi_{\epsilon,\tilde{g},J}(\tilde{\pi}_{\cdot})\right\} \\
 & =-\inf_{\tilde{\pi}_{\cdot}\in\mathscr{O}}\left\{ \Phi_{\epsilon,\tilde{g},J}(\tilde{\pi}_{\cdot})-\Psi_{\epsilon,\tilde{g},J}(\tilde{\pi}_{\cdot})\right\} +O(\delta)+O\left(\frac{1}{N}\right)
\end{align*}
Hence, we obtain 
\[
\limsup_{N\rightarrow\infty}\frac{1}{N}\log\widetilde{\mathbb{Q}}_{N}\left[\mathscr{O}\right]\le\inf_{\epsilon,\delta,\tilde{g},J}\sup_{\tilde{\pi}_{\cdot}\in\mathscr{O}}\Omega_{\epsilon,\delta,N,\tilde{g},J}(\tilde{\pi}_{\cdot})
\]
where $\Omega_{\epsilon,\delta,N,\tilde{g},J}(\tilde{\pi}_{\cdot})$
is defined as 
\[
\max\left\{ -\left(\Phi_{\epsilon,\tilde{g},J}(\tilde{\pi}_{\cdot})-\Psi_{\epsilon,\tilde{g},J}(\tilde{\pi}_{\cdot})\right)+O(\delta),\,\limsup_{N\rightarrow\infty}\frac{1}{N}\log\mathbb{P}_{N}\left[\mathscr{B}_{N,\epsilon,\delta}^{c}\right]\right\} .
\]
Then, by the Minimax lemma (\textit{cf.} Lemma 3.2 of Appendix 2 of
\cite{KL}), we have 
\begin{equation}
\limsup_{N\rightarrow\infty}\frac{1}{N}\log\widetilde{\mathbb{Q}}_{N}\left[\mathscr{K}\right]\le\sup_{\tilde{\pi}_{\cdot}\in\mathscr{K}}\inf_{\epsilon,\delta,\tilde{g},J}\Omega_{\epsilon,\delta,N,\tilde{g},J}(\tilde{\pi}_{\cdot})\label{eq: expona}
\end{equation}
for all compact sets $\mathscr{K}\subset C([0,\,T],\,\mathscr{M}(\mathbb{T})^{m})$.
Notice that, by Lemma \ref{lem: Domain of Rate function}, 
\[
\limsup_{N\rightarrow\infty}\frac{1}{N}\log\widetilde{\mathbb{Q}}_{N}\left[\mathscr{K}\right]=\limsup_{N\rightarrow\infty}\frac{1}{N}\log\widetilde{\mathbb{Q}}_{N}\left[\mathscr{K}\cap\mathscr{D}_{color}^{m}\right]
\]
and therefore we can replace $\sup_{\tilde{\pi}_{\cdot}\in\mathscr{K}}$
in (\ref{eq: expona}) by $\sup_{\tilde{\rho}(\cdot,\,x)dx\in\mathscr{\mathscr{K\cap}D}_{color}^{m}}$.
Moreover, for $\tilde{\rho}\in\mathscr{D}_{color}^{m}$, we have\footnote{The precise form is $\Phi_{\epsilon,\tilde{g},J}(\tilde{\rho}(\cdot,\,x)dx)$
and so on. } 
\[
\Phi_{\epsilon,\tilde{g},J}(\tilde{\rho})-\Psi_{\epsilon,\tilde{g},J}(\tilde{\rho})=\Phi_{\tilde{g},J}(\tilde{\rho})-\Psi_{\tilde{g},J}(\tilde{\rho})+o_{\epsilon}(1)
\]
where $\Phi_{\tilde{g},J}(\tilde{\rho})$ and $\Psi_{\tilde{g},J}(\tilde{\rho})$
are derived from $\Phi_{\epsilon,\tilde{g},J}(\tilde{\rho})$ and
$\Psi_{\epsilon,\tilde{g},J}(\tilde{\rho})$ respectively, by replacing
$\rho*\iota_{\epsilon}$ and $\tilde{\rho}*\iota_{\epsilon}$ by $\rho$
and $\tilde{\rho}$. Thus we can rewrite (\ref{eq: expona}) as 
\begin{align*}
\limsup_{N\rightarrow\infty}\frac{1}{N}\log\widetilde{\mathbb{Q}}_{N}\left[\mathscr{K}\right]\\
\le\sup_{\tilde{\rho}\in\mathscr{\mathscr{K\cap}D}_{color}^{m}}\inf_{\epsilon,\delta,\tilde{g},J}\max\biggl\{ & -\left(\Phi_{\tilde{g},J}(\tilde{\rho})-\Psi_{\tilde{g},J}(\tilde{\rho})+o_{\epsilon}(1)+O(\delta)\right),\\
 & \,\,\,\,\,\,\,\,\,\,\,\,\,\,\,\,\,\,\,\,\,\,\,\limsup_{N\rightarrow\infty}\frac{1}{N}\log\mathbb{P}_{N}\left[\mathscr{B}_{N,\epsilon,\delta}^{c}\right]\biggr\}.
\end{align*}
Now letting $\epsilon\rightarrow0$ and then $\delta\rightarrow0$
so that we obtain, 
\begin{equation}
\limsup_{N\rightarrow\infty}\frac{1}{N}\log\widetilde{\mathbb{Q}}_{N}\left[\mathscr{K}\right]\le-\inf_{\tilde{\rho}\in\mathscr{\mathscr{K\cap}D}_{color}^{m}}\left[\sup_{\tilde{g},J}\Bigl\{\Phi_{\tilde{g},J}(\tilde{\rho})-\Psi_{\tilde{g},J}(\tilde{\rho})\Bigr\}\right].\label{eq: step 11}
\end{equation}
Consequently, it suffices to prove the following lemma. 
\begin{lem}
\label{lem: upppp}For each $\tilde{\rho}\in\mathscr{D}_{color}^{m}$,
\begin{equation}
\sup_{\substack{\tilde{g}\in C^{1,2}([0,\,T],\,\mathbb{T})^{m}\\
J\in C^{1,2}([0,\,T],\,\mathbb{T})
}
}\left\{ \Phi_{\tilde{g},J}(\tilde{\rho})-\Psi_{\tilde{g},J}(\tilde{\rho})\right\} \ge I_{dyn}^{m}(\tilde{\rho}).\label{eq: objs}
\end{equation}
\end{lem}
\begin{proof}
We first assume that $\tilde{\rho}\in C^{1,2}([0,T]\times\mathbb{T})^{m}$.
In this case, 
\[
F_{\rho}(t,\,x)=x+\frac{1}{\lambda+1}\int_{\mathbb{T}}\nu(y-x)\rho(t,\,y)dy
\]
satisfies $\nabla F_{\rho}=\frac{\lambda+\rho}{\lambda+1}$ (\textit{cf.}
Proposition 5 of \cite{G1}) and hence invertible for each $t.$ Let
$G_{\rho}(t,\,x)$ be its inverse and then derivatives of $G_{\rho}$
are given by 
\begin{align}
\partial_{x}G_{\rho}(t,\,x) & =\frac{\lambda+1}{\lambda+\rho(t,\,x)}\label{eq:gt1}\\
\partial_{xx}G_{\rho}(t,\,x) & =-\frac{(\lambda+1)^{2}}{(\lambda+\rho(t,\,x))^{3}}\rho_{x}(t,\,x)\label{eq:gt2}\\
\partial_{t}G_{\rho}(t,\,x) & =-\frac{1}{\lambda+\rho(t,\,x)}\int_{\mathbb{T}}\nu(y-x)\rho_{t}(t,\,y)dy.\label{eq: gt}
\end{align}
For given $\tilde{f}=(f^{(1)},\,f^{(2)},\,\cdots,\,f^{(m)})^{\dagger}\in C^{1,2}([0,T]\times\mathbb{T})^{m}$,
we take corresponding $\tilde{g}$ and $J$ by 
\begin{align}
\tilde{g}(t,\,x)= & \tilde{f}(t,\,G_{\rho}(t,\,x))\label{eq: gtil}\\
J(t,\,x)= & \int_{\mathbb{T}}\frac{\sum_{c=1}^{m}\rho_{c}(t,\,y)f_{x}^{(c)}(t,\,y)}{\lambda+\rho(t,\,y)}\nu(y-x)dy.\label{eq:jtil}
\end{align}
Under these choices, we will show that 
\begin{equation}
\Phi_{\tilde{g},J}(\tilde{\rho})-\Psi_{\tilde{g},J}(\tilde{\rho})=\Lambda_{\tilde{f}}(\tilde{\rho})\label{eq:hus}
\end{equation}
where $\Lambda_{\tilde{f}}(\tilde{\rho})$ is defined in (\ref{eq: naham}).

We first compute $\Phi_{\tilde{g},J}(\tilde{\rho})$. The main trick
is to rewrite $\partial_{t}G_{\rho}(t,\,x)$ in (\ref{eq: gt}) as
\[
-\frac{1}{\lambda+\rho(t,\,x)}\int_{\mathbb{T}}\nu(y-x)\left(\rho_{t}-\frac{1}{2}\rho_{xx}\right)(t,\,y)dy-\frac{\rho_{x}(t,\,x)}{2(\lambda+\rho(t,\,x))}
\]
and then we obtain 
\begin{align}
\Phi_{\tilde{g},J} & (\tilde{\rho})\label{eq: hus1}\\
= & \int_{\mathbb{T}}\tilde{f}^{\dagger}\tilde{\rho}(T,\,x)dx-\int_{\mathbb{T}}\tilde{f}^{\dagger}\tilde{\rho}(0,\,x)dx\nonumber \\
 & -\int_{0}^{T}\int_{\mathbb{T}}\left[\tilde{f}_{t}-\frac{\rho_{x}}{2(\lambda+\rho)}\tilde{f}_{x}+\frac{\lambda}{2(\lambda+\rho)}\tilde{f}_{xx}\right]\tilde{\rho}(t,\,x)dxdt\nonumber \\
 & +\int_{0}^{T}\int_{\mathbb{T}}\left[\frac{1}{\lambda+\rho}\int_{\mathbb{T}}\nu(y-x)\left(\rho_{t}-\frac{1}{2}\rho_{xx}\right)(t,\,y)dy\right]\tilde{f}_{x}^{\dagger}\tilde{\rho}(t,\,x)dxdt\nonumber \\
 & +\int_{\mathbb{T}}J\cdot\rho(T,\,x)dx-\int_{\mathbb{T}}J\cdot\rho(0,\,x)dx-\int_{0}^{T}\int_{\mathbb{T}}\left[J_{t}+\frac{1}{2}J_{xx}\right]\rho(t,\,x)dxdt.\nonumber 
\end{align}
With our choice of $J$, the third and fourth lines cancel each other
by the integration by part. 

To simplify $\Psi_{\tilde{g},J}(\tilde{\rho})$, we start from an
observation that $J$ satisfies
\[
J_{x}(t,\,x)=-K_{\tilde{\rho},\tilde{g}}(t)+\frac{\sum_{c=1}^{m}f_{x}^{(c)}\rho}{\lambda+\rho}(t,\,x)
\]
where $K_{\tilde{\rho},\tilde{g}}(t)$ is defined in (\ref{eq:APPPP2}).
This enable us to compute 
\begin{align}
\Psi_{\tilde{g},J}(\tilde{\rho})\label{eq:hus2}\\
=\int_{0}^{T}\int_{\mathbb{T}} & \Biggl[\frac{\lambda^{2}}{2(\lambda+\rho)^{2}}\sum_{c=1}^{m}\left(f_{x}^{(c)}\right)^{2}\rho_{c}+\frac{\lambda}{(\lambda+\rho)^{2}}\left(\sum_{c=1}^{m}f_{x}^{(c)}\rho_{c}\right)^{2}\nonumber \\
 & \,\,\,+\frac{\rho}{2(\lambda+\rho)^{2}}\left(\sum_{c=1}^{m}f_{x}^{(c)}\rho_{c}\right)^{2}+\frac{\lambda}{2(\lambda+\rho)^{2}}\sum_{c=1}^{m}\left(f_{x}^{(c)}\right)^{2}\rho\rho_{c}\nonumber \\
 & \,\,\,-\frac{\lambda}{2(\lambda+\rho)^{2}}\left(\sum_{c=1}^{m}f_{x}^{(c)}\rho_{c}\right)^{2}\Biggr]dxdt\nonumber \\
=\int_{0}^{T}\int_{\mathbb{T}} & \frac{1}{2}\nabla\tilde{f}^{\dagger}A(\tilde{\rho})\nabla\tilde{f}dxdt.\nonumber 
\end{align}
Now, (\ref{eq: hus1}) and (\ref{eq:hus2}) complete the proof of
(\ref{eq:hus}).

For general $\tilde{\rho}\in\mathscr{D}_{color}^{m}$, we will approximate
$\tilde{\rho}$ by $\tilde{\rho}_{\epsilon}$. For given $\tilde{f}\in C^{1,2}$,
we can take $\tilde{g}_{\epsilon}$ and $J_{\epsilon}$ as (\ref{eq: gtil})
and (\ref{eq:jtil}) which correspond to $\tilde{\rho}_{\epsilon}$
instead of $\tilde{\rho}$. Then, by the previous step, we have
\[
\Phi_{\tilde{g}_{\epsilon},J_{\epsilon}}(\tilde{\rho}_{\epsilon})-\Psi_{\tilde{g}_{\epsilon},J_{\epsilon}}(\tilde{\rho}_{\epsilon})=\Lambda_{\tilde{f}}(\tilde{\rho}_{\epsilon}).
\]
Note that $\Phi_{\tilde{g}_{\epsilon},J_{\epsilon}}(\tilde{\rho}_{\epsilon})-\Phi_{\tilde{g}_{\epsilon},J_{\epsilon}}(\tilde{\rho})$
and $\Psi_{\tilde{g}_{\epsilon},J_{\epsilon}}(\tilde{\rho}_{\epsilon})-\Psi_{\tilde{g}_{\epsilon},J_{\epsilon}}(\tilde{\rho})$
are $o_{\epsilon}(1)$ since 
\begin{align*}
\left|G_{\rho_{\epsilon}}(t,\,F_{\rho}(t,\,x))-x\right| & =\left|G_{\rho_{\epsilon}}(t,\,F_{\rho}(t,\,x))-G_{\rho_{\epsilon}}(t,\,F_{\rho_{\epsilon}}(t,\,x))\right|\\
 & \le\frac{\lambda+1}{\lambda}\left|F_{\rho}(t,\,x)-F_{\rho_{\epsilon}}(t,\,x)\right|\\
 & \le\frac{1}{\lambda}||\rho_{\epsilon}(t,\,\cdot)-\rho(t,\,\cdot)||_{L_{1}}
\end{align*}
for each $t$ because $\partial_{x}G_{\rho_{\epsilon}}=\frac{\lambda+1}{\lambda+\rho_{\epsilon}}<\frac{\lambda+1}{\lambda}$.
Moreover, we have $\Lambda_{\tilde{f}}(\tilde{\rho}_{\epsilon})\rightarrow\Lambda_{\tilde{f}}(\tilde{\rho})$
as $\epsilon\rightarrow0$ as in the proof of Theorem \ref{thm: LSC}
and hence 
\[
\sup_{\substack{\tilde{g}\in C^{1,2}([0,\,T],\,\mathbb{T})^{m}\\
J\in C^{1,2}([0,\,T],\,\mathbb{T})
}
}\left\{ \Phi_{\tilde{g},J}(\tilde{\rho})-\Psi_{\tilde{g},J}(\tilde{\rho})\right\} \ge\Lambda_{\tilde{f}}(\tilde{\rho}).
\]
holds for each $\tilde{f}\in C^{1,2}$ and $\tilde{\rho}\in\mathscr{D}_{color}^{m}$.
Consequently, we can complete the proof by taking supremum over $\tilde{f}$,
\end{proof}
Heretofore, we have established the large deviation upper bound for
compact sets with the rate function $I_{dyn}^{m}(\cdot)$, but we
can easily improve this result to the rate function $I_{color}^{m}(\cdot)$
where the argument may depend on the initial configuration. Moreover,
since we have the exponential tightness by Theorem \ref{thm: SUPER EXP TIGHTNESS},
the upper bound also holds for closed sets.
\begin{thm}
Under Assumptions 2 and 3, $\{\widetilde{\mathbb{Q}}_{N}\}_{N=1}^{\infty}$
satisfies the large deviation upper bound with the good rate function
$I_{color}^{m}(\cdot)$ and scale $N$. More precisely, for every
closed set $\mathscr{C}\in C([0,\,T],\,\mathscr{M}(\mathbb{T})^{m})$,
we have
\[
\limsup_{N\rightarrow\infty}\frac{1}{N}\log\widetilde{\mathbb{Q}}_{N}\left[\mathscr{C}\right]\le-\inf_{\tilde{\rho}(\cdot,\,x)dx\in\mathscr{C}}I_{color}^{m}(\tilde{\rho}).
\]

\end{thm}

\subsection{Lower Bound\label{sub:Lower-Bound}}

\subsubsection{Perturbed Process }

In general, the large deviation lower bound for the interacting particle
system can be derived by observing the limit behavior of a suitably
perturbed system. Stating this succinctly, the lower bound can be
obtained by computing the relative entropy of such a perturbed process
with respect to the original process. Thus, we should start by carefully
defining the perturbations. 

Basically, we perturb our system in two ways. First, if the color
of particle $x_{i}^{N}(\cdot)$ is $c,$ then we add the drift $b_{c}(t,\,x_{i}^{N}(\cdot))$
to this particle. Note that the drift function depends on the color
of the particle. Second, we change the jump rate between different
colors. In the original process, we have a jump process $M_{ij}^{N}(t)$
along the local time $A_{ij}^{N}(t)$, which is the Poisson jump process
with a constant intensity $\lambda N$. We will also change this jump
rate to $\lambda N+\gamma_{c_{1},c_{2}}(t,\,x_{i}^{N}(t))$ if the
color of particles $x_{i}^{N}(\cdot)$ and $x_{j}^{N}(\cdot)$ are
$c_{i}$ and $c_{j}$, respectively. Then our perturbations can be
summarized by the $m$-dimensional vector $\tilde{b}$ and an $m\times m$
matrix $\tilde{\Gamma}$, where
\begin{align*}
\tilde{b}(t,\,x) & =\left(b_{1}(t,\,x),\,b_{2}(t,\,x),\,\cdots,\,b_{m}(t,\,x)\right)^{\dagger}\\
\tilde{\Gamma}(t,\,x) & =\left\{ \gamma_{c_{1},c_{2}}(t,\,x):1\le c_{1},\,c_{2}\le m\right\} 
\end{align*}
both of which should satisfy the following conditions:
\begin{enumerate}
\item $\tilde{b}$ and $\tilde{\Gamma}$ are smooth. 
\item $\tilde{b}(t,\,x)\equiv0$ and $\tilde{\Gamma}(t,\,x)\equiv0$ for
$t\in[0,\,\eta]$ for some $\eta>0$.
\item $\tilde{\Gamma}$ is skew-symmetric : $\gamma_{c_{1},c_{2}}=-\gamma_{c_{2},c_{1}}$
and $\gamma_{c,c}=0$. \end{enumerate}
\begin{rem}
The third condition is not artificial in that changing $\gamma_{c_{1},c_{2}}$
and $\gamma_{c_{2},c_{1}}$ by the same amount or the presence of
$\gamma_{c,\,c}$ does not affect the dynamic of $\tilde{\mu}^{N}(\cdot)$;
thus, we can assume the skew-symmetry of $\tilde{\Gamma}$ without
loss of generality. 
\end{rem}
Let $\mathscr{P}_{0}$ be the set of all $(\tilde{b},\,\tilde{\Gamma})$
which satisfies all of these conditions. For each $(\tilde{b},\,\tilde{\Gamma})\in\mathscr{P}_{0}$,
a canonical way to describe the perturbed process is the martingale
formulation. Indeed, we can understand this process by the measure
$\mathbb{P}_{N}^{\tilde{b},\tilde{\Gamma}}$ on $C([0,\,T],\,\mathbb{T}^{N})$
such that for any $f\in\bar{C}(G_{N})$,
\begin{align}
M_{f}^{\tilde{b},\tilde{\Gamma}}(t)= & f(x^{N}(t))-f(x^{N}(0))-\frac{1}{2}\int_{0}^{t}\Delta f(x^{N}(s))ds\label{eq:tanaka for driven system}\\
 & -\sum_{c=1}^{m}\sum_{i\in I_{c}^{N}}\int_{0}^{t}b_{c}(t,\,x_{i}^{N}(s))\nabla_{i}f(x^{N}(s))ds\nonumber \\
 & -\sum_{\substack{1\le c_{1},\,c_{2}\le m\\
i\in I_{c_{1}}^{N},\,j\in I_{c_{2}}^{N}
}
}\int_{0}^{t}\mathfrak{U}_{ij}^{\lambda,\tilde{\Gamma}}f(x^{N}(s))dA_{ij}^{N}(s)\nonumber 
\end{align}
where 
\[
\mathfrak{U}_{ij}^{\lambda,\tilde{\Gamma}}f(x)=D_{ij}f(x)-(\lambda N+\gamma_{c_{1},c_{2}})(f_{ij}(x)-f_{ji}(x))
\]
is a martingale with respect to the original filtration. The martingale
$M_{f}^{\tilde{b},\tilde{\Gamma}}(t)$ also can be represented as
(\ref{eq:martingale}). Remark here that the rigorous existence and
uniqueness of this perturbed process are due to Girsanov's Theorem. 

The perturbed process $\mathbb{P}_{N}^{\tilde{b},\tilde{\Gamma}}$
is not too far from the original process $\mathbb{P}_{N}$ in the
following sense. 
\begin{lem}
\label{lem: GIRSANOV}For each $(\tilde{b},\,\tilde{\Gamma})\in\mathscr{P}_{0}$
and $1\le p\le2$, 
\[
\frac{1}{N}\log\mathbb{E}_{N}\left[\left(\frac{d\mathbb{P}_{N}^{\tilde{b},\tilde{\Gamma}}}{d\mathbb{P}_{N}}\right)^{p}\right]\le C+O\left(\frac{1}{N}\right)
\]
where the constant $C$ could possibly depend only on $p,\,\tilde{b},\,\tilde{\Gamma}$.\end{lem}
\begin{proof}
By Girsanov's Theorem, 
\begin{equation}
\frac{d\mathbb{P}_{N}^{\tilde{b},\tilde{\Gamma}}}{d\mathbb{P}_{N}}=\exp\left\{ \sum_{c=1}^{m}\sum_{i\in I_{N}^{c}}U_{i}^{(c)}+\sum_{\substack{1\le c_{1},\,c_{2}\le m,\,i\in I_{c_{1}}^{N},\,j\in I_{c_{2}}^{N}}
}V_{ij}^{(c_{1},\,c_{2})}\right\} \label{eq: GIRSANOV}
\end{equation}
where 
\begin{align*}
U_{i}^{(c)}= & \int_{0}^{T}b_{c}(t,\,x_{i}^{N}(t))d\beta_{i}(t)-\frac{1}{2}\int_{0}^{T}b_{c}^{2}(t,\,x_{i}^{N}(t))dt\\
V_{ij}^{(c_{1},\,c_{2})}= & \int_{0}^{T}\log\left(1+\frac{\gamma_{c_{1},c_{2}}(t,\,x_{i}^{N}(t))}{\lambda N}\right)\left[dM_{ij}^{N}(t)+\lambda NdA_{ij}^{N}(t)\right]\\
 & -\int_{0}^{T}\gamma_{c_{1},c_{2}}(t,\,x_{i}^{N}(t))dA_{ij}^{N}(t)
\end{align*}
under $\mathbb{P}_{N}$. Since $(\tilde{b},\,\tilde{\Gamma})\in\mathscr{P}_{0}$,
we obtain 
\[
\mathbb{E}_{N}\left[\left(\frac{d\mathbb{P}_{N}^{\tilde{b},\tilde{\Gamma}}}{d\mathbb{P}_{N}}\right)^{p}\right]\le e^{CN\left(p(p-1)+O\left(\frac{1}{N}\right)\right)}\mathbb{E}_{N}\exp\left\{ Cp(p-1)NA^{N}(T)\right\} 
\]
for some constant $C$ only depending on $\tilde{b}$ and $\tilde{\Gamma}$.
The last expectation can be controlled by (\ref{eq:Estimate of Total Local Time})
and we are done. 
\end{proof}
An important implication of this lemma is the following corollary. 
\begin{cor}
\label{cor: replacemtne lemma for driven system}Theorems \ref{thm: REPLACEMENT_colored}
and \ref{thm: SUPER EXP TIGHTNESS} are still valid under $\mathbb{P}_{N}^{\tilde{b},\tilde{\Gamma}}$
instead of $\mathbb{P}_{N}$ for any $(\tilde{b},\,\tilde{\Gamma})\in\mathscr{P}_{0}$.
\end{cor}

\subsubsection{Limit Theory of Perturbed Process\label{sub:Hydrodynamic-Limit-for}}

Let $\widetilde{\mathbb{Q}}_{N}^{\tilde{b},\tilde{\Gamma}}$ be the
probability measure on $C([0,\,T],$$\mathbb{\mathscr{M}}(\mathbb{T})^{m})$
induced by $\tilde{\mu}^{N}(\cdot)$ under the process $\mathbb{P}_{N}^{\tilde{b},\tilde{\Gamma}}$.
Then, $\widetilde{\mathbb{Q}}_{N}^{\tilde{b},\tilde{\Gamma}},\,N\in\mathbb{N}$
is a tight sequence because of Corollary \ref{cor: replacemtne lemma for driven system}.
Now, we can characterize all limit points of this sequence as the
solution of a certain quasi-linear PDE. 
\begin{thm}
\label{thm : HDL for Driven System} Suppose that $(\tilde{b},\,\tilde{\Gamma})\in\mathscr{P}_{0}$
and $\tilde{\mu}^{N}(0)\rightharpoonup\tilde{\gamma}(x)dx$ weakly
as $N\rightarrow\infty$. Then, the support of any weak limit of $\widetilde{\mathbb{Q}}_{N}^{\tilde{b},\tilde{\Gamma}},\,N\in\mathbb{N}$
is concentrated on the set of $\tilde{\rho}(t,\,x)\in\mathscr{D}_{color}^{m}$
which is the weak solution of 
\begin{equation}
\frac{\partial\tilde{\rho}}{\partial t}=\frac{1}{2}\nabla\cdot\left[D(\tilde{\rho})\nabla\tilde{\rho}\right]-\nabla\cdot\left[A(\tilde{\rho})\left(\tilde{b}-\frac{1}{\lambda}\tilde{\Gamma}\tilde{\rho}\right)\right]\label{eq: HDL DRIVEN EQUATION}
\end{equation}
with initial condition $\tilde{\gamma}(x)$.
\end{thm}
Let $\widetilde{\mathbb{Q}}_{\infty}^{\tilde{b},\tilde{\Gamma}}$
be a weak limit of $\widetilde{\mathbb{Q}}_{N}^{\tilde{b},\tilde{\Gamma}},\,N\in\mathbb{N}$.
Then $\widetilde{\mathbb{Q}}_{\infty}^{\tilde{b},\tilde{\Gamma}}$
is concentrated on $\mathscr{D}_{color}^{m}$ due to Lemmas \ref{lem: Domain of Rate function}
and \ref{lem: GIRSANOV}. We start by studying the limit of the uncolored
empirical density $\mu^{N}(\cdot)$ which is no more the solution
of the heat equation. 
\begin{lem}
\label{lem: driven rho} Let $\tilde{\rho}(\cdot,\,x)dx$ be any weak
limit point of \textup{$\left\{ \tilde{\mu}^{N}(\cdot)\right\} _{N=1}^{\infty}$}.
Then, $\rho=\sum_{c=1}^{m}\rho_{c}$ satisfies 
\begin{equation}
\frac{\partial\rho}{\partial t}=\frac{1}{2}\Delta\rho-\nabla(\tilde{b}\cdot\tilde{\rho})\label{eq:hysyr}
\end{equation}
in a weak sense. \end{lem}
\begin{proof}
By the Ito's formula, 
\begin{align*}
 & \frac{1}{N}\sum_{i=1}^{N}f(T,\,x_{i}^{N}(T))-\frac{1}{N}\sum_{i=1}^{N}f(0,\,x_{i}^{N}(0))\\
 & =\frac{1}{N}\int_{0}^{T}\sum_{i=1}^{N}\left\{ f_{t}+b_{c(i)}f_{x}+\frac{1}{2}f_{xx}\right\} (t,\,x_{i}^{N}(t))dt+\frac{1}{N}\int_{0}^{T}\sum_{i=1}^{N}f_{x}(t,\,x_{i}^{N}(t))d\beta_{i}(t)
\end{align*}
where $c(i)$ is the color of particle $x_{i}^{N}(\cdot)$. Then (\ref{eq:hysyr})
is straightforward since the last term is negligible.
\end{proof}

\begin{proof}[\textit{Proof of Theorem \ref{thm : HDL for Driven System}}]
The main machinery is again $z_{i}^{N}(t)$ in (\ref{eq:z(t)}).
However, we should be careful since $z_{i}^{N}(t)$ is not a martingale
under $\mathbb{P}_{N}^{\tilde{b},\tilde{\Gamma}}$ but instead satisfies
\begin{align}
dz_{i}^{N}(t)= & \,d\mathscr{M}_{i}^{N}(t)+\frac{1}{N(\lambda+1)}\sum_{c=1}^{m}\sum_{j\in I_{c}^{N}}b_{c}(t,\,x_{j}^{N}(t))dt\label{eq: z_i (t) for driven system}\\
 & \,+\frac{\lambda}{\lambda+1}b_{c_{0}}(t,\,x_{i}^{N}(t))dt+\frac{1}{(\lambda+1)}\sum_{c=1}^{m}\gamma_{c_{0},c}(t,\,x_{i}^{N}(t))dA_{i,c}^{N}(t)\nonumber 
\end{align}
where $c_{0}$ is the color of the particle $x_{i}^{N}(t)$ and $\mathscr{M}_{i}^{N}(t)$
is the martingale given by (\ref{eq: z(t) as martingale}) which was
just $z_{i}^{N}(t)$ under $\mathbb{P}_{N}$. For given 
\[
\tilde{g}=(g^{(1)},\,g^{(2)},\,\cdots,\,g^{(m)})^{\dagger}\in C^{1,2}([0,\,T]\times\mathbb{T}^{m})
\]
we can apply Ito's formula such that 
\begin{equation}
\frac{1}{N}\sum_{c=1}^{m}\sum_{i\in I_{c}^{N}}g^{(c)}(T,\,z_{i}^{N}(T))-\frac{1}{N}\sum_{c=1}^{m}\sum_{i\in I_{c}^{N}}g^{(c)}(0,\,z_{i}^{N}(0))=\Theta_{1}+\Theta_{2}+\Theta_{3}+\Theta_{4}\label{eq:mach}
\end{equation}
where
\begin{align*}
\Theta_{1} & =\int_{0}^{T}\frac{1}{N}\sum_{c=1}^{m}\sum_{i\in I_{c}^{N}}g_{t}^{(c)}(t,\,z_{i}^{N}(t))dt\\
\Theta_{2} & =\int_{0}^{T}\frac{1}{N}\sum_{c=1}^{m}\sum_{i\in I_{c}^{N}}g_{x}^{(c)}(t,\,z_{i}^{N}(t))d\mathscr{M}_{i}^{N}(t)\\
\Theta_{3} & =\int_{0}^{T}\frac{1}{2N}\sum_{c=1}^{m}\sum_{i\in I_{c}^{N}}g_{xx}^{(c)}(t,\,z_{i}^{N}(t))d\left\langle \mathscr{M}_{i}^{N},\mathscr{M}_{i}^{N}\right\rangle {}_{t}\\
\Theta_{4} & =\int_{0}^{T}\frac{1}{N}\sum_{c=1}^{m}\sum_{i\in I_{c}^{N}}g_{x}^{(c)}(t,\,z_{i}^{N}(t))\left[\mathbf{a}_{i}(x^{N}(t))dt+\sum_{k=1}^{m}\frac{\gamma_{c,k}(t,\,x_{i}^{N}(t))}{\lambda+1}dA_{i,k}^{N}(t)\right]
\end{align*}
where
\[
\mathbf{a}_{i}(x^{N}(t))=\frac{1}{N(\lambda+1)}\sum_{k=1}^{m}\sum_{j\in I_{k}^{N}}b_{k}(t,\,x_{j}^{N}(t))dt+\frac{\lambda}{\lambda+1}b_{c}(t,\,x_{i}^{N}(t))dt.
\]
We first claim that $\Theta_{2}$ is negligible since the order of
the quadratic variation is $O(1/N)$. In the formula (\ref{eq: QV of z(t)})
for $\mathscr{M}_{i}^{N}(t)$, the Brownian part is easy to compute.
For the quadratic variation of the Poisson part, we only need to check
\begin{equation}
\mathbb{E}_{N}^{\tilde{b},\tilde{\Gamma}}\left[A^{N}(T)\right]=\mathbb{E}_{N}^{\tilde{b},\tilde{\Gamma}}\left[\frac{1}{N^{2}}\sum_{i\neq j}A_{ij}^{N}(T)\right]\le C\label{eq:bdd of dr}
\end{equation}
for some $C$ where $\mathbb{E}_{N}^{\tilde{b},\tilde{\Gamma}}$ denotes
the expectation with respect to $\mathbb{P}_{N}^{\tilde{b},\tilde{\Gamma}}$.
To prove (\ref{eq:bdd of dr}), let us define $R_{N}(x)=\frac{1}{N}\sum_{i\neq j}g(x_{i}-x_{j})$
for $x\in\mathbb{\mathbb{T}}^{N}$ where $g(z)=\frac{z(1-z)}{2}\in C(\mathbb{T})$.
Then, by Tanaka's formula (\ref{eq:tanaka for driven system}), 
\begin{align*}
 & R_{N}(x^{N}(T))-R_{N}(x^{N}(0))+\frac{T(N-1)}{2}-\frac{2}{N}\sum_{i\neq j}A_{ij}^{N}(T)\\
 & -\frac{1}{N}\sum_{i\neq j}\int_{0}^{T}b_{c(i)}(t,\,x_{i}^{N}(t))g'(x_{i}^{N}(t)-x_{j}^{N}(t))dt\\
 & =\frac{1}{N}\sum_{i=1}^{N}\int_{0}^{N}\left[\sum_{j:j\neq i}g'(x_{i}^{N}(t)-x_{j}^{N}(t))\right]d\beta_{i}(t)
\end{align*}
and we can check (\ref{eq:bdd of dr}) by simply taking the expectation. 

Now we substitute $d\left\langle \mathscr{M}_{i}^{N},\mathscr{M}_{i}^{N}\right\rangle _{t}$
in $\Theta_{3}$ by (\ref{eq: QV of z(t)}) and then apply the replacement
lemma for the perturbed process (Corollary \ref{cor: replacemtne lemma for driven system})
to mollify the local times in $\Theta_{3}$ and $\Theta_{4}$ by local
densities. By doing so, we obtain 
\begin{align*}
 & \limsup_{\epsilon\rightarrow0}\limsup_{N\rightarrow\infty}\\
 & \mathbb{P}_{N}^{\tilde{b},\tilde{\Gamma}}\Biggl[\Biggl|\frac{1}{N}\sum_{c=1}^{m}\sum_{i\in I_{c}^{N}}\Biggl\{ g^{(c)}(T,\,z_{i}^{N}(T))-\frac{1}{N}\sum_{c=1}^{m}\sum_{i\in I_{c}^{N}}g^{(c)}(0,\,z_{i}^{N}(0))\\
 & -\int_{0}^{T}\left(g_{t}^{(c)}+\frac{\lambda(\lambda+\rho_{\epsilon,i}(x^{N}(t)))}{2(\lambda+1)^{2}}g_{xx}^{(c)}+\mathbf{j}_{i}(x^{N}(t))g_{x}^{(c)}\right)(t,\,z_{i}^{N}(t))dt\Biggr\}\Biggr|>\delta\Biggr]\\
 & =0
\end{align*}
where 
\[
\mathbf{j}_{i}(x^{N}(t))=\mathbf{a}_{i}(x^{N}(t))+\sum_{k=1}^{m}\frac{\gamma_{c,k}(t,\,x_{i}^{N}(t))}{\lambda+1}\rho_{\epsilon,i}^{(k)}(x^{N}(t)).
\]
Now we represent all the terms as a function of the density fields
of $\tilde{\mu}_{N}(\cdot)$ and then send $N\rightarrow\infty$ along
the subsequence of $\mathbb{N}$ along which $\widetilde{\mathbb{Q}}_{N}^{\tilde{b},\tilde{\Gamma}}\rightharpoonup\widetilde{\mathbb{Q}}_{\infty}^{\tilde{b},\tilde{\Gamma}}$
weakly. Then, 
\begin{equation}
\limsup_{\epsilon\rightarrow0}\widetilde{\mathbb{Q}}_{\infty}^{\tilde{b},\tilde{\Gamma}}\left[\tilde{\rho}:\left|\mathfrak{I}_{\tilde{g},\tilde{\rho}}(T)-\mathfrak{I}_{\tilde{g},\tilde{\rho}}(0)-\int_{0}^{T}\mathfrak{K}_{\tilde{g},\tilde{\rho},\epsilon}(t)dt\right|>\delta\right]=0\label{eq: CKNA}
\end{equation}
where 
\begin{align*}
\mathfrak{I}_{\tilde{g},\tilde{\rho}}(t)= & \left\langle \tilde{\rho}(t,\,x)dx,\,\tilde{g}(t,\,F_{\rho}(t,\,x))\right\rangle \\
\mathfrak{K}_{\tilde{g},\tilde{\rho},\epsilon}(t)= & \Biggl\langle\tilde{\rho}(t,\,x)dx,\,\Biggl(\tilde{g}_{t}+\frac{\lambda\left(\lambda+\rho*\iota_{\epsilon}(x)\right)}{2(\lambda+1)^{2}}\tilde{g}_{xx}\Biggr)\left(t,\,F_{\rho}(t,\,x)\right)+\tilde{\mathbf{k}}_{\tilde{\rho},\epsilon}(t,\,x)\Biggr\rangle.
\end{align*}
Here, $\tilde{\mathbf{k}}_{\tilde{\rho},\epsilon}=(\mathbf{k}_{\tilde{\rho},\epsilon}^{(1)},\,\mathbf{k}_{\tilde{\rho},\epsilon}^{(2)},\,\cdots,\,\mathbf{k}_{\tilde{\rho},\epsilon}^{(m)})^{\dagger}$
is defined by
\begin{align*}
 & \mathbf{k}_{\tilde{\rho},\epsilon}^{(c)}(t,\,x)\\
 & =\frac{g_{x}^{(c)}(t,\,F_{\rho}(t,\,x))}{\lambda+1}\Biggl[\int_{\mathbb{T}}\tilde{b}(t,\,y)^{\dagger}\tilde{\rho}(t,\,y)dy+\lambda b_{c}(t,\,x)+\sum_{k=1}^{m}\gamma_{c,k}(t,\,x)\rho_{k}*\iota_{\epsilon}(x)\Biggl]
\end{align*}
for $c=1,\,2,\,\cdots,\,m$.

The final step is to substitute $\tilde{g}(t,\,x)=\tilde{f}(t,\,G_{\rho}(t,\,x))$
where $G_{\rho}=F_{\rho}^{-1}$ is the function defined in Lemma \ref{lem: upppp}.
Of course, this is possible only for $\tilde{\rho}$ is regular enough.
However for general $\tilde{\rho}\in\mathscr{D}_{color}^{m}$, we
can use $\tilde{g}_{\epsilon}(t,\,x)=\tilde{f}(t,\,G_{\rho_{\epsilon}}(t,\,x))$
instead and then send $\epsilon\rightarrow0$ at the final stage to
obtain the desired result as in Lemma \ref{lem: upppp}. We will not
repeat this procedure here. 

For $\tilde{\rho}\in C^{1,2}$, we can compute various derivatives
of $\tilde{g}$ in terms of those of $\tilde{f}$ by using (\ref{eq:gt1}),
(\ref{eq:gt2}) and (\ref{eq: gt}). Furthermore, we can explicitly
compute (\ref{eq: gt}) by using Lemma \ref{lem: driven rho} in a
way that
\begin{align*}
\frac{\partial}{\partial t}G_{\rho}(t,\,x)= & -\frac{1}{\lambda+\rho(t,\,x)}\int_{\mathbb{T}}\nu(y-x)\rho_{t}(t,\,y)dy\\
= & -\frac{1}{\lambda+\rho(t,\,x)}\int_{\mathbb{T}}v(y-x)\left\{ \frac{1}{2}\Delta\rho(t,\,y)-\nabla\left[\tilde{b}(t,\,y)\cdot\tilde{\rho}(t,\,y)\right]\right\} dy\\
= & -\frac{\rho_{x}(t,\,x)}{2(\lambda+\rho(t,\,x))}+\frac{\tilde{b}\cdot\tilde{\rho}(t,\,x)-\int_{\mathbb{T}}\tilde{b}(t,\,y)\cdot\tilde{\rho}(t,\,y)dy}{(\lambda+\rho)}
\end{align*}
where we integrated by part at the last equality. By letting $\epsilon\rightarrow0$
at (\ref{eq: CKNA}), we obtain 
\[
\widetilde{\mathbb{Q}}_{\infty}^{\tilde{b},\tilde{\Gamma}}\left[\left\{ \tilde{\rho}(t,\,x)dx:\int_{\mathbb{T}}\tilde{f}\cdot\tilde{\rho}(T,\,x)dx-\int_{\mathbb{T}}\tilde{f}\cdot\tilde{\rho}(0,\,x)dx-\mbox{\ensuremath{\mathbf{H}}}_{1}-\mbox{\ensuremath{\mathbf{H}}}_{2}=0\right\} \right]=1
\]
where 
\begin{align*}
\mbox{\ensuremath{\mathbf{H}}}_{1} & =\int_{0}^{T}\int_{\mathbb{T}}\Biggl(\tilde{f}_{t}+\frac{\lambda}{2(\lambda+\rho)}\tilde{f}_{xx}-\frac{\rho_{x}(\lambda+2\rho)}{2(\lambda+\rho)^{2}}\tilde{f}_{x}\Biggr)(t,\,x)\cdot\tilde{\rho}(t,\,x)dx\\
\mbox{\ensuremath{\mathbf{H}}}_{2} & =\sum_{c=1}^{m}\int_{0}^{T}\int_{\mathbb{T}}\left[\frac{\lambda b_{c}+\sum_{k=1}^{m}\rho_{k}b_{k}}{\lambda+\rho}+\sum_{k=1}^{m}\frac{\gamma_{c,k}\rho_{k}}{\lambda+\rho}\right]\rho_{c}f_{x}^{(c)}(t,\,x)dx
\end{align*}
By performing the integration by part, we can rewrite $\mathbf{H}_{2}$
as 
\begin{equation}
-\int_{0}^{T}\int_{\mathbb{T}}\tilde{f}^{\dagger}\nabla\cdot\left[A(\tilde{\rho})\left(\tilde{b}-\frac{1}{\lambda}\tilde{\Gamma}\tilde{\rho}\right)\right]dx.\label{eq:hdlt}
\end{equation}
Note that we used the skew-symmetry of $\tilde{\Gamma}$ here. This
completes the proof, since $\mbox{\ensuremath{\mathbf{H}}}_{1}$ and
$\mbox{\ensuremath{\mathbf{H}}}_{2}$ correspond to $\frac{1}{2}\nabla\cdot\left[D(\tilde{\rho})\nabla\tilde{\rho}\right]$
and $-\nabla\cdot\left[A(\tilde{\rho})\left(\tilde{b}-\frac{1}{\lambda}\tilde{\Gamma}\tilde{\rho}\right)\right]$
respectively in (\ref{eq: HDL DRIVEN EQUATION}).
\end{proof}

\subsubsection{Uniqueness and Approximation Procedure\label{sub:Uniqueness}}

The lower bound computation, based on the limit theory of the perturbed
system presented in the previous subsection, also requires the uniqueness
of PDE (\ref{eq: HDL DRIVEN EQUATION}). Let $\mathscr{D}_{0}^{m}$
consist of $\tilde{\rho}$ satisfying $I_{color}^{m}(\tilde{\rho})<\infty$
then $\mathscr{D}_{0}^{m}\subset\mathscr{D}_{color}^{m}$ by Lemma
\ref{lem: Domain of Rate function}. If we can prove the uniqueness
of (\ref{eq: HDL DRIVEN EQUATION}) for the class of $\mathscr{D}_{0}^{m}$,
then we can directly compute the lower bound. Of course, the uniqueness
of a quasi-linear PDE such as (\ref{eq: HDL DRIVEN EQUATION}) whose
diffusion coefficient is not elliptic is hard to achieve at the desired
level of generality. Instead, we establish a somewhat narrower uniqueness
result, which should entail an additional approximation theorem. Thus,
Theorem \ref{thm: Uniqueness Theorem } gives the uniqueness result
and Theorem \ref{thm: apploxmatn} provides the corresponding approximation
procedure. We remark here that our methodology in the current subsection
originates from and is similar to the methodology described in Sections
5 and 6 of \cite{QRV}; hence, some details, especially related to
the approximation procedure, are common to all of theses sections
and will be omitted. 

Let a subclass $\mathscr{E}_{0}^{m}$ of $\mathscr{D}_{0}^{m}$ be
the collection of $\tilde{\rho}$ which is smooth on $(0,\,T]\times\mathbb{T}$,
solves (\ref{eq: HDL DRIVEN EQUATION}) for some $(\tilde{b},\,\tilde{\Gamma})\in\mathscr{P}_{0}$
and satisfies 
\begin{align}
\min_{1\le c\le m}\,\inf_{(t,\,x)\in[\eta,\,T]\times\mathbb{T}}\rho_{c}(t,\,x)>\epsilon & \,\,\,\,\text{for\,\ some\ }\epsilon>0\label{eq: bdddd1}
\end{align}
where $\eta$ comes from the second condition of $\mathscr{P}_{0}$.
Then, we can state the uniqueness theorem as following theorem. 
\begin{thm}
\label{thm: Uniqueness Theorem }Suppose that $\tilde{u}\in\mathscr{E}_{0}^{m}$
is a solution of (\ref{eq: HDL DRIVEN EQUATION}) for \textup{$(\tilde{b},\,\tilde{\Gamma})\in\mathscr{P}_{0}$}
with initial condition $\gamma_{0}(x)$ which satisfies (\ref{eq:finiteness of entropy}).
If $\tilde{v}\in\mathscr{D}_{0}^{m}$ is another solution of the same
equation with the same initial condition, then $\tilde{v}=\tilde{u}$.
\end{thm}
Since the diffusion matrix is not symmetric, the usual technique based
on the propagation of the Sobolev norm of $\tilde{v}-\tilde{u}$ is
not available here. Instead, we examine the relative entropy of $\tilde{v}$
with respect to $\tilde{u}$ which requires $\tilde{v}\in L_{\infty}([\eta,\,T]\times\mathbb{T}^{m})$.
This boundedness does not automatically follow from the membership
of $\mathscr{D}_{0}^{m}$ and therefore, we require an independent
argument to demonstrate this.

Let $v=\sum_{c=1}^{m}v_{c}$ where $\tilde{v}=(v_{1},\,v_{2},\,\cdots,\,v_{m})^{\dagger}$,
then it is enough to show $v\in L_{\infty}([\eta,\,T]\times\mathbb{T})$.
First note that $v$ is the solution of the heat equation in $[0,\,\eta]$
and therefore $v(\eta,\,\cdot)$ is a bounded function. In $[\eta,\,T],$
we can add each coordinates of (\ref{eq: HDL DRIVEN EQUATION}) to
obtain the equation for $v$:
\begin{equation}
v_{t}=\frac{1}{2}\Delta v-\nabla\left(\sum_{c=1}^{m}b_{c}v_{c}\right)=\frac{1}{2}\Delta v+\nabla(bv)\label{eq: driven_total}
\end{equation}
where 
\[
b=\frac{1}{v}\sum_{c=1}^{m}b_{c}v_{c}\in L_{\infty}([\eta,\,T]\times\mathbb{T}).
\]
Therefore, we obtain $v\in L_{\infty}([\eta,\,T]\times\mathbb{T})$
from the following lemma. 
\begin{lem}
\label{lem: boundedness lemma}Suppose that $w$ is the weak solution
of 
\begin{equation}
\frac{\partial w}{\partial t}=\frac{1}{2}\Delta w+\nabla(bw)\label{driven PDE}
\end{equation}
with the bounded non-negative initial condition $w_{0}(x)$. If $b\in L_{\infty}([0,\,T]\times\mathbb{T})$,
then $w\in L_{\infty}([0,\,T]\times\mathbb{T})$.\end{lem}
\begin{proof}
We first extend the equation to $\mathbb{R}$. More precisely, we
periodically extend $b$ to $\mathbb{R}$ and call it $\hat{b}$ and
then, consider the equation 
\begin{equation}
\frac{\partial\hat{w}}{\partial t}=\frac{1}{2}\Delta\hat{w}+\nabla\left[\hat{b}\hat{w}\right]\label{eq:extended driven PDE}
\end{equation}
where the initial condition is $w_{0}(x)$ for $0\le x\le1$ and $0$
otherwise. To analyze (\ref{eq:extended driven PDE}), let us consider
the diffusion 
\[
dX_{t}=dW_{t}-\hat{b}(t,\,X_{t})dt
\]
on $\mathbb{R}$ where $W_{t}$ is standard Brownian motion under
the Wiener measure $P$. Note that the existence and uniqueness of
$X_{t}$ are guaranteed by Girsanov's Theorem. Then (\ref{eq:extended driven PDE})
is the forward equation for $X_{t}$ and therefore $\hat{w}(\cdot,\,\cdot)$
can be represented as 
\[
\hat{w}(t,\,y)=\int_{0}^{1}p(0,\,x;\,t,\,y)w_{0}(x)dx
\]
where the $p(0,\,x;\,t,\,y)$ is the transition kernel of $X_{t}.$
To compute this kernel, we assume that the Brownian motion under $P$
starts from $x$ and then consider a probability measure $Q$ on $C([0,\,T],\,\mathbb{R})$
defined by 
\[
\frac{dQ}{dP}=\exp\left\{ \int_{0}^{T}\hat{b}(s,\,X_{s})dW_{s}-\frac{1}{2}\int_{0}^{T}\hat{b}^{2}(s,\,X_{s})ds\right\} 
\]
so that $X_{t}$ is a Brownian motion starting from $x$ under $Q$.
Then, 
\[
P(X_{t}\in[y,\,y+dy])=\mathbb{E}^{Q}\left[\mathds{1}_{X_{t}\in[y,\,y+dy]}e^{-\int_{0}^{t}\hat{b}(s,\,X_{s})dX_{s}-\frac{1}{2}\int_{0}^{t}\hat{b}^{2}(s,\,X_{s})ds}\right]
\]
and therefore the kernel can be written as 
\begin{align}
 & p(0,\,x;\,t,\,y)\label{eq:kernel_l}\\
 & =q(0,\,x;\,t,\,y)\mathbb{E}^{Q}\exp\left\{ -\int_{0}^{t}\hat{b}(s,\,Z_{s}^{x,y})dZ_{s}^{x,y}-\frac{1}{2}\int_{0}^{t}\hat{b}^{2}(s,\,Z_{s}^{x,y})ds\right\} .\nonumber 
\end{align}
where 
\[
q(0,\,x;t,\,y)=\frac{1}{\sqrt{2\pi t}}\exp\frac{-(y-x)^{2}}{2t}
\]
is the standard heat kernel and $\left\{ Z_{s}^{x,y},\,\mathscr{F}_{s}\right\} _{s\le t}$
is the 1D Brownian bridge connecting $x$ at time $0$ and $y$ at
time $t$ under $Q$. 

Our aim is to estimate the kernel $p$ by using (\ref{eq:kernel_l}).
Observe that $Z_{s}^{x,y}$ satisfies $dZ_{s}^{x,y}=\frac{y-Z_{s}^{x,y}}{t-s}ds+dW_{s}$
where $\{W_{s}\}_{s\le t}$ is a Brownian motion under $Q$. Therefore,
we have
\begin{equation}
\mathbb{E}^{Q}\exp\left\{ -\int_{0}^{t}\hat{b}(s,\,Z_{s}^{x,y})dZ_{s}^{x,y}-\frac{1}{2}\int_{0}^{t}\hat{b}^{2}(s,\,Z_{s}^{x,y})ds\right\} \le A_{1}^{\frac{1}{2}}A_{2}^{\frac{1}{2}}\label{eq:uuta}
\end{equation}
where
\begin{align*}
A_{1} & =\mathbb{E}^{Q}\exp\left\{ -2\int_{0}^{t}\hat{b}(s,\,Z_{s}^{x,y})dW_{s}-\int_{0}^{t}\hat{b}^{2}(s,\,Z_{s}^{x,y})ds\right\} \\
A_{2} & =\mathbb{E}^{Q}\exp\left\{ -2\int_{0}^{t}\hat{b}(s,\,Z_{s}^{x,y})\frac{y-Z_{s}^{x,y}}{t-s}ds\right\} 
\end{align*}
It is easy to see that $A_{1}$ is bounded by $\exp\{T||b||_{\infty}^{2}\}$.
For $A_{2},$ note that $Z_{s}^{x,y}$ has an alternative expression
$Z_{s}^{x,y}=\frac{x(t-s)+ys}{t}+\left(t-s\right)\overline{W}_{\frac{s}{t(t-s)}}$
where $\left\{ \overline{W}_{s}\right\} _{s\ge0}$ is another Brownian
motion and hence, we can bound $A_{2}$ as
\begin{equation}
A_{2}\le e^{C|x-y|}\mathbb{E}^{Q}\exp\left\{ C\int_{0}^{t}\left|\overline{W}_{\frac{s}{t(t-s)}}\right|ds\right\} .\label{eq:uu1}
\end{equation}
where $C$ could possibly depend on $b$ only. Now we have to estimate
the expectation in (\ref{eq:uu1}). By Jensen's inequality, 
\begin{align}
 & \mathbb{E}^{Q}\exp\left\{ C\int_{0}^{t}\left|\overline{W}_{\frac{s}{t(t-s)}}\right|ds\right\} \label{eq:uu2}\\
 & =\mathbb{E}^{Q}\exp\left\{ \int_{0}^{t}\frac{1}{2\sqrt{t}\sqrt{t-s}}2C\sqrt{t}\sqrt{t-s}\left|\overline{W}_{\frac{s}{t(t-s)}}\right|ds\right\} \nonumber \\
 & \le\int_{0}^{t}\frac{1}{2\sqrt{t}\sqrt{t-s}}\mathbb{E}^{Q}\exp\left\{ 2C\sqrt{t}\sqrt{t-s}\left|\overline{W}_{\frac{s}{t(t-s)}}\right|\right\} ds\nonumber \\
 & \le2e^{2C^{2}T}.\nonumber 
\end{align}
By (\ref{eq:uuta}), (\ref{eq:uu1}) and (\ref{eq:uu2}), we obtain
an estimate for the kernel $p$ as 
\[
p(0,\,x;\,t,\,y)\le C_{1}e^{C_{2}|x-y|}q(0,\,x;\,t,\,y)
\]
where constant $C_{1}$, $C_{2}$ only depend on $b,\,T$. This kernel
estimates implies the uniform boundedness of $w$.
\end{proof}

\begin{proof}[Proof of Theorem \ref{thm: Uniqueness Theorem }]
Since the equation is linear parabolic with smooth coefficients on
$[0,\,\eta]$, the uniqueness is automatic at there. Thus, it suffices
to establish the uniqueness on $[\eta,\,T]\times\mathbb{T}$. By Lemma
\ref{lem: boundedness lemma}, we know that not only $u$, $v$ but
also $u_{c}$, $v_{c}$ for all $c$ are uniformly bounded by some
number $M>0$ in this region. We can define the relative entropy at
time $t$ such a manner that  
\[
H(t)=\int_{\mathbb{T}}\sum_{c=1}^{m}v_{c}(t,\,x)\log\frac{v_{c}(t,\,x)}{u_{c}(t,\,x)}dx
\]
then by the elementary property of the relative entropy, 
\begin{equation}
H(t)\ge\int_{\mathbb{T}}\sum_{c=1}^{m}\left\{ \sqrt{v_{c}(t,\,x)}-\sqrt{u_{c}(t,\,x)}\right\} ^{2}dx\ge\frac{1}{4M}K(t)\label{eq:bount1}
\end{equation}
where $K(t)=\int_{\mathbb{T}}\sum_{c=1}^{m}\left\{ v_{c}(t,\,x)-u_{c}(t,\,x)\right\} ^{2}dx$.
Note that $H(\eta)=0$ and therefore we can compute $H(t)$ as 
\begin{align}
H(t) & =\int_{\eta}^{t}\int_{\mathbb{T}}\partial_{t}\left[\sum_{c=1}^{m}v_{c}(s,\,x)\log\frac{v_{c}(s,\,x)}{u_{c}(s,\,x)}\right]dxds\label{eq: Entt1}\\
 & =\int_{\eta}^{t}\int_{\mathbb{T}}\left[\log\frac{\tilde{v}}{\tilde{u}}\right]^{\dagger}\partial_{t}\tilde{v}-\left(\frac{\tilde{v}}{\tilde{u}}\right)^{\dagger}\partial_{t}\tilde{u}dxds\nonumber 
\end{align}
where
\[
\log\frac{\tilde{v}}{\tilde{u}}=\left(\log\frac{v_{1}}{u_{1}},\,\log\frac{v_{2}}{u_{2}},\,\cdots,\,\log\frac{v_{m}}{u_{m}}\right)^{\dagger}\,\,\,\text{and}\,\,\,\frac{\tilde{v}}{\tilde{u}}=\left(\frac{v_{1}}{u_{1}},\,\frac{v_{2}}{u_{2}},\,\cdots,\,\frac{v_{m}}{u_{m}}\right).^{\dagger}
\]
Now, we replace $\partial_{t}\tilde{u}$ and $\partial_{t}\tilde{v}$
by the RHS of (\ref{eq: HDL DRIVEN EQUATION}) and then apply integration
by part. At this point, the only object that we cannot control is
$\nabla\tilde{v}$ and therefore we should simplify the result to
the following form:
\[
\int_{\eta}^{t}\int_{\mathbb{T}}-\left\Vert A\nabla\tilde{v}+B\right\Vert ^{2}+Cdxds.
\]
If we carry out such a computation, then the result is given by 
\begin{align}
 & -\frac{1}{2}\int_{\eta}^{t}\int_{\mathbb{T}}\left\Vert \mathbf{S}(\tilde{v})^{\frac{1}{2}}\nabla\tilde{v}+\mathbf{S}(\tilde{v})^{-\frac{1}{2}}\left[\mathbf{U^{-}}\tilde{b}-\mathbf{G}-\frac{1}{2}\mathbf{U}^{+}\chi(\tilde{u})\nabla\tilde{u}\right]\right\Vert ^{2}dxds\label{eq: entt2}\\
 & +\frac{1}{2}\int_{\eta}^{t}\int_{\mathbb{T}}\left\Vert \mathbf{S}(\tilde{v})^{-\frac{1}{2}}\left[\mathbf{U}^{-}\tilde{b}-\mathbf{G}+\frac{1}{2}\mathbf{U}^{-}\chi(\tilde{u})\nabla\tilde{u}\right]\right\Vert ^{2}dxds\nonumber 
\end{align}
where
\begin{align*}
\mathbf{U^{\pm}} & =\chi(\tilde{v})A(\tilde{v})\pm\chi(\tilde{u})A(\tilde{u})\\
\mathbf{S}(\tilde{v}) & =\chi(\tilde{v})A(\tilde{v})\chi(\tilde{v})\\
\mathbf{G} & =\frac{1}{\lambda}\left[\chi(\tilde{v})A(\tilde{v})\tilde{\Gamma}\tilde{v}-\chi(\tilde{u})A(\tilde{u})\tilde{\Gamma}\tilde{u}\right].
\end{align*}
Now, we will ignore the first term in (\ref{eq: entt2}). For the
second term, note first that each elements of $\mathbf{U^{-}}$ and
$\mathbf{G}$ are bounded by $C\sum_{c=1}^{m}|u_{c}-v_{c}|$ for some
constant $C$. Moreover $\left|\chi(\tilde{u})\nabla\tilde{u}\right|$
is uniformly bounded by (\ref{eq: bdddd1}) and $\mathbf{S}(\tilde{v})^{-1}\le\frac{M(\lambda+M)}{\lambda}I_{m}$
where $I_{m}$ is $m\times m$ identity matrix. Thus (\ref{eq: entt2})
is bounded by $C\int_{\eta}^{t}K(s)ds$ for some constant $C$. Thus,
the uniqueness follows from Grownall's Lemma. 
\end{proof}
Since our uniqueness theorem is not for the class of $\mathscr{D}_{0}^{m}$
but instead for $\mathscr{E}_{0}^{m}$, we need an additional approximation
procedure. Since the rate function is lower semicontinuous by Theorem
\ref{thm: LSC}, it is enough to establish the following theorem. 
\begin{thm}
\label{thm: apploxmatn}For each $\tilde{\rho}\in\mathscr{D}_{0}^{m}$,
we can find a sequence $\bigl\{\tilde{\rho}^{(k)}\bigr\}_{k=1}^{\infty}\subset\mathscr{E}_{0}^{m}$
such that $\tilde{\rho}^{(k)}(0,\,x)=\tilde{\rho}(0,\,x)$ for all
$k$, \textup{$\tilde{\rho}^{(k)}\rightharpoonup\tilde{\rho}$} weakly
and
\[
\limsup_{k\rightarrow\infty}I_{dyn}^{m}(\tilde{\rho}^{(k)})\le I_{dyn}^{m}(\tilde{\rho}).
\]

\end{thm}
In general, this procedure is not difficult if the rate function is
convex. Unfortunately, within the context of our work, the rate function
is not convex and requires careful analysis. For this purpose, we
adopted the general method suggested in \cite{QRV}, where comprehensive
details can be found. Therefore, we only outline the whole procedure
here; and additionally highlight selected points that do not directly
follow from their result, due to the difference between our model
and the SSEP. 

Our strategy is to divide the approximation into three steps as $\mathscr{E}_{0}^{m}\subset\mathscr{E}_{1}^{m}\subset\mathscr{E}_{2}^{m}\subset\mathscr{D}_{0}^{m}$,
where the two intermediate classes $\mathscr{E}_{1}^{m}$ and $\mathscr{E}_{2}^{m}$
are explained now. The subclass $\mathscr{E}_{2}^{m}$ consists of
$\tilde{\rho}\in\mathscr{D}_{0}^{m}$ that satisfies $\frac{\partial\tilde{\rho}}{\partial t}=\frac{1}{2}\nabla\left[D(\tilde{\rho})\nabla\rho\right]$
for $t\in[0,\,\eta]$ for some $\eta>0$. The membership of $\mathscr{E}_{1}^{m}$
additionally requires that for some $\alpha>0$, $\rho_{c}(t,\,x)\ge\alpha\rho(t,\,x)$
holds for all $x\in\mathbb{T},\,1\le c\le m$ and $t\ge\eta'$ for
some $0<\eta'<\eta$. 

The first step is to approximate $\mathscr{D}_{0}^{m}$ by $\mathscr{E}_{2}^{m}$
and which is Theorem 6.2 of \cite{QRV}. The strategy is to estimate
$\tilde{\rho}\in\mathscr{D}_{0}^{m}$ by $\tilde{\rho}^{(\eta)}\in\mathscr{E}_{2}^{m}$
defined by 
\[
\tilde{\rho}^{(\eta)}(t,\,x)=\begin{cases}
\tilde{R}(t,\,x) & \text{ for }0\le t\le\eta\\
\tilde{R}(2\eta-t,\,x) & \text{ for }\eta\le t\le2\eta\\
\tilde{\rho}(t-2\eta,\,x) & \text{ for }2\eta\le t\le T.
\end{cases}
\]
where $\tilde{R}$ is the solution of $\partial_{t}\tilde{R}=\frac{1}{2}\nabla\cdot\left[D(\tilde{R})\nabla\tilde{R}\right]$
with initial condition $\tilde{\rho}(0,\,x)$. We refer the proof
in \cite{QRV}. 

The second step is to approximate $\mathscr{E}_{2}^{m}$ by $\mathscr{E}_{1}^{m}$
and this step corresponds to the Theorem 6.3 of \cite{QRV}. For this
step, we first select a smooth increasing function $e:[0,\,T]\rightarrow\mathbb{R}$
satisfying $e\equiv0$ on $[0,\,\eta_{1}]$ and $e\equiv1$ on $[\eta_{2},\,T]$
for some $0<\eta_{1}<\eta_{2}<\eta$. Then we can approximate $\tilde{\rho}\in\mathscr{E}_{2}^{m}$
by 
\[
\rho_{c}^{(k)}(t,\,x)=\left(1-\frac{e(t)}{k}\right)\rho_{c}(t,\,x)+\frac{e(t)\bar{\rho}_{c}}{k}\rho(t,\,x)\in\mathscr{E}_{1}^{m}.
\]
One can find a proof of this step in \cite{QRV} as well but we present
a little bit simpler one. 

Let us define 
\[
I_{dyn}^{m}(\tilde{R\,};[\eta_{2},\,T])=\frac{1}{2}\int_{\eta_{2}}^{T}\left\Vert \frac{\partial\tilde{R}}{\partial t}-\frac{1}{2}\nabla\cdot\left[D(\tilde{R})\nabla\tilde{R}\right]\right\Vert _{-1,\,A(\tilde{R})}^{2}dt
\]
and then it suffices to show 
\begin{equation}
\limsup_{k\rightarrow\infty}I_{dyn}^{m}\left(\tilde{\rho}^{(k)};[\eta_{2},\,T]\right)\le I_{dyn}^{m}\left(\tilde{\rho};[\eta_{2},\,T]\right)\label{eq:tltse}
\end{equation}
since we can choose $\eta_{1}$ to arbitrarily close number to $\eta_{2}$.
It is easy to see that the rate function $I_{dyn}^{m}(\cdot\,;[\eta_{2},\,T])$
is convex on the set
\[
D_{\rho}=\left\{ \tilde{R}\in\mathscr{D}_{0}^{m}:\sum_{c=1}^{m}R_{c}(t,\,x)=\rho(t,\,x)\,\,\,\forall(t,\,x)\in[0,\,T]\times\mathbb{T}\right\} .
\]
Since $\tilde{\rho}^{(k)}(t,\,x)=\left(1-\frac{1}{k}\right)\tilde{\rho}(t,\,x)+\frac{1}{k}\hat{\rho}(t,\,x)$
on $t\ge\eta_{2}$ where
\[
\hat{\rho}(t,\,x)=(\bar{\rho}_{1}\rho(t,\,x),\,\bar{\rho}_{2}\rho(t,\,x),\,\cdots,\,\bar{\rho}_{m}\rho(t,\,x))^{\dagger}\in D_{\rho}
\]
we have
\begin{equation}
I_{dyn}^{m}\left(\tilde{\rho}^{(k)};[\eta_{2},\,T]\right)\le\left(1-\frac{1}{k}\right)I_{dyn}^{m}\left(\tilde{\rho};[\eta_{2},\,T]\right)+\frac{1}{k}I_{dyn}^{m}\left(\hat{\rho};[\eta_{2},\,T]\right)\label{eq: tats}
\end{equation}
due to convexity. We can easily check that 
\[
I_{dyn}^{m}\left(\hat{\rho};[\eta_{2},\,T]\right)=\int_{\eta_{2}}^{T}\left\Vert \rho_{t}-\frac{1}{2}\Delta\rho\right\Vert _{-1,\rho}^{2}dt<\infty
\]
and therefore (\ref{eq:tltse}) directly follows from (\ref{eq: tats}). 

The last step is to approximate $\mathscr{E}_{1}^{m}$ by $\mathscr{E}_{0}^{m}$.
In \cite{QRV}, this step has been carried out by Theorem 6.4, which
consists of Lemmas 6.5, 6.6, 6.7 and 6.8. In particular, Lemmas 6.5,
6.7 and 6.8 are quite robust and we can apply their arguments directly
to our model as well. It would therefore suffice to show that a similar
to Lemma 6.6 of \cite{QRV} is valid for our model. This is verified
by the following lemma.
\begin{lem}
\label{lem: customized Lemma}Suppose that $r$ and $\rho$ are non-negative
weakly differentiable functions on $\mathbb{T}$ satisfying
\begin{equation}
\int_{\mathbb{T}}\frac{\left|\nabla\rho\right|^{2}}{\rho}dx<\infty\,\,\,\text{ and }\,\,\,\int_{\mathbb{T}}\frac{\left|\nabla r\right|^{2}}{(\lambda+\rho)r}dx<\infty\label{eq: bound KC2}
\end{equation}
and $r\le\rho$. Then, $\left\{ \frac{\left|\nabla r_{\epsilon}\right|^{2}}{(\lambda+\rho_{\epsilon})r_{\epsilon}}\right\} _{\epsilon>0}$
is a uniformly integrable family on $\mathbb{T}$. \end{lem}
\begin{proof}
Notice that $\frac{\left|\nabla r_{\epsilon}(x)\right|^{2}}{(\lambda+\rho_{\epsilon}(x))r_{\epsilon}(x)}\le2\left(A_{1}+A_{2}\right)$
where
\begin{align*}
A_{1} & =\frac{1}{r_{\epsilon}}\left[\int_{\mathbb{T}}\frac{r(x+y)}{\sqrt{\lambda+\rho(x+y)}}\nabla\phi_{\epsilon}(y)dy\right]^{2}\\
A_{2} & =\frac{1}{r_{\epsilon}}\left[\int_{\mathbb{T}}\left(\frac{1}{\sqrt{\lambda+\rho_{\epsilon}(x)}}-\frac{1}{\sqrt{\lambda+\rho(x+y)}}\right)r(x+y)\nabla\phi_{\epsilon}(y)dy\right]^{2}.
\end{align*}
We can bound $A_{1}$ as 
\[
A_{1}=\frac{1}{r_{\epsilon}}\left(\nabla\frac{r}{\sqrt{\lambda+\rho}}\right)_{\epsilon}^{2}\le\left[\frac{1}{r}\left(\nabla\frac{r}{\sqrt{\lambda+\rho}}\right)^{2}\right]_{\epsilon}\le\left[2\frac{\left|\nabla r\right|^{2}}{(\lambda+\rho)r}+\frac{\left|\nabla\rho\right|^{2}}{2\lambda\rho}\right]_{\epsilon}
\]
and hence this part is uniformly integrable by (\ref{eq: bound KC2}). 

By applying Cauchy-Schwarz's inequality to $A_{2}$, we obtain 
\begin{align*}
A_{2} & \le\int_{\mathbb{T}}\left(\frac{1}{\sqrt{\lambda+\rho_{\epsilon}(x)}}-\frac{1}{\sqrt{\lambda+\rho(x+y)}}\right)^{2}r(x+y)\frac{\left(\nabla\phi_{\epsilon}(y)\right)^{2}}{\phi_{\epsilon}(y)}dy\\
 & \le\frac{1}{\lambda}\int_{\mathbb{T}}\left(\sqrt{\rho_{\epsilon}(x)}-\sqrt{\rho(x+y)}\right)^{2}\frac{\left(\nabla\phi_{\epsilon}(y)\right)^{2}}{\phi_{\epsilon}(y)}dy.
\end{align*}
Therefore, $A_{2}\le\frac{2}{\lambda}(B_{1}+B_{2})$ where 
\begin{align*}
B_{1} & =\int_{\mathbb{T}}\left(\sqrt{\rho_{\epsilon}(x)}-\left(\sqrt{\rho(x)}\right)_{\epsilon}\right)^{2}\frac{\left(\nabla\phi_{\epsilon}(y)\right)^{2}}{\phi_{\epsilon}(y)}dy\\
B_{2} & =\int_{\mathbb{T}}\left(\left(\sqrt{\rho(x)}\right)_{\epsilon}-\sqrt{\rho(x+y)}\right)^{2}\frac{\left(\nabla\phi_{\epsilon}(y)\right)^{2}}{\phi_{\epsilon}(y)}dy.
\end{align*}
Since $\int_{\mathbb{T}}\frac{\left(\nabla\phi_{\epsilon}(y)\right)^{2}}{\phi_{\epsilon}(y)}dy=\frac{C}{\epsilon^{2}}$
for some constant $C$, we can bound $B_{1}$ and $B_{2}$ as
\begin{align}
B_{1} & \le\frac{C}{\epsilon^{2}}\int_{\mathbb{T}}\int_{\mathbb{T}}\left(\sqrt{\rho(x+z)}-\sqrt{\rho(x+w)}\right)^{2}\phi_{\epsilon}(z)\phi_{\epsilon}(w)dzdw\label{eq: uii1}\\
B_{2} & \le\int_{\mathbb{T}}\int_{\mathbb{T}}\left(\sqrt{\rho(x+z)}-\sqrt{\rho(x+y)}\right)^{2}\phi_{\epsilon}(z)\frac{\left(\nabla\phi_{\epsilon}(y)\right)^{2}}{\phi_{\epsilon}(y)}dydz\label{eq:uii2}
\end{align}
respectively. Since $\sqrt{\rho}\in H^{1}(\mathbb{T})$ by (\ref{eq: bound KC2}),
we can conclude that RHSs of (\ref{eq: uii1}) and (\ref{eq:uii2})
are uniformly integrable by Lemma 6.5 of \cite{QRV}
\end{proof}

\subsubsection{Proof of Lower Bound\label{sub:Proof-of-Lower}}

Now we are ready to establish the large deviation lower bound for
$\bigl\{\widetilde{\mathbb{Q}}_{N}\bigr\}_{N=1}^{\infty}$. 
\begin{thm}
Under Assumptions 2 and 3, $\bigl\{\widetilde{\mathbb{Q}}_{N}\bigr\}_{N=1}^{\infty}$
satisfies the large deviation lower bound with the rate function $I_{color}^{m}(\cdot)$.
In other words, for any $\tilde{\rho}\in\mathscr{D}_{0}^{m}$ and
its neighborhood $\mathscr{O}$, we have 
\begin{equation}
-I_{color}^{m}(\tilde{\rho})\le\liminf_{N\rightarrow\infty}\frac{1}{N}\log\widetilde{\mathbb{Q}}_{N}\left[\mathscr{O}\right].\label{eq: lower bddd}
\end{equation}
\end{thm}
\begin{proof}
Thanks to Theorem \ref{thm: apploxmatn}, it suffices to prove (\ref{eq: lower bddd})
for $\tilde{\rho}\in\mathscr{E}_{0}^{m}$. For such a $\tilde{\rho},$
we can find a smooth function $U(t,\,x)$ on $[0,\,T]\times\mathbb{T}$
satisfying 
\[
\frac{\partial\tilde{\rho}}{\partial t}=\frac{1}{2}\nabla\cdot\left[D(\tilde{\rho})\nabla\tilde{\rho}\right]-\nabla\left[A(\tilde{\rho})\nabla U\right]
\]
and $I_{dyn}^{m}(\tilde{\rho})=\frac{1}{2}\int_{0}^{T}\int_{\mathbb{T}}\nabla U^{\dagger}A(\tilde{\rho})\nabla U$.

We first assume that $\tilde{\rho}(0,\,x)=\tilde{\rho}_{0}(x)$ so
that $I_{color}^{m}(\tilde{\rho})=I_{dyn}^{m}(\tilde{\rho})$. We
define $\Sigma_{\tilde{\rho},U}$ by 
\[
\Sigma_{\tilde{\rho},U}=\left\{ (\tilde{b},\,\tilde{\Gamma})\in\mathscr{P}_{0}:\tilde{b}-\frac{1}{\lambda}\tilde{\Gamma}\tilde{\rho}=\nabla U\right\} .
\]
Then, by Theorems \ref{thm : HDL for Driven System} and \ref{thm: Uniqueness Theorem },
we have 
\begin{equation}
\lim_{N\rightarrow\infty}\mathbb{P}_{N}^{\tilde{b},\tilde{\Gamma}}(\tilde{\mu}^{N}(\cdot)\in\mathscr{\mathscr{O}})=1.\label{eq: llnn}
\end{equation}
for each $(\tilde{b},\,\tilde{\Gamma})\in\Sigma_{\tilde{\rho},U}$.
Then we can estimate $\frac{1}{N}\log\widetilde{\mathbb{Q}}_{N}\left[\mathscr{O}\right]$
such that 
\begin{equation}
\liminf_{N\rightarrow\infty}\frac{1}{N}\log\widetilde{\mathbb{Q}}_{N}\left[\mathscr{O}\right]\ge-\inf_{(\tilde{b},\,\tilde{\Gamma})\in\Sigma_{\tilde{\rho},U}}\limsup_{N\rightarrow\infty}\mathbb{E}_{N}^{\tilde{b},\tilde{\Gamma}}\left[\frac{1}{N}\log\frac{d\mathbb{P}_{N}^{\tilde{b},\tilde{\Gamma}}}{d\mathbb{P}_{N}}\right].\label{eq: MINN}
\end{equation}
by the standard argument, e.g., Chapter 10.5 of \cite{KL}. 

Now, we compute the RHS of (\ref{eq: MINN}). The first step is to
recall Girsanov's formula (\ref{eq: GIRSANOV}) to deduce 
\begin{equation}
\frac{1}{N}\log\frac{d\mathbb{P}_{N}^{\tilde{b},\tilde{\Gamma}}}{d\mathbb{P}_{N}}=\frac{1}{N}\sum_{c=1}^{m}\sum_{i\in I_{N}^{c}}U_{i}^{(c)}+\frac{1}{N}\sum_{\substack{\substack{1\le c_{1},\,c_{2}\le m\\
i\in I_{N}^{c_{1}},\,j\in I_{N}^{c_{2}}
}
}
}V_{ij}^{(c_{1},\,c_{2})}\label{eq:sejin}
\end{equation}
where
\begin{align*}
U_{i}^{(c)} & =\int_{0}^{T}b_{c}(t,\,x_{i}^{N}(t))\left[dx_{i}^{N}(t)-d\tilde{A}_{i}^{N}(t)\right]-\frac{1}{2}\int_{0}^{T}b_{c}^{2}(t,\,x_{i}^{N}(t))dt\\
V_{ij}^{(c_{1},\,c_{2})} & =\int_{0}^{T}\log\left(1+\frac{\gamma_{c_{1},c_{2}}(t,\,x_{i}^{N}(t))}{\lambda N}\right)dJ_{ij}^{N}(t)-\int_{0}^{T}\gamma_{c_{1},c_{2}}(t,\,x_{i}^{N}(t))dA_{ij}^{N}(t)
\end{align*}
and $J_{ij}^{N}(t)$ is the jump process related with the martingale
$M_{ij}^{N}(t)$. In particular, under $\mathbb{P}_{N}^{\tilde{b},\tilde{\Gamma}}$,
\begin{align*}
dx_{i}^{N}(t) & =d\beta_{i}(t)+d\tilde{A}_{i}^{N}(t)+b_{c_{1}}(t,\,x_{i}^{N}(t))\\
dJ_{ij}^{N}(t) & =dM{}_{ij}^{N}(t)+\left(\lambda N+\gamma_{c_{1},c_{2}}(t,\,x_{i}^{N}(t))\right)dA_{ij}^{N}(t)
\end{align*}
where $c_{1}$ and $c_{2}$ are colors of particles $x_{i}^{N}(\cdot)$
and $x_{j}^{N}(\cdot)$, respectively. Therefore, (\ref{eq:sejin})
can be rewritten as 
\begin{align}
 & \frac{1}{2N}\sum_{c=1}^{m}\sum_{i\in I_{N}^{c}}\int_{0}^{T}b_{c}^{2}(t,\,x_{i}^{N}(t))dt\label{eq: khsd}\\
 & +\frac{1}{2\lambda N}\sum_{\substack{c_{1}<c_{2},\,i\in I_{N}^{c_{1}}}
}\int_{0}^{T}\gamma_{c_{1},c_{2}}^{2}(t,\,x_{i}^{N}(t))dA_{i,c_{2}}^{N}(t)+O\left(\frac{1}{N}\right).\nonumber 
\end{align}
To use the replacement lemma, we define a set $\mathscr{B}_{N}(\epsilon,\,\delta)\subset C([0,\,T],\,\mathbb{T}^{N})$
such that $x(\cdot)\in\mathscr{B}_{N}(\epsilon,\,\delta)$ if and
only if $\left|\int_{0}^{T}\mathbf{V}_{N,\epsilon}^{\tilde{\Gamma}}(t,\,x(t))dt\right|<\delta$
where 
\begin{align*}
 & \mathbf{V}_{N,\epsilon}^{\tilde{\Gamma}}(t,\,x)\\
 & =\frac{1}{N^{2}}\sum_{\substack{1\le c_{1},\,c_{2}\le m\\
i\in I_{c_{1}}^{N},\,j\in I_{c_{2}}^{N}
}
}\gamma_{c_{1},c_{2}}^{2}(t,\,x_{i})\left[\frac{1}{2\epsilon}\chi_{\epsilon}(x_{j}-x_{i})-\left(\delta^{+}(x_{j}-x_{i})+\delta^{+}(x_{i}-x_{j})\right)\right].
\end{align*}
Then, by Corollary \ref{cor: replacemtne lemma for driven system},
$\mathscr{B}_{N}(\epsilon,\,\delta)^{c}$ is super-exponentially negligible
and hence 
\begin{equation}
\limsup_{N\rightarrow\infty}\mathbb{E}_{N}^{\tilde{b},\tilde{\Gamma}}\left[\frac{1}{N}\log\frac{d\mathbb{P}_{N}^{\tilde{b},\tilde{\Gamma}}}{d\mathbb{P}_{N}}\right]=\limsup_{\epsilon\rightarrow0}\limsup_{N\rightarrow\infty}\mathbb{E}_{N}^{\tilde{b},\tilde{\Gamma}}\left[\mathds{1}_{\mathscr{B}_{N}(\epsilon,\,\delta)}\frac{1}{N}\log\frac{d\mathbb{P}_{N}^{\tilde{b},\tilde{\Gamma}}}{d\mathbb{P}_{N}}\right].\label{eq: khst}
\end{equation}
On $\mathscr{B}_{N}(\epsilon,\,\delta)$, we can approximate (\ref{eq: khsd})
by 
\begin{align}
 & \frac{1}{2N}\sum_{c=1}^{m}\sum_{i\in I_{N}^{c}}\int_{0}^{T}b_{c}^{2}(t,\,x_{i}^{N}(t))dt\label{khss}\\
 & +\frac{1}{2\lambda N}\sum_{c_{1}<c_{2},\,i\in I_{N}^{c_{1}}}\int_{0}^{T}\gamma_{c_{1},c_{2}}^{2}(t,\,x_{i}^{N}(t))\rho_{i,\epsilon}^{(c_{2})}(x_{i}^{N}(t))dt+O(\delta)+O\left(\frac{1}{N}\right).\nonumber 
\end{align}
Consequently, we can conclude from (\ref{eq: khst}) and (\ref{khss})
that 
\begin{align}
 & \limsup_{N\rightarrow\infty}\mathbb{E}_{N}^{\tilde{b},\tilde{\Gamma}}\left[\frac{1}{N}\log\frac{d\mathbb{P}_{N}^{\tilde{b},\tilde{\Gamma}}}{d\mathbb{P}_{N}}\right]\label{eq:sjin}\\
 & =\frac{1}{2}\int_{0}^{T}\int_{\mathbb{T}}\left\{ \sum_{c=1}^{m}b_{c}^{2}\rho_{c}+\frac{1}{\lambda}\sum_{c_{1}<c_{2}}\gamma_{c_{1},c_{2}}^{2}\rho_{c_{1}}\rho_{c_{2}}(t,\,x)\right\} dxdt\nonumber 
\end{align}
since $\tilde{\mu}^{N}(t)\rightharpoonup\tilde{\rho}(t,\,x)dx$ by
Theorem \ref{thm : HDL for Driven System} and \ref{thm: Uniqueness Theorem }. 

To complete the calculation of the RHS of (\ref{eq: MINN}), we optimize
(\ref{eq:sjin}) over $(\tilde{b},\,\tilde{\Gamma})\in\Sigma_{\tilde{\rho},U}$.
This can be done by the Lagrange multiplier method and the optimizer
turns out to be 
\begin{align*}
\bar{b}_{c} & =\frac{\lambda}{\lambda+\rho}\nabla U_{c}+\frac{1}{\lambda+\rho}\sum_{k=1}^{m}\rho_{k}\nabla U_{k}\\
\bar{\gamma}_{c_{1},c_{2}} & =\frac{\lambda}{\lambda+\rho}(\nabla U_{c_{1}}-\nabla U_{c_{2}}).
\end{align*}
With these optimizers, the RHS of (\ref{eq:sjin}) becomes $\frac{1}{2}\int_{0}^{T}\int_{\mathbb{T}}\nabla U^{\dagger}A(\tilde{\rho})\nabla U=I_{dyn}^{m}(\tilde{\rho})$. 

By following this approach, we completed the proof when $\tilde{\rho}(0,\,x)=\tilde{\rho}^{0}(x)$.
The case for the general initial condition is also easy to obtain
by the same argument by tilting the initial configuration appropriately. 
\end{proof}
We conclude this section by summarizing the results that were obtained
for the LDP for the empirical density of colors.
\begin{thm}
\label{thm: Large deviation for Color} Under Assumptions 2 and 3,
$\bigl\{\widetilde{\mathbb{Q}}_{N}\bigr\}_{N=1}^{\infty}$ satisfies
the LDP with the good rate function $I_{color}^{m}(\cdot)$ and scale
$N$. In other words, for any measurable set $A\subset C([0,\,T],\,\mathscr{M}(\mathbb{T})^{m})$,
we have
\[
-\inf_{\tilde{\pi}\in A^{o}}I_{color}^{m}(\tilde{\pi}_{\cdot})\le\liminf_{N\rightarrow\infty}\frac{1}{N}\log\widetilde{\mathbb{Q}}_{N}(A)\le\limsup_{N\rightarrow\infty}\frac{1}{N}\log\widetilde{\mathbb{Q}}_{N}(A)\le-\inf_{\tilde{\pi}\in\bar{A}}I_{color}^{m}(\tilde{\pi}_{\cdot}).
\]

\end{thm}

\section{Empirical Process }

\subsection{Propagation of Chaos}

We start by explaining the relationship between the propagation of
chaos, which is the LLN of the empirical process, and the LLN of the
empirical density of colors in a more general set up. 

Consider the empirical process $R_{N}=\frac{1}{N}\sum_{i=1}^{N}\delta_{x_{i}^{N}(\cdot)}$
that induces a probability measure $P_{N}$ on $\mathscr{M}_{1}(C([0,\,T],\,\mathbb{T}))$.
The limit theory for $\bigl\{ P_{N}\bigr\}_{N=1}^{\infty}$ can be
obtained by verifying the tightness and identifying the unique limit
point. The tightness can be demonstrated by the general technique
introduced in \cite{R}. Regarding the identification of the limit
point, the limit theory of the empirical density of colors plays a
significant role. Suppose that the limiting particle density $\rho(t,\,x)$
is the unique solution of a certain parabolic equation $\partial_{t}\rho=\mathscr{L}\rho$
with the initial condition $\rho^{0}(x)$ under Assumption 1. Furthermore,
assume that if we color the particles by an arbitrary number of colors
such that Assumption 2 holds, then the limiting particle density of
each color $c$ denoted by $\rho_{c}$ evolves as the unique solution
of the parabolic PDE $\partial_{t}\rho_{c}=\mathscr{A}_{\rho}^{*}\rho_{c}$
with the initial condition $\rho_{c}^{0}(dx)$, where $\mathscr{A_{\rho}}$
is a time-inhomogeneous generator that could possibly depend on $\rho(t,\,x)$.
\begin{rem}
For our model, $\mathscr{L}=\frac{1}{2}\Delta$ and $\mathscr{A}_{\rho}$
is given by (\ref{eq: th generator}). 
\end{rem}
Under these assumptions, we can compute the limit of finite dimensional
marginal densities of the empirical process. For instance, we can
calculate the limiting joint density 
\begin{equation}
\lim_{N\rightarrow\infty}\mathbb{E}_{N}\left[\frac{\left|\left\{ i:x_{i}^{N}(0)\in A,\,x_{i}^{N}(t)\in B\right\} \right|}{N}\right]=\lim_{N\rightarrow\infty}\mathbb{E}_{N}\left[\sum_{i=1}^{N}\mathds{1}_{A}(x_{i}^{N}(0))\mathds{1}_{B}(x_{i}^{N}(t))\right]\label{eq: toawl}
\end{equation}
in the following manner: we color the particle $x_{i}^{N}(\cdot)$
by color $1$ if $x_{i}^{N}(0)\in A$ and by color $2$ otherwise.
If $\mu^{N}(0)\rightharpoonup\rho^{0}(dx)$, then $\mu_{1}^{N}(0)\rightharpoonup\mathds{1}_{A}(x)\rho^{0}(dx)$
and therefore, we can compute the limiting particle density $\rho_{1}(t,\,\cdot)$
of color $1$ at time $t$ by the solution of $\partial_{t}\rho_{1}=\mathscr{A}_{\rho}^{*}\rho_{1}$
with the initial condition $\mathds{1}_{A}(x)\rho^{0}(dx)$. Therefore
(\ref{eq: toawl}) can be computed as $\int_{B}\rho_{1}(t,\,x)dx$.
We can use the same method to compute the joint distribution for any
finite number of times. (see \cite{R,S} for details.)

Therefore, any limit points of $\{P_{N}\}_{N=1}^{\infty}$ should
be the diffusion process with the generator $\mathscr{A}_{\rho}$.
Consequently, we can establish the limit theory of $\{P_{N}\}_{N=1}^{\infty}$
as soon as the uniqueness and existence of such a diffusion process
with starting measure $\rho^{0}(dx)$ are valid. This general theory
can be applied to our model if the initial limiting particle density
is bounded. 
\begin{thm}
\label{thm: propagation of chaos}Suppose that $\mu^{N}(0)\rightharpoonup\rho^{0}(x)dx$
weakly for a bounded function $\rho^{0}(x)$ on $\mathbb{T}$ and
let $\rho(t,\,x)$ be the solution of the heat equation with initial
condition $\rho^{0}(x)$. Then $P_{N}\rightharpoonup\delta_{P}$ weakly
where $P$ is the unique diffusion process on $\mathscr{M}_{1}(C([0,\,T],\,\mathbb{T}))$
with the time-inhomogeneous generator $\mathscr{A}_{\rho}$ defined
by (\ref{eq: th generator}).\end{thm}
\begin{proof}
The tightness of $\bigl\{ P_{N}\bigr\}_{N=1}^{\infty}$ is a consequence
of Theorem \ref{thm: SUPER EXP TIGHTNESS} and the limit theory for
the empirical density of colors is presented by Theorem \ref{thm: HDL collor}.
The uniqueness result for the diffusion with the generator $\mathscr{A}_{\rho}$
and the starting density $\rho^{0}(x)$ which is bounded can be found
in Theorem 4 of \cite{G1}.
\end{proof}
Even though we have suggested a proof of Theorem \ref{thm: propagation of chaos}
by using the empirical density of colors as an intermediate tool,
this result was already established in \cite{G1} in a different way.
The stronger result in \cite{G1} showed the diffusive scaling limit
of one tagged particle to be the diffusion with the generator $\mathscr{A}_{\rho}$
and also showed that any two tagged particles are asymptotically independent.
Of course, these results imply the propagation of chaos for our model. 
\begin{rem}
For the general starting measure $\rho^{0}(dx)$, our methodology
is still valid for the tightness and the identification of the limit
point step. However, the uniqueness of the diffusion process with
the generator $\mathscr{A}_{\rho}$ causes a problem. If $\rho^{0}(dx)$
is a singular measure, then the uniqueness generally does not hold.
However, \cite{G2} suggested a way to circumvent this pathological
phenomenon by, roughly speaking, appropriately decomposing each mass
at a point into a left and right mass. We were also able to extend
our result to this regime. 
\end{rem}
The remaining part of this article is devoted to explaining the LDP
corresponding to Theorem \ref{thm: propagation of chaos} under Assumption
3. A methodology for the SSEP for $d\ge2$ has been developed in \cite{QRV}\footnote{The original result was valid only for $d\ge3$ but extended to $d=2$
in \cite{LOV}.} and relies on the LDP for the empirical density of colors and Dawson-G\"artner's
projective limit theory. The robustness of their method is such that
we can almost apply it directly to our model. The only thing that
has to be checked for our model is a certain class of martingale problems

\subsection{Martingale Problem\label{sub:Martingale-Problem}}

When we define the rate function $\mathscr{I}(Q)$ for the LDP of
the empirical process in the next subsection, what we need is the
perturbed diffusions with the generator $\mathscr{A}_{\rho}+b\nabla$
for an appropriate class of $b$. The existence and uniqueness of
such diffusions are not trivial and should be proven independently.
In this subsection, we carry this out with the help of the results
in \cite{QV}. 

Suppose that $\rho(t,\,x)$ is weakly continuous in time, weakly differentiable
in space and also satisfies
\begin{equation}
\int_{\mathbb{T}}\rho(0,\,x)\log\rho(0,\,x)dx<\infty\,\,\,\,\,\text{and \,\,\,\ }\int_{0}^{T}\int_{\mathbb{T}}\frac{(\nabla\rho)^{2}}{\rho}dxdt<\infty.\label{eq: DV0}
\end{equation}
Then we define a class $\mathscr{B}_{\rho}$ consisting of measurable
functions $b(t,\,x)$ on $[0,\,T]\times\mathbb{T}$ such that
\begin{equation}
\frac{\partial\rho}{\partial t}=\frac{1}{2}\Delta\rho-\nabla(b\rho)\,\,\,\,\,\text{and \,\,\,\ }\int_{0}^{T}\int_{\mathbb{T}}b^{2}\rho dxdt<\infty\label{eq: DV1}
\end{equation}
where the first equation is weak sense. 

For measurable function $c(t,\,x)$ on $[0,\,T]\times\mathbb{T}$,
we define the generator $\mathscr{A}_{\rho,c}$ by $\mathscr{A}_{\rho,c}=\mathscr{A}_{\rho}+c\nabla$
so that
\begin{equation}
\mathscr{A}_{\rho,c}=\frac{\lambda}{2(\lambda+\rho)}\Delta+\left(-\frac{(2\lambda+\rho)\nabla\rho}{2(\lambda+\rho)^{2}}+c\right)\nabla.\label{eq: DV3}
\end{equation}
Then (\ref{eq: DV1}) implies that $\rho$ satisfies $\partial_{t}\rho=\mathscr{A}_{\rho,b}^{*}\rho$
for each $b\in\mathscr{B}_{\rho}$. Basically, we want to build a
unique diffusion process with generator $\mathscr{A}_{\rho,b}$ for
$b\in\mathscr{B}_{\rho}$ with marginal density $\rho$ to define
the rate function of empirical process. However, the coefficients
of the generator $\mathscr{A}_{\rho,b}$ only have limited regularities
and therefore the existence and uniqueness in the spirit of Stroock
and Varadhan is not valid here. Although there are some results on
general coefficients (e.g., \cite{K}), these usually assume uniform
ellipticity for the generator. In our case, the diffusion coefficient
is $\frac{\lambda}{2(\lambda+\rho)}$, which may not be uniformly
elliptic since $\rho$ can be unbounded in general. For the SSEP,
Quastel and Varadhan \cite{QV} solved this difficulty by limiting
the sense of the martingale problem in a suitable fashion. They considered
the solution of the martingale problem not to start from a specific
point $x$ but from some initial distribution $p_{0}(x)$. By doing
so, they achieved a proper existence and uniqueness result in this
context. Of course, we shall follow their approach and the main result
can be stated as follows. 
\begin{thm}
\label{thm: Uniqueness and Ex of MP}Suppose that $\rho$ satisfies
(\ref{eq: DV0}) and $\mathscr{B}_{\rho}\neq\phi$. 
\begin{enumerate}
\item For each $b\in\mathscr{B}_{\rho}$, there exists the unique diffusion
process $P^{b}$ on $\mathbb{T}$ with the generator $\mathscr{A}_{\rho,b}$
with the marginal density $\rho(t,\,x)$ at each time $t\in[0,\,T]$. 
\item For each measurable function $R^{0}$ on $\mathbb{T}$ satisfying
$0\le R^{0}(\cdot)\le C\rho(0,\,\cdot)$, there exist unique diffusion
$P_{R^{0}}^{b}$ with the generator $\mathscr{A}_{\rho,b}$ and the
marginal density $R(t,\,x)$ which is the unique solution of $\partial_{t}R=\mathscr{A}_{\rho,b}^{*}R$
with initial condition $R^{0}$ and satisfies $0\le R\le C\rho$ on
$[0,\,T]\times\mathbb{T}$.
\end{enumerate}
\end{thm}
The proofs are identical to those in Section 5 of \cite{QV}. The
only obstacle when we apply the argument of \cite{QV} is the fact
that our model possibly has an unbounded density $\rho(t,\,x)$ whereas
the SSEP has an $a\,priori$ bound $1$. We can overcome this by proving
the following lemma as a substitute to Theorem 3.12 in \cite{QV}. 
\begin{lem}
\label{lem: Uniqueness} Suppose that $\rho$ satisfies (\ref{eq: DV0})
and $\mathscr{B}_{\rho}\neq\phi$. For each measurable function $R^{0}$
on $\mathbb{T}$ satisfying $0\le R^{0}(\cdot)\le C\rho(0,\,\cdot)$
for some constant $C$, there exists the unique weak solution $R(t,\,x)$
of the forward equation
\begin{equation}
\frac{\partial R}{\partial t}=\mathscr{A}_{\rho,b}^{*}R\label{eq: FRD EQN}
\end{equation}
with initial condition $R^{0}$ and satisfying $0\le R\le C\rho$
on $[0,\,T]\times\mathbb{T}$ for some constant $C$. Moreover, $R$
also satisfies the energy estimate 
\begin{equation}
\int_{0}^{T}\int_{\mathbb{T}}\frac{(\nabla R)^{2}}{(\lambda+\rho)R}dxdt\le C_{1}+C_{2}\int_{0}^{T}\int_{\mathbb{T}}\frac{\left|\nabla\rho\right|^{2}}{\rho}dxdt\label{eq: ENE UNI-1}
\end{equation}
for some constants $C_{1}$, $C_{2}$.\end{lem}
\begin{proof}
For each $b\in\mathscr{B_{\rho}}$, we define $b^{\epsilon}=\frac{(b\rho)_{\epsilon}}{\rho_{\epsilon}}$.
Then, it is easy to see that $b^{\epsilon}\in\mathscr{B}_{\rho_{\epsilon}}$
and $\rho_{\epsilon}$ is the unique weak solution of $\partial_{t}\rho_{\epsilon}=\mathscr{A}_{\rho_{\epsilon},b^{\epsilon}}^{*}\rho_{\epsilon}$.
Note that we can write $\mathscr{A}_{\rho_{\epsilon},b^{\epsilon}}^{*}$
explicitly as
\begin{equation}
\mathscr{A}_{\rho_{\epsilon},b^{\epsilon}}^{*}u=\nabla\left[\frac{\lambda}{2(\lambda+\rho_{\epsilon})}\nabla u+\left(\frac{\nabla\rho_{\epsilon}}{2(\lambda+\rho_{\epsilon})}-b^{\epsilon}\right)u\right]\label{eq:5}
\end{equation}
and it is easy to check that this generator satisfies the conditions
of Theorem 3.12 in \cite{QV}, namely,
\begin{align*}
\int_{0}^{T}\int_{\mathbb{T}}\frac{(\nabla\rho_{\epsilon})^{2}}{\rho_{\epsilon}}\times\frac{\lambda}{2(\lambda+\rho_{\epsilon})}dxdt & <\infty\\
\int_{0}^{T}\int_{\mathbb{T}}\left[\frac{\nabla\rho_{\epsilon}}{2(\lambda+\rho_{\epsilon})}-b^{\epsilon}\right]^{2}\frac{2(\lambda+\rho_{\epsilon})}{\lambda}\rho_{\epsilon}dxdt & <\infty
\end{align*}
since we have 
\begin{align}
\frac{(\nabla\rho_{\epsilon})^{2}}{\rho_{\epsilon}} & \le\left(\frac{(\nabla\rho)^{2}}{\rho}\right)_{\epsilon}\label{eq:fu1}\\
(b^{\epsilon})^{2}\rho_{\epsilon} & =\frac{(b\rho)_{\epsilon}^{2}}{\rho_{\epsilon}}\le(b^{2}\rho)_{\epsilon}\label{eq:fu2}
\end{align}
and $\rho_{\epsilon}$ is uniformly bounded by some constant $M_{\epsilon}$.
Therefore, we can apply Theorem 3.12 of \cite{QV} such that there
exists a unique solution $R^{\epsilon}$ of
\begin{equation}
\frac{\partial R^{\epsilon}}{\partial t}=\mathscr{A}_{\rho_{\epsilon},b_{\epsilon}}^{*}R^{\epsilon}\label{eq: approximated forward}
\end{equation}
with the initial condition $R_{\epsilon}^{0}(x)$ that satisfies $0\le R^{\epsilon}\le C\rho_{\epsilon}$
on $[0,\,T]\times\mathbb{T}$ as well as the energy estimate
\begin{equation}
\int_{0}^{T}\int_{\mathbb{T}}\frac{\left(\nabla R^{\epsilon}\right)^{2}}{(\lambda+\rho_{\epsilon})R^{\epsilon}}<C_{1}+C_{2}\int_{0}^{T}\int_{\mathbb{T}}\frac{(\nabla\rho)^{2}}{\rho}dxdt\label{eq: approximated energy}
\end{equation}
for some constant $C_{1},\,C_{2}$. This energy estimate can be derived
from (3.26) of \cite{QV} and (\ref{eq:fu1}). 

Our aim is to send $\epsilon$ to $0$ in (\ref{eq: approximated forward})
in a proper way. To this end, let us first prove that 
\[
\left\{ \frac{R^{\epsilon}}{\sqrt{\rho_{\epsilon}}}\right\} _{\epsilon>0},\,\,\left\{ \frac{\nabla R^{\epsilon}}{\lambda+\rho_{\epsilon}}\right\} _{\epsilon>0}\mbox{\, and \,}\left\{ R^{\epsilon}\right\} _{\epsilon>0}
\]
are uniformly bounded in $L_{2}([0,\,T]\times\mathbb{T})$, respectively.
The boundedness of the first of these terms is obvious and that of
the second term follows directly from (\ref{eq: approximated energy}).
For the last term, since $\rho\in L_{2}$ by (\ref{eq: DV0}), 
\[
\int_{0}^{T}\int_{\mathbb{T}}\left(R^{\epsilon}\right)^{2}dxdt\le C^{2}\int_{0}^{T}\int_{\mathbb{T}}\rho_{\epsilon}^{2}dxdt\le C^{2}\int_{0}^{T}\int_{\mathbb{T}}\rho^{2}dxdt
\]
by Lemma \ref{lem: L_2 boundness-1}. Therefore, we can take a subsequence
$\left\{ \epsilon_{k}\right\} _{k=1}^{\infty}$ which converges to
$0$ and also satisfies 
\[
R^{\epsilon_{k}}\rightharpoonup R,\,\,\frac{\nabla R^{\epsilon_{k}}}{\lambda+\rho_{\epsilon_{k}}}\rightharpoonup U\,\,\,\mbox{and }\,\,\frac{R^{\epsilon_{k}}}{\sqrt{\rho_{\epsilon_{k}}}}\rightharpoonup V
\]
weakly in $L_{2}$ for some $R,\,U$ and $V$, respectively. 

We now claim that 
\[
U=\frac{\nabla R}{\lambda+\rho}\,\,\,\mbox{and }\,\,V=\frac{R}{\sqrt{\rho}}.
\]
For $U$, we know that $\lambda+\rho_{\epsilon_{k}}\rightarrow\lambda+\rho$
strong in $L_{2}$ by Lemma \ref{APProximation Lemma} and $\frac{\nabla R^{\epsilon_{k}}}{\lambda+\rho_{\epsilon_{k}}}$
is uniformly bounded in $L_{2}$ by (\ref{eq: approximated energy}).
Therefore by (3) of Lemma \ref{lem: frequently used} we can verify
that $U=\frac{\nabla R}{\lambda+\rho}$. For $V$, by (2) of Lemma
\ref{lem: frequently used}, we have $\frac{R^{\epsilon_{k}}}{\sqrt{\rho_{\epsilon_{k}}}}\cdot\sqrt{\rho_{\epsilon_{k}}}\rightharpoonup V\cdot\sqrt{\rho}$
weakly in $L_{1}$ and therefore $V\sqrt{\rho}=R$ or equivalently
$V=\frac{R}{\sqrt{\rho}}$.

These weak convergences in $L_{2}$ imply that
\begin{align}
\frac{\nabla R^{\epsilon_{k}}}{\lambda+\rho_{\epsilon_{k}}} & \rightharpoonup\frac{\nabla R}{\lambda+\rho}\label{eq:1est}\\
\frac{\nabla\rho_{\epsilon_{k}}}{\lambda+\rho_{\epsilon_{k}}}R^{\epsilon_{k}} & \rightharpoonup\frac{\nabla\rho}{\lambda+\rho}R\label{eq:2est}\\
b^{\epsilon_{k}}R^{\epsilon_{k}} & \rightharpoonup bR\label{eq:3est}
\end{align}
weakly in $L_{1}$ also by Lemma \ref{lem: frequently used}. More
precisely, (\ref{eq:1est}) is derived directly from our definition
of $\left\{ \epsilon_{k}\right\} _{k=1}^{\infty}$ and (\ref{eq:2est})
holds because $\frac{\nabla\rho_{\epsilon_{k}}}{\lambda+\rho_{\epsilon_{k}}}\rightarrow\frac{\nabla\rho}{\lambda+\rho}$
strongly in $L_{2}$, due to the uniform integrability of the form
of (\ref{eq:fu1}). Similarly, (\ref{eq:3est}) is obtained as a consequence
of (2) of Lemma \ref{lem: frequently used} since we have $b^{\epsilon_{k}}\sqrt{\rho_{\epsilon_{k}}}\rightarrow b\sqrt{\rho}$
strongly in $L_{2}$ by (\ref{eq:fu2}) and $\frac{R^{\epsilon_{k}}}{\sqrt{\rho_{\epsilon_{k}}}}\rightharpoonup\frac{R}{\sqrt{\rho}}$
weakly in $L_{2}$ as we observed before. Now, (\ref{eq:1est}), (\ref{eq:2est})
and (\ref{eq:3est}) allow us to take the limit in (\ref{eq: approximated forward})
along the sequence $\left\{ \epsilon_{k}\right\} _{k=1}^{\infty}$
and by doing so we obtain $\frac{\partial R}{\partial t}=\mathscr{A}_{\rho,b}^{*}R$.
Consequently, we proved the existence.

The energy estimate (\ref{eq: ENE UNI-1}) can be obtained by repeating
the argument of Theorem 4.1 in \cite{QV}. Although this theorem requires
the $L_{\infty}$ boundedness of $R$, our bound $R\in L_{2}(0,\,T,\,L_{\infty}(\mathbb{T}))$
turns out to be sufficient for applying their argument to our specific
diffusion coefficient $\frac{\lambda}{\lambda+\rho}$. 

Finally, let us consider the uniqueness issue. Suppose that $u,\,v$
are two solutions then we consider the evolution of $\frac{(u-v)^{2}}{\rho}$,
which is a well-defined function since $0\le u,\,v\le C\rho$, such
a manner that 
\begin{align*}
\int_{\mathbb{T}}\frac{\left(u-v\right)^{2}}{\rho} & (s,\,x)dx-\int_{\mathbb{T}}\frac{\left(u-v\right)^{2}}{\rho}(0,\,x)dx\\
=\int_{0}^{s}\int_{\mathbb{T}} & -\frac{\left(u-v\right)^{2}}{\rho^{2}}\nabla\left(\frac{1}{2}\nabla\rho-b\rho\right)\\
 & +2\frac{u-v}{\rho}\nabla\left[\frac{\lambda}{2(\lambda+\rho)}\nabla(u-v)+\left(\frac{\nabla\rho}{2(\lambda+\rho)}-b\right)(u-v)\right]dxdt\\
=\int_{0}^{s}\int_{\mathbb{T}} & \nabla\left(\frac{\left(u-v\right)^{2}}{\rho^{2}}\right)\left(\frac{1}{2}\nabla\rho-b\rho\right)dxdt\\
 & -2\nabla\left(\frac{u-v}{\rho}\right)\left[\frac{\lambda}{2(\lambda+\rho)}\nabla(u-v)+\left(\frac{\nabla\rho}{2(\lambda+\rho)}-b\right)(u-v)\right]dxdt\\
=\int_{0}^{s}\int_{\mathbb{T}} & -\frac{\lambda}{\rho^{3}(\lambda+\rho)}\left[(u-v)\nabla\rho-\rho\nabla(u-v)\right]^{2}dxdt.
\end{align*}
This computation guarantees the uniqueness. 
\end{proof}

\subsection{\label{sub:Large-Deviation-Principle}Large Deviation Theory of Empirical
Process }

We start by defining the rate function for empirical process. 
\begin{defn}[Rate function for empirical process]
\label{Def}Let $Q\in\mathscr{M}_{1}(C([0,\,T],\,\mathbb{T}))$ has
the marginal density $q(t,\,x)$ which satisfies (\ref{eq: DV0}),
$\mathscr{B}_{q}\neq\phi$ and $H\left[Q\left|P^{b}\right.\right]<\infty$
for some $b\in\mathscr{B}_{q}$ where the diffusion $P^{b}$ is the
one defined in Theorem \ref{thm: Uniqueness and Ex of MP}. Then,
we can find\footnote{Alternative way to define $b_{Q}$ is the unique minimizer of the
relative entropy $H\left[Q\left|P^{b}\right.\right]$ over $b\in\mathscr{B}_{q}$.} $b_{Q}\in\mathscr{B}_{q}$ such that the corresponding diffusion
process $P^{b_{Q}}$ with marginal density $q(t,\,x)$ satisfies
\[
E^{Q}\left[\int_{0}^{T}\phi(t,\,x(t))dx(t)\right]=E^{P^{b_{Q}}}\left[\int_{0}^{T}\phi(t,\,x(t))dx(t)\right]
\]
for any smooth $\phi$ (\textit{cf.} Theorem 7.3 of \cite{QRV}).
Then, the dynamic rate function $\mathscr{I}_{dyn}(Q)$ is defined
by 
\begin{equation}
\mathscr{I}_{dyn}(Q)=H\left[Q\left|P^{b_{Q}}\right.\right]+\frac{1}{2}\int_{0}^{T}\int_{\mathbb{T}}b_{Q}^{2}qdxdt.\label{eq: The final rate function}
\end{equation}
For all the other cases, $\mathscr{I}_{dyn}(Q)$ is defined to be
infinite. In addition, due to Assumption 3, we define the Sanov-type
initial rate function $\mathscr{I}_{init}(Q)$ by
\[
\mathscr{I}_{init}(Q)=\int_{\mathbb{T}}q(0,\,x)\log\frac{q(0,\,x)}{\rho^{0}(x)}dx.
\]
Finally, the full rate function is defined by 
\[
\mathscr{I}(Q)=\mathscr{I}_{dyn}(Q)+\mathscr{I}_{init}(Q).
\]
The functional $\mathscr{I}(\cdot)$ defined in this manner is lower
semicontinuous and has compact level sets(\textit{cf. }Theorem 7.4
of \cite{QRV}). 
\end{defn}
Now, we can state the LDP for the empirical process in a concrete
form. The following theorem can be proven by the general method presented
in Sections 7 and 9 of \cite{QRV}, which relies on the LDP for colored
system and Dawson-G\"artner's projective limit theorem (\textit{cf.}
Theorem 4.6.1 of \cite{DZ}).
\begin{thm}
Under Assumption 3, $\{P_{N}\}_{N=1}^{\infty}$ satisfies the LDP
with the good rate function $\mathscr{I}(\cdot)$ defined in Definition
\ref{Def}, and scale $N$. In other words, for each measurable set
$A\subset\mathscr{M}_{1}(C([0,\,T],\,\mathbb{T}))$, 
\[
-\inf_{Q\in A^{o}}\mathscr{I}(Q)\le\liminf_{N\rightarrow\infty}\frac{1}{N}\log P_{N}(A)\le\limsup_{N\rightarrow\infty}\frac{1}{N}\log P_{N}(A)\le-\inf_{Q\in\bar{A}}\mathscr{I}(Q)
\]
where the topology is the usual topology of weak convergence for measures. 
\end{thm}
\textbf{Acknowledgment.} The author would like to thank the advisor
Professor S. R. Srinivasa Varadhan for introducing this problem and
sharing his unlimited knowledge and insight through numerous discussions.

\end{document}